\documentclass[10pt]{amsart}

\usepackage{varioref}
\usepackage[french,english]{babel}
\usepackage[utf8]{inputenc}	
\usepackage[T1]{fontenc}

\usepackage[normalem]{ulem}
\usepackage{latexsym}

\usepackage{amsmath}			
\usepackage{amsfonts}			
\usepackage{amssymb}			
\usepackage{amsthm}
\usepackage{amsbsy}				

\usepackage{wasysym}
\usepackage{dsfont}			
\usepackage{shuffle}		
\usepackage{mathrsfs}

\usepackage{graphicx}
\usepackage{float}	
\usepackage{comment}
\usepackage{layout}

\usepackage[all]{xy}			
\usepackage{tikz}				
\usetikzlibrary{cd,calc,shapes,shapes.geometric,decorations.pathreplacing,arrows} 

\usepackage{tabularx}

\usepackage{enumitem}	
\setitemize[0]{font=\textbf, label=--}

\usepackage{color}
\definecolor{Chocolat}{rgb}{0.36, 0.2, 0.09}
\definecolor{BleuTresFonce}{rgb}{0.215, 0.215, 0.36}

\usepackage[colorlinks,final,backref=page,hyperindex]{hyperref}
\hypersetup{citecolor=BleuTresFonce, linkcolor=Chocolat}

\usepackage{cleveref}

\theoremstyle{plain}
\newtheorem{thm}{Theorem}[section]
\newtheorem*{thm*}{Theorem}
\newtheorem{lem}[thm]{Lemma}
\newtheorem{prop}[thm]{Proposition}
\newtheorem*{prop*}{Proposition}
\newtheorem{cor}[thm]{Corollary}
				
\theoremstyle{definition}
\newtheorem{defi}[thm]{Definition}
\newtheorem*{defi*}{Definition}
\newtheorem{defiprop}[thm]{Definition/Proposition}
\newtheorem{exmp}[thm]{Example}

\newtheorem{nonexmp}[thm]{Non-example}

\newtheorem{nota}[thm]{Notation}
\newtheorem{rem}[thm]{Remark}

\newcommand{\Z}{\mathbb{Z}}
\newcommand{\N}{\mathbb{N}}

\newcommand{\K}{\mathcal{K}}

\newcommand{\calP}{\mathcal{P}}

\newcommand{\id}{\mathrm{id}}

\newcommand{\F}{\mathcal{F}}

\newcommand{\scrF}{\mathscr{F}}

\renewcommand{\phi}{\varphi}

\newcommand{\As}{\mathcal{A}s}
\newcommand{\Com}{\mathcal{C}om}
\newcommand{\UAs}{\mathcal{UA}s}
\newcommand{\UCom}{\mathcal{UC}om}
\newcommand{\Lie}{\mathcal{L}ie}
\newcommand{\Pois}{\mathcal{P}ois}
\newcommand{\DLie}{{\mathcal{DL}ie}}

\newcommand{\DPois}{\mathcal{DP}ois}

\newcommand{\Ob}{\mathrm{Ob}\,}
\newcommand{\op}{^{\mathrm{op}}}
\newcommand{\ord}{^{\mathrm{or}}}

\newcommand{\Val}{^{\mathrm{Val}}}

\newcommand{\C}{\mathsf{C}}
\newcommand{\D}{\mathsf{D}}
\newcommand{\Set}{\mathsf{Set}}

\newcommand{\Vect}{\mathsf{Vect}}

\newcommand{\Ch}{\mathsf{Ch}}

\renewcommand{\mod}{\mathsf{mod}}
\newcommand{\Smod}{\mathfrak{S}\mbox{-}\mathsf{mod}}
\newcommand{\Smodred}{\mathfrak{S}\mbox{-}\mathsf{mod}^{\mathrm{red}}}
\newcommand{\Sbimod}{\mathfrak{S}\mbox{-}\mathsf{bimod}}
\newcommand{\Sbimodred}{\mathfrak{S}\mbox{-}\mathsf{bimod}^{\mathrm{red}}}

\newcommand{\Fin}{\mathsf{Fin}}

\newcommand{\Proto}{\mathsf{protoperads}}
\newcommand{\Prope}{\mathsf{properads}}
\newcommand{\Func}{\mathsf{Func}}

\newcommand{\Hom}{\mathrm{Hom}}

\newcommand{\Wall}{\mathcal{W}^{\mathrm{conn}}}
\newcommand{\Colo}{\mathcal{C}\mathrm{ol}}
\newcommand{\Ind}{\mathrm{Ind}}
\newcommand{\Res}{\mathrm{Res}}
\newcommand{\For}{\mathrm{For}}

\newcommand{\End}{\mathrm{End}}

\newcommand{\commu}{\circlearrowleft}

\newcommand{\pushout}{\mbox{\large{$\ulcorner$}}}

\newcommand{\DB}[2]{\{\hspace{-3pt}\{#1,#2\}\hspace{-3pt}\}}
\newcommand{\TB}[3]{\{\hspace{-3pt}\{#1,#2,#3\}\hspace{-3pt}\}}

\newcommand{\Ibox}{I_{\boxtimes}}

\newcommand{\croix}[1][]
{%
	\begin{tikzpicture}[baseline=(textbox.base),inner sep=0pt]
	\node[cross out,draw,text width=\dimexpr#1] (textbox) {\strut};
	\useasboundingbox (textbox);
	\end{tikzpicture}%
}

\title{Protoperads I: combinatorics and definitions}
\address{LAGA, Universit\'e Paris 13, 99 Avenue Jean Baptiste Cl\'ement 93430, Villetaneuse, France}
\email{leray@math.univ-paris13.fr}
\author{Johan \textsc{Leray}} 
\date{\today}
\keywords{Combinatorics, Species, Properad, Protoperad}
\subjclass[2010]{05E25,18D50,18G35,55U10}
\thanks{
	This article is the combinatorial part of the  PhD thesis of the author, supported by the project "Nouvelle \'Equipe", convention n$^\circ$2013-10203/10204 between La R\'egion des Pays de Loire and the University of Angers.  The author thanks the Centre Henri Lebesgue ANR-11-LABX-0020-01 for its stimulating mathematical research programs. This paper was finished at the University Paris 13, where the author was financed by a postdoctoral allocation given by DIM Math Innov. The author is indebted to G. Powell who has carefully read and corrected the first version of this paper.
}

\begin{document}
	
\begin{abstract}
	This paper is the first of two articles which develop the notion of protoperads. In this one, we construct a new monoidal product on the category of reduced $\mathfrak{S}$-modules. We study the associated monoids, called \emph{protoperads}, which are a generalization of operads. As operads encode algebraic operations with several inputs and one outputs, protoperads encode algebraic operations with the same number of inputs and outputs.
	We describe the underlying combinatorics of protoperads, and show that there exists a notion of free protoperad. We also show that the monoidal product introduced here is related to Vallette's one on the category of $\mathfrak{S}$-bimodules, via the induction functor.
\end{abstract}

\maketitle

\section*{Introduction}
This is the first of two papers in which the author develops the notion of \emph{protoperad}, which is a kind of properad (see \cite{Val03,Val07}), and the homotopy theory of these new objects (see \cite{Ler18ii}). Properad is an algebraic notion which encodes types of bialgebras, i.e. operations with several inputs and several outputs.

The motivation for this work is to determine what is a double Poisson bracket up to homotopy. Double Poisson structure, defined by Van den Bergh in  \cite{VdB08}, give a Poisson structure in noncommutative algebraic geometry (see \cite{Gin05,VdB08-2}) under the \emph{Kontsevich-Rosenberg principle}, i.e. if $A$ is a double Poisson algebra, then the family of affine representation schemes $\mathrm{Rep}_n(A)$, represented by commutative algebras $A_n$, has a Poisson structure.

Berest, et al define derived representation schemes (see  \cite{BKR13,BCER12}) as the left derived functor $\mathbb{L}(-)_n$. A natural question is the following: what is a double Poisson structure compatible with the derived side of the Berest's construction? What is a double Poisson bracket up to homotopy? Double Poisson structure is \emph{properadic} in nature. It is encoded by the properad $\DPois$ (see \cite{Ler18ii}); it gives us a good framework to study what is its version up to homotopy (cf. \cite{Val03,MV09I,MV09II}). The structure of double Poisson algebras up to homotopy is controlled by a cofibrant resolution $\DPois_\infty$ of the properad $\DPois$ (cf. \cite{MV09I,MV09II} for the model structure of the category of properads).  Properads are algebraic objects which encode some algebraic structures, there are included in a large family of such objects:
\begin{center}
	Associative alg. $\subset$ NS-Operads $\subset$  Operads $\subset$ Dioperads $\subset$ Properads $\subset$ Props.
\end{center}
Let $V$ be a $k$-vector space: algebras encode algebraic structures with one input and one output $V\rightarrow V$; non-symmetric and  (symmetric) operads encode algebraic structures with several input and one output $V^{\otimes m} \rightarrow V$;  dioperads, properads and props encodes algebraic structures with several input and outputs $V^{\otimes m} \rightarrow V^{\otimes n}$. We can resume that by expliciting in which categories these objects live and  there underlying combinatorics.
\begin{center} 
\begin{tabular}{*{7}{m{.15\linewidth}|}}
	& Associative Algebras  & NS-Operads &Operads & Dioperads & Properads & Props \\ \hline 
	Live in the category & $\Vect_k$ & $\mathbb{N}\mbox{-}\mathsf{mod}_k^{\mathrm{red}}$& $\Smodred_k $ & $\Sbimodred_k$ & $\Sbimodred_k$ & $\Sbimod_k$ \\ 
	Monoid  for the product & $\otimes_k$ & $\circ^{\mathrm{ns}}$ & $\circ$ & $\boxtimes_{c,\varnothing}^{\mathrm{Gan}}$ & $\boxtimes_c\Val$ & \croix[20mm] \\
	Generators &
	$\begin{tikzpicture}[scale=0.2,baseline=0]
	\draw (0,2) -- (0,-2);
	\draw[fill=black] (0,0) circle (6pt);
	\draw (0,0) node[left] {\tiny{$p$}};
	\end{tikzpicture}$
	&
	$\begin{tikzpicture}[scale=0.2,baseline=0]
	\draw (-2,2) -- (0,0);
	\draw (-1,2) -- (0,0);
	\draw (2,2) -- (0,0);
	\draw (0,0) -- (0,-1.5);
	\draw (0.5,2) node {\tiny{$\cdots$}};
	\draw[fill=black] (0,0) circle (6pt);
	\draw (0,0) node[left] {\tiny{$p$}};
	\end{tikzpicture}$ 
	&
	$\begin{tikzpicture}[scale=0.2,baseline=0]
	\draw (-2,2) -- (0,0);
	\draw (-2,2) node[above]  {\tiny{$1$}};
	\draw (-1,2) -- (0,0);
	\draw (-1,2) node[above]  {\tiny{$2$}};
	\draw (2,2) -- (0,0);
	\draw (2,2) node[above]  {\tiny{$m$}};
	\draw (0,0) -- (0,-1.5);
	\draw (0.5,2) node {\tiny{$\cdots$}};
	\draw[fill=black] (0,0) circle (6pt);
	\draw (0,0) node[left] {\tiny{$p$}};
	\end{tikzpicture}$
	& 
	$\begin{tikzpicture}[scale=0.2,baseline=0]
	\draw (-2,2) -- (0,0);
	\draw (-2,2) node[above]  {\tiny{$1$}};
	\draw (-1,2) -- (0,0);
	\draw (-1,2) node[above]  {\tiny{$2$}};
	\draw (2,2) -- (0,0);
	\draw (2,2) node[above]  {\tiny{$m$}};
	\draw (0,0) -- (-1,-2) node[below] {\tiny{$1$}};
	\draw (0,0) -- (1,-2) node[below] {\tiny{$n$}};
	\draw (0,-2) node {\tiny{$\cdots$}};
	\draw (0.5,2) node {\tiny{$\cdots$}};
	\draw[fill=black] (0,0) circle (6pt);
	\draw (0,0) node[left] {\tiny{$p$}};
	\end{tikzpicture}$
	&
	$\begin{tikzpicture}[scale=0.2,baseline=0]
	\draw (-2,2) -- (0,0);
	\draw (-2,2) node[above]  {\tiny{$1$}};
	\draw (-1,2) -- (0,0);
	\draw (-1,2) node[above]  {\tiny{$2$}};
	\draw (2,2) -- (0,0);
	\draw (2,2) node[above]  {\tiny{$m$}};
	\draw (0,0) -- (-1,-2) node[below] {\tiny{$1$}};
	\draw (0,0) -- (1,-2) node[below] {\tiny{$n$}};
	\draw (0,-2) node {\tiny{$\cdots$}};
	\draw (0.5,2) node {\tiny{$\cdots$}};
	\draw[fill=black] (0,0) circle (6pt);
	\draw (0,0) node[left] {\tiny{$p$}};
	\end{tikzpicture}$
	&
	$\begin{tikzpicture}[scale=0.2,baseline=0]
	\draw (-2,2) -- (0,0);
	\draw (-2,2) node[above]  {\tiny{$1$}};
	\draw (-1,2) -- (0,0);
	\draw (-1,2) node[above]  {\tiny{$2$}};
	\draw (2,2) -- (0,0);
	\draw (2,2) node[above]  {\tiny{$m$}};
	\draw (0,0) -- (-1,-2) node[below] {\tiny{$1$}};
	\draw (0,0) -- (1,-2) node[below] {\tiny{$n$}};
	\draw (0,-2) node {\tiny{$\cdots$}};
	\draw (0.5,2) node {\tiny{$\cdots$}};
	\draw[fill=black] (0,0) circle (6pt);
	\draw (0,0) node[left] {\tiny{$p$}};
	\end{tikzpicture}$ \\
	Composition controlled by &
	$\begin{tikzpicture}[scale=0.2,baseline=(n.base)]
	\node (n) at (0,1) {};	
	\draw (0,2) -- (0,-2);
	\draw[fill=black] (0,-1) circle (6pt)  node[left] {\tiny{$p$}};
	\draw[fill=black] (0,1) circle (6pt)  node[left] {\tiny{$q$}};
	\end{tikzpicture}$
	&
	\begin{tikzpicture}[scale=0.2,baseline=(n.base)]
	\node (n) at (1,5) {};
	\draw (-2,2) -- (0,0) -- (0,-1.5);
	\draw (0,0) -- (1,4) -- (-0.5,6);
	\draw (1,4) -- (2.5,6);
	\draw (0,0) -- (-0.5,2);
	\draw (1,6) node {\tiny{$\cdots $}};
	\draw (2,2) -- (0,0);
	\draw (0.5,2) node {\tiny{$\cdots$}};
	\draw[fill=black] (0,0) circle (6pt);
	\draw (0,0) node[left] {\tiny{$p$}};
	\draw[fill=black] (1,4) circle (6pt);
	\draw (1,4) node[right] {\tiny{$q$}};
	\end{tikzpicture}
	&
	\begin{tikzpicture}[scale=0.2,baseline=(n.base)]
	\node (n) at (1,6) {};
	\draw (-2,2) -- (0,0) -- (0,-1.5);
	\draw (-2,2) node[above]  {\tiny{$1$}};
	\draw (0,0) -- (1,4) -- (-0.5,6);
	\draw (1,4) -- (2.5,6);
	\draw (0,0) -- (-0.5,2);
	\draw (-0.5,6) node[above] {\tiny{$i$}};
	\draw (2.5,6) node[above] {\tiny{$i+n-1$}};
	\draw (1,6) node {\tiny{$\cdots $}};
	\draw (-0.5,2) node[above]  {\tiny{$i-1$}};
	\draw (2,2) -- (0,0);
	\draw (2,2) node[above]  {\tiny{$m$}};
	\draw (0.5,2) node {\tiny{$\cdots$}};
	\draw[fill=black] (0,0) circle (6pt);
	\draw (0,0) node[left] {\tiny{$p$}};
	\draw[fill=black] (1,4) circle (6pt);
	\draw (1,4) node[right] {\tiny{$q$}};
	\end{tikzpicture}
	&
	connected oriented graphs \emph{without genus}
	&
	connected oriented graphs with genus
	&
	oriented graphs with genus \\
	Example of encoded structures 
	& Chain complexes & 
	Associative algebras
	&
	Lie algebras
	&
	Involutive Lie bialgebras
	&
	Lie bialgebras
	& 
	Loday infinitesimal bialgebras
	\\
	A reference 
	&
	\cite[Chapter 1]{LV12}
	&
	\cite[Chapter 5]{LV12}
	& 
	\cite[Chapter 5]{LV12}
	&
	\cite{Gan03}
	&
	\cite{Val07}
	&
	\cite{Mac65}
\end{tabular}
~~
\end{center}

For such an object $\calP$, e.g. $\calP$ an operad, a properad, etc... , the notion of $\calP$-algebra up to homotopy is given by cofibrant resolution of $\calP$ (see the Homotopy Tranfert Theorem for operads \cite{LV12}). The homotopy theory of such objects is more or less complicated. 

These successive algebraic structures are increasingly more complicate since they encode more and more type of algebras. As a consequence, their homotopy theory is also more and more elaborate. If props encode the largest category of algebraic structures, we do not yet have the same homotopical tools, like the Koszul duality theory, that hold for properads, operads, etc. The chain complex structures are encoded by the algebra of dual numbers. The non-symmetric operad framework is the minimal one to encode associative algebras, symmetric operad framework is the minimal one to encode associative commutative algebras, and so on: dioperads to encode bigebras without genus in the underlying combinatorics, as involutive Lie bialgebras.




In general, it is a hard problem to determine a usable (minimal) cofibrant resolution of a properad. However, if a (quadratic) properad $\mathcal{P}$ is Koszul, (i.e. the homology of its bar construction is concentrated in a certain weight), then we have an explicit minimal resolution of $\mathcal{P}$: Vallette defines the Koszul duality for properads (see \cite{Val03,Val07,MV09I}). Recall that Koszul duality theory does not exists for props. Unfortunately, it is also a hard problem to show that a properad is Koszul. Almost all known examples of Koszul properads come from Koszulness results for operads. For operads, there are technical tools for showing that an operad is Koszul, including rewriting methods, PBW or Gr\"obner bases and distributive laws (see \cite{Hof10,LV12,DK10}). The only other known example of a Koszul properad which does not come from an operadic result is the Frobenius properad (and its Koszul dual) (see \cite{CMW16}).

A double Poisson bracket on an associative algebra $A$ is a double Lie bracket with some compatibilities with the product of $A$. The properad $ \DLie$, which encodes double Lie structure,  is defined by generators and relations, i.e. $\DLie = \scrF(V_{\DLie})/\langle R_{\mathcal{DJ}}\rangle$, with $V_\DLie$ concentrated in arity $(2,2)$:
\[
V_\DLie=
\begin{tikzpicture}[scale=0.2,baseline=-1]
\draw (0,0.5) node[above] {$\scriptscriptstyle{1}$};
\draw (2,0.5) node[above] {$\scriptscriptstyle{2}$};
\draw[fill=black] (-0.3,-0.5) rectangle (2.3,0);
\draw[thin] (0,-1) -- (0,0.5);
\draw[thin] (2,-1) -- (2,0.5);
\draw (0,-1) node[below] {$\scriptscriptstyle{1}$};
\draw (2,-1) node[below] {$\scriptscriptstyle{2}$};
\end{tikzpicture}~~\otimes sgn(\mathfrak{S}_2)\big\uparrow_{\mathfrak{S}_2^{op}}^{\mathfrak{S}_2\times \mathfrak{S}_2^{op}}
\]
and the relation
\[
R_{\mathcal{DJ}}=
\begin{tikzpicture}[scale=0.2,baseline=-1]
\draw (0,3) node[below] {$\scriptscriptstyle{1}$};
\draw (2,3) node[below] {$\scriptscriptstyle{2}$};
\draw (4,3) node[below] {$\scriptscriptstyle{3}$};
\draw[fill=black] (1.7,0.5) rectangle (4.3,1);
\draw[fill=black] (-0.3,-0.5) rectangle (2.3,0);
\draw[thin] (0,-1) -- (0,1.5);
\draw[thin] (2,-1) -- (2,1.5);
\draw[thin] (4,-1) -- (4,1.5);
\draw (0,-1) node[below] {$\scriptscriptstyle{1}$};
\draw (2,-1) node[below] {$\scriptscriptstyle{2}$};
\draw (4,-1) node[below] {$\scriptscriptstyle{3}$};
\end{tikzpicture}
~~+~~
\begin{tikzpicture}[scale=0.2,baseline=-1]
\draw (0,3) node[below] {$\scriptscriptstyle{2}$};
\draw (2,3) node[below] {$\scriptscriptstyle{3}$};
\draw (4,3) node[below] {$\scriptscriptstyle{1}$};
\draw[fill=black] (1.7,0.5) rectangle (4.3,1);
\draw[fill=black] (-0.3,-0.5) rectangle (2.3,0);
\draw[thin] (0,-1) -- (0,1.5);
\draw[thin] (2,-1) -- (2,1.5);
\draw[thin] (4,-1) -- (4,1.5);
\draw (0,-1) node[below] {$\scriptscriptstyle{2}$};
\draw (2,-1) node[below] {$\scriptscriptstyle{3}$};
\draw (4,-1) node[below] {$\scriptscriptstyle{1}$};
\end{tikzpicture}
~~+~~
\begin{tikzpicture}[scale=0.2,baseline=-1]
\draw (0,3) node[below] {$\scriptscriptstyle{3}$};
\draw (2,3) node[below] {$\scriptscriptstyle{1}$};
\draw (4,3) node[below] {$\scriptscriptstyle{2}$};
\draw[fill=black] (1.7,0.5) rectangle (4.3,1);
\draw[fill=black] (-0.3,-0.5) rectangle (2.3,0);
\draw[thin] (0,-1) -- (0,1.5);
\draw[thin] (2,-1) -- (2,1.5);
\draw[thin] (4,-1) -- (4,1.5);
\draw (0,-1) node[below] {$\scriptscriptstyle{3}$};
\draw (2,-1) node[below] {$\scriptscriptstyle{1}$};
\draw (4,-1) node[below] {$\scriptscriptstyle{2}$};
\end{tikzpicture}
\ .
\]
A double Lie bracket on a chain complex $A$ is given by a morphism of properads $\DLie \rightarrow \End_A$ where $\End_A$ is the properad of endomorphisms of $A$ (see \cite{Val07} for the definition).
The operad $\Pois\cong\Com\circ\Lie$, which encodes Poisson structures, is Koszul because the operads $\Com$ and $\Lie$, which respectively encode commutative and Lie algebras, are Koszul and satisfy a distributive law (see \cite[Sect. 8.6.3]{LV12}). We have an analogous statement for double Poisson structures: as the properad $\As$, which encodes associative algebras, is Koszul (see \cite{LV12}), the properad $\DPois\cong \As \boxtimes_c\Val \DLie$ is Koszul if, and only if, the properad $\DLie$ is Koszul (see \cite[Sect. 4]{Ler18ii}). As the study of the Koszulness of $\DLie$ is still difficult, the idea is to use the diagonal symmetry of the generator and the relation of this properad to simplify the problem. In fact, the $\mathfrak{S}$-bimodules (which are families of representations of permutation groups) $V_\DLie$ and $R_{\mathcal{DJ}}$ are respectively  induced from representations of the groups $\mathfrak{S}_2$ and $\mathfrak{S}_3$, by the diagonal morphism  
\begin{equation}\label{eq::induction_group_morphism} 
\begin{array}{ccc}
G& \longrightarrow & G\times G^{op} \\
g& \longmapsto &  g\times g^{-1}
\end{array} \ ,
\end{equation}
for $G$ in $\{\mathfrak{S}_2,\mathfrak{S}_3\}$ respectively. This fundamental observation allows us to reduce the problem to one in the category of $\mathfrak{S}$-modules, and to define the minimal framework, called \emph{protoperad}, which encode double Lie algebras. We will see (in \cite{Ler18ii}) that the homotopy theory of these new objects is more simpler than the homotopy of properads. We resume the most important results of this article in the following theorem.

\begin{thm*}[see Def. \ref{def::prod_connexe_Smod}, Thm. \ref{thm::Ind_monoidal_connexe}, Prop. \ref{prop::proto_libre}]
	The category $\Smodred_k$ of reduced $\mathfrak{S}$-modules, i.e. the full sub-category of functors $P:\Fin\op\rightarrow\Ch_k$ such that $P(\varnothing)=0$
	is monoidal for the connected composition product $\boxtimes_c$. The monoids in this category are called \textbf{protoperads}. There exists the free protoperad functor, denoted by $\scrF(-)$ and the functor
	\[
	\begin{array}{cccc}
		\Ind : & \big(\Smodred_k,\boxtimes_c\big) & \longrightarrow  & \big(\Sbimodred_k,\boxtimes_c^{\mathrm{Val}}\big) 
	\end{array} 
	\] 
	which is exact and satisfies $\Ind\circ \scrF=\scrF\Val\circ\Ind$, where $\scrF\Val$ is the functor of free properad.
\end{thm*}

Using the framework of protoperads is successful: we prove in \cite{Ler18ii} that there is a Koszul duality theory for protoperads which is compatible  with that for properads via $\Ind$, and we use it to prove that the properads $\DLie$ and hence $\DPois$ are Koszul.

\paragraph{\textsc{Section \ref{sect::bricks_and_walls} -- \nameref{sect::bricks_and_walls}}}
We develop the combinatorics for protoperads.
The combinatorics is controlled by \emph{walls}. A wall over a non-empty finite set $S$ is a set of subsets of $S$, equipped with a particular partial order and such that the union of these subsets is $S$. We represent a wall diagrammatically  as follows:  
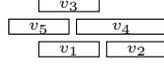
\[
\begin{tikzpicture}[scale=0.4,baseline=1.4ex]
\draw (0,0) rectangle (2,0.5);
\draw (1,0.2) node {\tiny{$v_1$}};
\draw (2.25,0) rectangle (4.25,0.5);
\draw (3.25,0.2) node {\tiny{$v_2$}};
\draw (0,1.5) rectangle (2,2);
\draw (1,1.7) node {\tiny{$v_3$}};
\draw (1.25,0.75) rectangle (4.25,1.25);
\draw (2.75,0.95) node {\tiny{$v_4$}};
\draw (-1,0.75) rectangle (1,1.25);
\draw (0,0.95) node {\tiny{$v_5$}};
\end{tikzpicture}
\]
the width of each brick giving us the number of entries and exits of each of the operations $v_i$ which is represented. This is encoded by the functor $\Wall$ and certain  subfunctors. Let $S$ be a non-empty finite set and $n$ a natural number, a \emph{wall of $n$  bricks} over $S$ is the datum of a set $W=\{W_\alpha\}_{\alpha \in A}$ of $|A|=n$ non-empty subsets of $S$, which are called \emph{bricks of $W$}, satisfying:
\begin{itemize}
	\item the union of all bricks is $S$;
	\item for all element $s$ of $S$, the set of bricks $W_\alpha$ which contains the element $s$ is totally ordered
	\item \emph{a compatibility between orders}.
\end{itemize}
We define also the notion of \emph{connectedness} for a wall and denote by $\Wall(S)$, the set of connected walls over $S$.

\paragraph{\textsc{Section \ref{sect::product_on_Smod} -- \nameref{sect::product_on_Smod}}}
We review two monoidal products on $\Smodred_k:=\Func(\Fin\op,\Ch_k)^{\mathrm{red}}$  which is the full sub-category of functors $P: \Fin\op \rightarrow \Ch_k$ from the category $\Fin$ of finite sets with bijections, to the category $\Ch_k$ of $k$-chain complexes such that $P(\varnothing)=0$.
These product are \emph{the composition product} $\square$ , also called the Hadamard product, and \emph{ the concatenation product} $\otimes^{\mathrm{conc}}$ (see \Cref{subsect::prod_Smod}).

We also define the \emph{connected composition product} on $\Smodred_k$  (see \Cref{def::prod_connexe_Smod}), denoted by $\boxtimes_c$, which encodes algebraic structures which have the same number of inputs and outputs and a diagonal symmetry.
It is the bifunctor 
\[
- \boxtimes_c - : \Smodred_k \times \Smodred_k \longrightarrow \Smodred_k
\]
defined, for all reduced $\mathfrak{S}$-modules $P$ and $Q$ and for all non-empty finite sets $S$, by:	  
\[
~P \boxtimes_c Q (S):= \underset{(I,J)\in \mathcal{X}^{\mathrm{conn}}(S)}{\bigoplus}~\underset{\alpha}{\bigotimes}~P(I_\alpha) \otimes \underset{\beta}{\bigotimes}~Q(J_\beta).
\]

\emph{Protoperads} are the monoids for this monoidal structure. For a protoperad $\mathcal{P}$, we can view a homogeneous element $p$ of  $\mathcal{P}(S)$ as a labelled brick
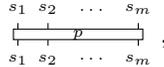
\[
\begin{tikzpicture}[scale=0.2,baseline=1]
	\draw (0,0) -- (0,1.5);
	\draw (0,0) node[below] {\tiny{$s_1$}};
	\draw (0,1.5) node[above] {\tiny{$s_1$}};
	\draw (2,0) -- (2,1.5);
	\draw (2,0) node[below] {\tiny{$s_2$}};
	\draw (2,1.5) node[above] {\tiny{$s_2$}};
	\draw (5,0) node[below] {\tiny{$\cdots$}};
	\draw (5,1.5) node[above] {\tiny{$\cdots$}};
	\draw	(8,0) -- (8,1.5);
	\draw (8,0) node[below] {\tiny{$s_m$}};
	\draw (8,1.5) node[above] {\tiny{$s_m$}};
	\draw[fill=white] (-0.3,0.4) rectangle (8.3,1.1);
	\draw (4,0.7) node {\tiny{$p$}};
\end{tikzpicture},~
\]
and the product $\mathcal{P}\boxtimes_c \mathcal{P}(S) \rightarrow \mathcal{P} (S)$ as the composition of two rows of bricks which are connected:
\[
\begin{tikzpicture}[scale=0.4,baseline=2ex]
	\draw (0,2) node[above] {\tiny{$s_1$}} to[out=270,in=90] (0,0) node[below]{\tiny{$s_1$}};
	\draw (2,2) node[above] {\tiny{$s_2$}} to[out=270,in=90] (2,0)node[below]{\tiny{$s_2$}};
	\draw (4,2) to[out=270,in=90] (4,0);
	\draw (6,2) to[out=270,in=90] (6,0);
	\draw (8,2) to[out=270,in=90] (8,0);
	\draw (9,2) to[out=270,in=90] (9,0);
	\draw (10,2) node[above] {\tiny{$s_m$}} to[out=270,in=90] (10,0) node[below]{\tiny{$s_m$}};
	\draw[fill=white] (-0.3,1.8) rectangle (2.3,2.2);
	\draw[fill=white] (3.7,1.8) rectangle (6.3,2.2); 
	\draw[fill=white] (7.7,1.8) rectangle (10.3,2.2); 
	\draw (3.1,2) node {\tiny{$\cdots$}};
	\draw (7,2) node {\tiny{$\cdots$}};
	\draw (1,2) node {\tiny{$p_1$}};
	\draw (5,2) node {\tiny{$p_j$}};
	\draw (9,2) node {\tiny{$p_s$}};
	\draw[fill=white] (-0.3,-0.2) rectangle (4.3,0.2);
	\draw[fill=white] (5.7,-0.2) rectangle (10.3,0.2);
	\draw (5,0) node {\tiny{$\cdots$}};
	\draw (2,0) node {\tiny{$p'_1$}};
	\draw (8,0) node {\tiny{$p'_r$}};
\end{tikzpicture}
~~\mapsto
\begin{tikzpicture}[scale=0.2,baseline=1]
	\draw (0,0) -- (0,1.5);
	\draw (0,0) node[below] {\tiny{$s_1$}};
	\draw (0,1.5) node[above] {\tiny{$s_1$}};
	\draw (2,0) -- (2,1.5);
	\draw (2,0) node[below] {\tiny{$s_2$}};
	\draw (2,1.5) node[above] {\tiny{$s_2$}};
	\draw (5,0) node[below] {\tiny{$\cdots$}};
	\draw (5,1.5) node[above] {\tiny{$\cdots$}};
	\draw	(8,0) -- (8,1.5);
	\draw (8,0) node[below] {\tiny{$s_m$}};
	\draw (8,1.5) node[above] {\tiny{$s_m$}};
	\draw[fill=white] (-0.3,0.4) rectangle (8.3,1.1);
	\draw (4,0.7) node {\tiny{$p$}};
\end{tikzpicture}
\]

\paragraph{\textsc{Section \ref{sect::bimodule} -- \nameref{sect::bimodule}}}
In this section, we also recall three monoidal structures on the category of reduced $\mathfrak{S}$-bimodules which are analogous to ones on the $\mathfrak{S}$-module category: the concatenation product, the composition product, and the connected composition product $\boxtimes_c^{\mathrm{Val}}$, defined by Vallette in \cite{Val03,Val07}. We do this review with an other point of view than the original one: we give an equivalent definition of the connected composition product, which is more adapted to \emph{species} and the functorial point of view of $\mathfrak{S}$-bimodules.

\paragraph{\textsc{Section \ref{sect::induction_functor} -- \nameref{sect::induction_functor}}}

The new product $\boxtimes_c$ is the avatar of the product $\boxtimes_c^{\mathrm{Val}}$ on the  category $\Sbimodred_k$. The most important property of the product $\boxtimes_c $ is its compatibility with the product of Vallette via the induction functor $ \Ind: \Smodred_k \rightarrow \Sbimodred_k$, defined using \Cref{eq::induction_group_morphism}. We prove the following.
\begin{thm*}[see \Cref{thm::Ind_monoidal_connexe}]
	The induction functor
	\[
	\Ind : \big(\Smodred_k,\boxtimes_c\big)  \longrightarrow   \big(\Sbimodred_k,\boxtimes_c^{\mathrm{Val}}\big)
	\] 
	is monoidal. In particular, it sends protoperads to properads,
\end{thm*}

\paragraph{\textsc{Section \ref{sect::protoperads} -- \nameref{sect::protoperads}}}
In this section, we give an equivalent definition of a protoperad, generalizing the definition of operads in terms of partial compositions.

\begin{prop*}[see \Cref{prop::proto_def_partielle}]
	A protoperad $\calP$ has canonically a partial composition system. Conversely, a partial composition system on a $\mathfrak{S}$-module $P$  canonically extends to a protoperad structure.
\end{prop*}

Using the work of Vallette on free monoids in abelian monoidal categories (see \cite{Val09}), we show that there exists a free protoperad functor. We also have a combinatorial description of the free protoperad.
\begin{thm*}[{see  \Cref{prop::proto_libre}}]
	Let $V$ be a reduced $\mathfrak{S}$-module. The free protoperad functor is the graded functor $\scrF^*(-)$ given, for all finite set $S$, and for all natural number $\rho$, by the isomorphism of right $\mathrm{Aut}(S)$-modules
	\[
		\scrF^\rho(V)(S) \cong \bigoplus_{{\substack{(\{W_\alpha\}_{\alpha\in A},\leqslant)\\ \in \mathcal{W}^{\mathrm{conn}}_\rho(S)}}} \bigotimes_{\alpha\in A} V(W_\alpha),
	\]
	with  $\Wall_\rho(S)$, the set of connected walls with $\rho$ bricks.
\end{thm*}

\paragraph{\textsc{Section \ref{sect::coloring} -- \nameref{sect::coloring}}}
In \cite{Ler18ii}, the author develops the Koszul duality for protoperads, which is related to Koszul duality theory of properads (see \cite{Val03,Val07}) via the functor $\Ind$. In this last section, we define the notion of a coloured wall, and we associate to a wall $W$ over a totally ordered set $S$, the \emph{colouring complex}, denoted by $\mathrm{C}_\bullet^{\Colo}(W)$. This is motivated by the combinatorial description of the bar construction of the free protoperad (see also \cite{Ler17,Ler18ii}). The principal result of this section is the following:

\begin{thm*}[see \Cref{prop::cpx_colo_acyclique}]
	Let $S$ be a finite totally ordered set, and $W$ a wall over $S$. If the set $\mathrm{Succ}(W)$ (see \Cref{subsect::recalls on posets} for the definition of $\mathrm{Succ}$) is non empty, then the colouring complex $\mathrm{C}_\bullet^{\Colo}(W)$ is acyclic.
\end{thm*}

This theorem corresponds to the acyclicity result for bar construction of free protoperad (see \cite{Ler18ii} for the definition of the bar construction of a protoperad).

\setcounter{tocdepth}{1} 
\tableofcontents

\section*{Notation}
We use the notation $\N^*$ for the set $\N-\{0\}$. We denote by $\Fin$, the category with finite sets as objects and bijections as morphisms and $\Set$, the category of all sets and all applications.   For two integers $a$ and $b$, we denote by $[\![a,b]\!]$ the set $[a,b]\cap \Z$, and, for $n\in \N^*$, $\mathfrak{S}_n$ is the automorphism group of $[\![1,n]\!]$, i.e. $\mathfrak{S}_n=\mathrm{Aut}_{\Fin}\left( [\![1,n]\!] \right)$. We denote by $\Ch_k$ the category of $\Z$-graded chain complexes over the field $k$.

Let $(\C,\odot)$ be a monoidal category: we denote by $\As(\C,\odot)$ the category of monoids without unit (e.g. semi-groups) in $\C$ and $\UAs(\C,\odot)$ the category of unital monoids in $\C$. If $(\C,\odot)$ is symmetric monoidal, we also denote by  $\Com(\C,\odot)$ the category of commutative monoids without unit and $\UCom(\C,\odot)$ the category of commutative unital monoids in $\C$.

A monoidal category  $(\C,\odot, I)$ is an \emph{abelian monoidal category} if $\C$ is also abelian: we do not suppose any compatibility between the monoidal product $\odot$ and the abelian structure.

\section{Bricks and walls}\label{sect::bricks_and_walls}

We begin by describing the combinatorial framework of this paper. The first section is about posets and, after that, we introduce the \emph{functor of walls}. Walls encode the combinatorics of "diagonal properads", as rooted trees govern the combinatorics of operads. In this section, we define two important functors:
\[
	\mathcal{X}^{\mathrm{conn}}:\Fin\op \rightarrow \Set\op 
	\quad \mathrm{and} \quad
	\Wall:\Fin\op \rightarrow \Set\op \ .
\]
The first one, $\mathcal{X}^{\mathrm{conn}}$, encodes the combinatorics of the new monoidal structure on the category of $\mathfrak{S}$-modules, \emph{the connected composition product} (see \Cref{sect::produit_connexe_Smod}). The second, $\Wall$ encodes the combinatorics of the free monoid for this monoidal structure (see \Cref{prop::proto_libre}).
\begin{rem}
	In this section, we construct (covariant) functors from the opposite category of finite sets to the opposite category of sets, i.e. $F:\Fin\op\rightarrow\Set\op$ or the category of chain complexes, i.e. $F:\Fin\op\rightarrow\Ch_k$. We choose to consider the opposite category of $\Fin$ to get a right action of the automorphism group $\mathrm{Aut}(S)$ on $F(S)$. This right action mimics the actions of symmetric groups on the leaves of trees in the operadic case.
\end{rem}

\subsection{Recollections on posets}\label{subsect::recalls on posets}

Let $k$ and $l$ be two elements of a poset $(K, \leqslant)$. We say that $k$ and $l$ are \emph{successors} if $k<l$ and if there does not exist an element $t$ in $K$ such that $k< t< l$. We denote by $\mathrm{Succ}(K)$, the set of pairs of successors of $K$. A \emph{chain} of a poset $K$ is an increasing sequence of elements of $K$ and  the \emph{length} of the chain is the number of elements of the chain: we denote the length of a chain $k_1<k_2<\ldots<k_r$ by $\mathrm{len}(k_1<k_2<\ldots<k_r)$. The \emph{height} of an element $k$ of a poset $(K,\leqslant)$ is the element $\mathfrak{h}(k)$ of $\N\cup\{\infty\}$ \label{def::poset::hauteur} defined by
\[
	\mathfrak{h}(k):= \mathrm{max}\big\{\mathrm{len}(c)\in\N^*~|~c= (\lambda_1<\lambda_2<\ldots<\lambda_{r-1}<k) \big\}.
\]

\begin{prop}\label{prop::projection_ordre}
	Let $(K,\leqslant)$ be a poset and $(k,l)$ in $\mathrm{Succ}(K)$. Then the surjection  
	\[
		\pi_k^l : K \twoheadrightarrow K/ _{k\sim l}
	\] 
	induces a partial order on $K/ _{k\sim l}$ defined, for all $r$ and $s$ in $K$, by
	\begin{itemize}
		\item $[r] \leqslant [s]$ if $r\leqslant s$ and $r,s\notin\{k,l\}$
		\item $[s] \leqslant [k\sim l]$  (resp. $s\geqslant [k\sim l]$) if $s\leqslant k$ or $s\leqslant l$ (resp. $s\geqslant k$ or $s\geqslant l$).
	\end{itemize}	
\end{prop} 
\begin{proof}
	Left to the reader.
\end{proof}
\begin{lem}\label{lem::ordre_partiel_canonique}
	Let $(R,\leqslant_R)$ and $(S,\leqslant_S)$ be two posets with injections $R\hookrightarrow T \hookleftarrow S$. If, for all $a$ and $b$ in $R\cap S$, $a\leqslant_R b$ if and only if $a\leqslant_S b$, then  $R\cup S$ has a canonical partial order which extends the partial orders $\leqslant_R$ et $\leqslant_S$.
\end{lem}
\begin{proof}
	For any $x$ and $y$ in $R\cup S$, we have $x\leqslant_{R\cup S} y$ if and only if one of the following assumption holds:
	\begin{itemize}
		\item $x$ and $y$ are in $R$ and $x\leqslant_R y$;
		\item $x$ and $y$ are in $S$ and $x\leqslant_S y$;
		\item $x$ is in $R$, $y$ is in  $S$ and there exists $t$ in $R\cap S$ such that $x\leqslant_R t \leqslant_S y$.
	\end{itemize}
\end{proof}
	
\subsection{The functors of walls}\label{sect::functors of walls}
In the rest of this section, we define some (covariant) functors from the category $\Fin\op$ to the category $\Set\op$, called \emph{functors of walls}. Let $\mathcal{F}$ be a functor of walls (see below for definitions) and $S$ a finite set with $n$ elements. An element $W$ of $\mathcal{F}(S)$ should represent  a morphism $\mathrm{Hom}_{\C}(V^{\otimes n}, V^{\otimes n})$, for $V$ a chain complex, with a diagonal action of $\mathfrak{S}_n$ by permutations inputs and outputs at the same time.
 
\begin{defi}[Functor of ordered walls $\mathcal{W}\ord$]
	For $n$ in $\N^*$, the covariant functor $\mathcal{W}_n\ord : \Fin\op \longrightarrow \Fin\op$ is given, for all finite set $S$, by 
	\[
	\mathcal{W}_n\ord(S):=
	\left\{ \big( W=(W_1,\ldots,W_n)\big)\Bigg| 
	\begin{array}{l}
		\cup_i W_i=S;~\forall i \in [\![1,n]\!], W_i\ne \varnothing;~ \\
		\forall s,t\in S, \Gamma^W_s:=\{W_i|s\in W_i\} \\
		\quad \mbox{ is totally ordered (by}\leqslant_s); \\
		\forall a,b\in\Gamma^W_s\cap \Gamma^W_t, \ a\leqslant_s b \Leftrightarrow a\leqslant_t b
	\end{array}
	~~\right\} 
	\]
	where an element $W$ of $\mathcal{W}\ord_n(S)$ is a poset for $\leqslant$ which is the induced partial order (cf. lemma \ref{lem::ordre_partiel_canonique}) on  $\cup_{s\in S} \Gamma^W_s= \{W_1,\ldots,W_n\}$. We denote by $(W, \leqslant)$, the elements of $\mathcal{W}\ord_n(S)$. The action of $\sigma$, an element of  $\mathrm{Aut}(S)$, on $\big((W_1,\ldots,W_n),\leqslant\big)$ in $\mathcal{W}_n\ord$ is induced by the canonical action on $S$,
	\[
		\big((W_1,\ldots,W_n),\leqslant\big)\cdot\sigma =\big((W_1\cdot \sigma,\ldots,W_n\cdot \sigma),\leqslant^\sigma \big)
	\]
	where $\leqslant^\sigma$ is induced by the total orders of the sets $\Gamma^{W\cdot \sigma}_{s}:=\{W_i\cdot \sigma|s\in W_i\cdot \sigma\}$. The functor  $\mathcal{W}\ord$, defined by
	\[
		\begin{array}{rccc}
		\mathcal{W}\ord : &\Fin\op & \longrightarrow &\mathtt{Sets}\op \\
		& S & \longmapsto & \coprod_{n\in\N^*} \mathcal{W}\ord_n(S)
		\end{array}~~.
	\]
\end{defi}

\begin{rem} 
	For all non-empty sets $S$, we have $\mathcal{W}\ord_0(S)=\varnothing$. For all integers $n>0$, the group $\mathfrak{S}_n$ acts freely on $\mathcal{W}\ord_n(S)$ by permuting the position of elements, i.e. for $\tau$ in $\mathfrak{S}_n$, we have
	$
	\tau\cdot \big((W_1,\ldots,W_n),\leqslant\big)=\big((W_{\tau^{-1}(1)},\ldots,W_{\tau^{-1}(n)}),\leqslant\big).
	$
	The partial order is the same because it doesn't depend of the $W_i$'s indexes. 
\end{rem}

\begin{exmp}
	 We consider  $W_a=\{1,2\}$, $W_b=\{3,4\}$ and $W_c=\{2,3\}$, three subset of $S=[\![1,4]\!]$.  Then, we have the following four elements of $\mathcal{W}_3\ord(S)$:
	\begin{itemize}
		\item $\big( (W_1,W_2,W_3),<^1\big)$  with  $<^1$ given by $W_1 <^1 W_3$ and $W_2 <^1 W_3$;
		\item $\big( (W_1,W_2,W_3),<^2\big)$ with $<^2$ given by $W_3 <^2 W_1$ and $W_3 <^2 W_2$;
		\item $\big( (W_1,W_2,W_3), <^3\big)$ with $<^3$ given by $W_1 <^3 W_3$ and $W_3 <^3 W_2$;
		\item $\big( (W_1,W_2,W_3), <^4\big)$ with $<^4$ given by $W_3 <^4 W_1$ and $W_2 <^4 W_3$.
	\end{itemize}
	This four elements of $\mathcal{W}_3\ord(S)$ are distinct. 
\end{exmp}

The \emph{vertical composition product} on $\mathcal{W}\ord$, is the natural transformation:
\[
	\mathcal{V} : \big(\mathcal{W}\ord\times \mathcal{W}\ord\big)(-) \longrightarrow \mathcal{W}\ord(-)
\]
given, for all finite set $S$, by $\mathcal{V}_{n,m,S}: \mathcal{W}_n\ord(S)\times \mathcal{W}_m\ord(S) \longrightarrow \mathcal{W}_{m+n}\ord(S)$ which sends the pair $\big((W,\leqslant_W),(L,\leqslant_L)\big)$ on $\big(R=(W_1,\ldots,W_n,L_1,\ldots L_m),\leqslant_W^L\big)$ where, for all $s$ in $S$, the total order of the poset $\Gamma_s^{R}$ is induced by the ones of $\Gamma_s^{W}$ and $\Gamma_s^{L}$ and by extension, for all $W_i$ in $\Gamma_s^{W}$ and all $L_j$ in $\Gamma_s^{L}$, we have $W_i\leqslant_W^L L_j$. This product is associative, so, for all finite set $S$, we have the following commutative diagram:
\[
\xymatrix@R=0.5cm{
	\mathcal{W}\ord(S)\times\mathcal{W}\ord(S)\times\mathcal{W}\ord(S) \ar[r]^(0.55){\mathcal{V}\times\mathrm{id}} \ar[d]_{\mathrm{id}\times\mathcal{V}} \ar@{}[rd]|\commu &
	\mathcal{W}\ord(S)\times\mathcal{W}\ord(S) \ar[d]^{\mathcal{V}}\\
	\mathcal{W}\ord(S)\times\mathcal{W}\ord(S) \ar[r]_{\mathcal{V}}&\mathcal{W}\ord(S).
}
\]

The \emph{(horizontal) concatenation product}  on $\mathcal{W}\ord$, is the natural transformation between bifunctors:
\[
\mathcal{H} : \mathcal{W}\ord(-_1)\times \mathcal{W}\ord(-_2)\longrightarrow \mathcal{W}\ord(-_1\amalg -_2)
\]
given, for all finite sets $S$ and $T$, by $\mathcal{H}_{n,m,S,T}: \mathcal{W}_n\ord(S)\times \mathcal{W}_m\ord(T) \longrightarrow \mathcal{W}_{m+n}\ord(S\amalg T)$ which sends $\big((W,\leqslant_W),(L,\leqslant_L)\big)$  to $\big(R=(W_1,\ldots,W_n,L_1,\ldots L_m),\leqslant_{W,L}\big)$ where, for all $s$ in $S$ and $t$ in $T$, we have the equalities  $\Gamma_s^{R}=\Gamma_s^{W}$ and $\Gamma_t^{R}=\Gamma_t^{L}$. This product is associative and commutative, so we have the following commutative diagrams: 
\[
\xymatrix@R=0.5cm{
	\mathcal{W}\ord(-_1)\times\mathcal{W}\ord(-_2)\times\mathcal{W}\ord(-_3) \ar[r]^(0.53){\mathcal{H}\times\mathrm{id}} \ar[d]_{\mathrm{id}\times\mathcal{H}} \ar@{}[rd]|\commu &
	\mathcal{W}\ord(-_1)\times\mathcal{W}\ord(-_2\amalg-_3) \ar[d]^{\mathcal{H}}\\
	\mathcal{W}\ord(-_1\amalg-_2)\times\mathcal{W}\ord(-_3) \ar[r]_{\mathcal{H}}&\mathcal{W}\ord(-_1\amalg -_2\amalg-_3)
}
\]
and
\[
\xymatrix@R=0.5cm{
	\mathcal{W}\ord(-_1) \times \mathcal{W}\ord(-_2) \ar[r]^{\mathcal{H}} \ar[d]_{\cong} \ar@{}[rd]|\commu  & \mathcal{W}\ord(-_1\amalg-_2) \ar[d]^{\cong} \\
	\mathcal{W}\ord(-_2) \times \mathcal{W}\ord(-_1) \ar[r]_{\mathcal{H}}  & \mathcal{W}\ord(-_2\amalg-_1)~. 
}
\]
We also have the following commutative diagram of natural transformations, called \emph{the interchanging law}:
\[
\begin{tikzcd}[column sep=small, row sep=small]
	\big(\mathcal{W}\ord\big)^{\times 2}(-_1) \times \big(\mathcal{W}\ord\big)^{\times 2}(-_2) 
	\ar[rr, "\mathrm{id}\times \sigma \times \mathrm{id}"] \ar[d,"\mathcal{V}\times\mathcal{V}"'] 
	& &
	\big(\mathcal{W}\ord(-_1)\times \mathcal{W}\ord(-_2)\big)^{\times 2} 
	 \ar[d,"\mathcal{H}\times \mathcal{H}"] \\
	\mathcal{W}\ord(-_1)\times \mathcal{W}\ord(-_2) 
	\ar[rd,"\mathcal{H}"]
	\ar[rr,phantom, "\commu"]
	& & \big(\mathcal{W}\ord(-_1\amalg -_2)\big)^{\times 2}
	\ar[ld,"\mathcal{V}"] \\
	& \mathcal{W}\ord(-_1\amalg -_2) &
\end{tikzcd} \ .
\]

Pass to $\mathcal{W}$, the functor of unordered walls:

\begin{defi}[Functor of walls $\mathcal{W}$]
	We define the functor $\mathcal{W}_n: \Fin\op \rightarrow \Fin\op$ by $\mathcal{W}_n:= \big( \mathcal{W}\ord_n \big)_{\mathfrak{S}_n}$
	which is given, for all finite sets $S$, by $\mathcal{W}_n(S):=$
	\[
		\left\{ \big( W=\{W_\alpha\}_{\alpha\in A},\leqslant\big)~~\Bigg|~~ 
		\begin{array}{l}
		|A|=n; \ \forall \alpha \in A, \ W_\alpha\ne\varnothing; \ \cup_\alpha W_\alpha=S;\\
		\forall s\in S,~\Gamma^W_s:=\{W_\alpha|s\in W_\alpha\} \\
		\mbox{is totally ordered (by }\leqslant_s) \\
		\forall s,t\in S, ~ \forall a,b\in\Gamma_s\cap \Gamma_t,~~a\leqslant_s b \Leftrightarrow a\leqslant_t b
		\end{array}
		~~\right\} 
	\]
	where $\leqslant$ is the canonical partial order of $\cup_{s\in S} \Gamma^W_s= \{W_\alpha\}_{\alpha\in A}$. We have the natural projection $\pi:\mathcal{W}\ord\longrightarrow \mathcal{W}$. 	The action of $\sigma$, an element of $\mathrm{Aut}(S)$, on $\big(\{W_\alpha\}_{\alpha\in A},\leqslant\big)\in\mathcal{W}_n$ is induced by the canonical action on $S$, i.e.
	\[
	\big(\{W_\alpha\}_{\alpha\in A},\leqslant\big)\cdot\sigma =\big(W=\{W_\alpha\cdot \sigma\}_{\alpha\in A},\leqslant^\sigma\big)
	\]
	where $\leqslant^\sigma$ is induced by total orders of $\Gamma^{W\cdot \sigma}_{s\cdot \sigma}:=\{W_\alpha\cdot \sigma|s\cdot \sigma\in W_\alpha\cdot \sigma\}$. We also define the functor 
	\[
	\begin{array}{rccc}
	\mathcal{W} : &\Fin\op & \longrightarrow &\mathtt{Sets}\op \\
	& S & \longmapsto & \coprod_{n\in\N^*} \mathcal{W}_n(S)
	\end{array}~~.
	\]
	An element $W$ of $\mathcal{W}(S)$ is called a \emph{wall} over $S$, and an element of a wall $W$ is called a \emph{brick} of $W$.
\end{defi}

\begin{prop}[Products on  $\mathcal{W}$]~~\label{prop::produits_sur_W}
	The products $\mathcal{V}$ and $\mathcal{H}$ on $\mathcal{W}\ord$ pass through tue quotient by the actions of the symmetric groups on the indexes of bricks, hence induce natural transformations 
	\[
	\mathcal{V} : \big(\mathcal{W}\times \mathcal{W}\big)(-) \longrightarrow \mathcal{W}(-) 
	~\mbox{ and }~ 
	\mathcal{H} : \mathcal{W}(-_1)\times \mathcal{W}(-_2)\longrightarrow \mathcal{W}(-_1\amalg -_2),
	\]
	respectively called the \textit{composition product} and \textit{concatenation product} on $\mathcal{W}$, such that we have the following commutative diagrams 
	\[
	\begin{tikzcd}
		\big(\mathcal{W}\ord\times \mathcal{W}\ord\big)(-) \ar[r,"\mathcal{V}"] 
		\ar[d, twoheadrightarrow, "\pi^{\times 2}"'] \ar[rd, phantom, "\commu" description] &\mathcal{W}\ord(-) \ar[d, two heads, "\pi"] \\
		\big(\mathcal{W}\times \mathcal{W}\big)(-) \ar[r, "\mathcal{V}"'] &\mathcal{W}(-)
	\end{tikzcd} \ ;
	\begin{tikzcd}
		\mathcal{W}\ord(-_1)\times \mathcal{W}\ord(-_2) \ar[rd, phantom, "\commu" description] \ar[r, "\mathcal{H}"] \ar[d, two heads, "\pi^{\times 2}"'] &\mathcal{W}\ord(-_1\amalg -_2) \ar[d, two heads, "\pi"] \\
		\mathcal{W}(-_1)\times \mathcal{W}(-_2) \ar[r, "\mathcal{H}"] & \mathcal{W}(-_1\amalg -_2)
	\end{tikzcd} \ ;
	\]
	\[
	\begin{tikzcd}[column sep= -0.4cm]
		\big(\mathcal{W}\times \mathcal{W}\big)(-_1)\times \big(\mathcal{W}\times \mathcal{W}\big)(-_2) 
		\ar[rr,"\id\times \sigma \times \id"] 
		\ar[d, "\mathcal{V}\times\mathcal{V}"']
		& &\mathcal{W}(-_1)\times \mathcal{W}(-_2)\times \mathcal{W}(-_1)\times \mathcal{W}(-_2) \ar[d,"\mathcal{H}\times \mathcal{H}"] \\
		\mathcal{W}(-_1)\times \mathcal{W}(-_2) \ar[rd, "\mathcal{H}"']
		\ar[rr, phantom, "\commu" description]
		& & \mathcal{W}(-_1\amalg -_2)\times \mathcal{W}(-_1\amalg -_2) \ar[ld,"\mathcal{V}"] \\
		& \mathcal{W}(-_1\amalg -_2) &
	\end{tikzcd} .
	\]
\end{prop}

\subsection{Two subfunctors of \texorpdfstring{$\mathcal{W}$}{W}}\label{rem::sous_foncteur_combinatoire}
We introduce here two important subfunctors of $\mathcal{W}$. For all finite non-empty sets $S$ and all $n$ in $\N^*$, we define the functor of ordered partitions $\mathcal{Y}\ord_n:\Fin\op \rightarrow \Fin\op$, by 
\[	
	\mathcal{Y}\ord_n(S):=
	\left\{ \ (K_1,\ldots,K_n) \ \left| \ 
	\amalg_{i=1}^n K_i=S; \ \forall i \in [\![1,n]\!], K_i\ne \varnothing \ \right.\right\}
\]
equipped with the natural injection $\mathcal{Y}\ord_n \hookrightarrow \mathcal{W}_n\ord$. By disjoint union, we also define the functor  $\mathcal{Y}\ord$ by 
\[
\begin{array}{rccc}
	\mathcal{Y}\ord : &\Fin\op & \longrightarrow &\Set\op \\
	& S & \longmapsto & \coprod_{n\in\N^*} \mathcal{Y}\ord_n(S)
\end{array}~~.
\]
Via the vertical composition, we have, for all finite sets $S$ and all $m$ and $n$ in $\N^*$, the isomorphism:

\begin{align*}
	&\mathcal{Y}\ord_m(S)\times \mathcal{Y}\ord_n(S)~\cong \\
	&\left\{ \big(R=(K_1,\ldots,K_m,L_1,\ldots,L_n)\big)  \left| 
	\begin{array}{l}
	\amalg_i K_i=S=\amalg_j L_j; \\
	 \forall i \in [\![1,m]\!], K_i\ne \varnothing; \ \forall j \in [\![1,n]\!], L_j\ne \varnothing; \\
	\forall s\in S, \ \exists ! i\in[\![1,m]\!], \ \exists ! j\in[\![1,n]\!]~\\
	\qquad\mbox{s.t. }\Gamma^R_s:=\{K_i,L_j\} \mbox{ with } K_i\leqslant_s L_j \\
	\forall s,t\in S, \ \forall a,b\in\Gamma_s^R\cap \Gamma_t^R,\\
	\qquad a\leqslant_s b \Leftrightarrow a\leqslant_t b
	\end{array}
	\right. \right\} ,
\end{align*} 
which gives us the natural injection $ 	\mathcal{Y}\ord_m(S)\times \mathcal{Y}\ord_n(S) \hookrightarrow \mathcal{W}\ord_{m+n}(S)$. Hence, we define, for all non-empty finite sets $S$, the functor $\mathcal{X}\ord$ of ordered pairs of partitions of finite sets, by
\[
	\mathcal{X}\ord(S):= \coprod_{m,n\in\N^*} \mathcal{Y}\ord_n(S)\times \mathcal{Y}\ord_m(S).
\]
equipped with the natural injection $\mathcal{X}\ord \hookrightarrow \mathcal{W}\ord$. This functor is important: it encodes the combinatorics of our new monoidal product, up to a property of connectedness (see \Cref{subsect::connected wall})

The natural surjection $\pi:\mathcal{W}\ord\twoheadrightarrow \mathcal{W}$ gives the following  commutative diagrams of natural transformations:
\[
\begin{tikzcd}
	\mathcal{Y}\ord\ar[r, hook] \ar[d, two heads]\ar[rd, phantom , "\commu"] &\mathcal{W}\ord\ar[d, two heads] \ar[rd, phantom ,"\commu" description ]&\mathcal{X}\ord\ar[l, hook'] \ar[d,two heads ] \\		\mathcal{Y}\ar[r, hook]  &\mathcal{W} &\mathcal{X} \ar[l, hook'] 
\end{tikzcd} \ ,
\] 
where $\mathcal{Y}$ (resp. $\mathcal{X}$) is the quotient of $\mathcal{Y}\ord$ (resp. $\mathcal{X}\ord$) by the action of the symmetric group on the indexes of bricks. The concatenation product restricts to the subfunctors $\mathcal{X}$ and $\mathcal{Y}$:
\[
\begin{tikzcd}[column sep= 2ex]
	\mathcal{Y}(-_1)\times\mathcal{Y}(-_2) \arrow[r, hookrightarrow] \ar[d,"\mu^{\mathrm{conc}}"'] \arrow[rd,phantom,"\commu" description] & \mathcal{W} (-_1)\times\mathcal{W}(-_2) \arrow[d,"\mathcal{H}"] \\
	\mathcal{Y}(-_1\amalg-_2) \arrow[r, hookrightarrow] & \mathcal{W} (-_1\amalg-_2) \
\end{tikzcd}
 \mathrm{and} 
\begin{tikzcd}[column sep= 2ex]
	\mathcal{X}(-_1)\times\mathcal{X}(-_2) \arrow[r, hookrightarrow] \arrow[d, "\mu^{\mathrm{conc}}"'] \arrow[rd,phantom,"\commu" description] & \mathcal{W} (-_1)\times\mathcal{W}(-_2) \ar[d, "\mathcal{H}"] \\
	\mathcal{X}(-_1\amalg-_2) \arrow[r, hookrightarrow] & \mathcal{W} (-_1\amalg-_2)
\end{tikzcd} \ .
\]	

\subsection{Connected walls}\label{subsect::connected wall}
Now, we introduce the notion of \emph{connectedness} of a wall. Let  $(W=\{W_\alpha\}_{\alpha\in A},\leqslant)$ be a wall in $\mathcal{W}(S)$. We define on $W$  \emph{the equivalence relation of connectedness} $\overset{conn.}{\sim}$: for two elements $a$ and $b$ of $A$, we say $W_a\overset{conn.}{\sim} W_b$ if there exist an integer $n\geqslant 2$  and a sequence $W_0,W_1,\ldots,W_{n-1},W_n$ of elements of $W$ with $W_0=W_a$ and $W_n=W_b$ such that, for all $i$ in $[\![0,n-1]\!]$, 
	\[
	W_i\cap W_{i+1}\ne \varnothing \ \mbox{and} \ (W_i,W_{i+1})\in \mathrm{Succ}(W) \ \mbox{or} \ (W_{i+1},W_{i})\in \mathrm{Succ}(W).
	\]
\begin{defi}[Projection $\mathcal{K}$] \label{def::projection_K}
	We define the projection $\mathcal{K}$ as follows: for a finite set $S$, we have
	\[
	\begin{array}{cccc}
	\mathcal{K}_S : & \mathcal{W}(S) & \longrightarrow & \mathcal{Y}(S) \subset \mathcal{W}(-) \\
	& W & \longmapsto & \left\{ \left. \bigcup_{B_\alpha\in\pi^{-1}([B])} B_\alpha \  \right| \ [B]\in \pi(W) \right\}
	\end{array} \ ,
	\]
	where $\pi$ is the projection of $W$ to its quotient by $\overset{conn.}{\sim}$.
\end{defi} 
We have the natural commutative diagram
\[
\xymatrix{
	\mathcal{X} \ar@{}[rd]|\commu\ar@{^{(}->}[r] \ar[d]_{\mathcal{K}} & \mathcal{W} \ar[d]^{\mathcal{K}} \\
	\mathcal{Y} \ar@{^{(}->}[r] & \mathcal{W}.
}
\]
\begin{lem}\label{lem::K_associative}
	The projection $\mathcal{K}$ is associative, i.e. the following diagram of natural transformation commutes:
	\[
	\begin{tikzcd}
		\mathcal{Y}^{\times 3} \arrow[r, "\mathcal{Y}\times \mathcal{K}"] \arrow[d, "\mathcal{K}\times \mathcal{Y}"'] \arrow[rd, phantom, "\commu" description] &  \mathcal{Y}^{\times 2} \arrow[d, "\mathcal{K}"] \\
		\mathcal{Y}^{\times 2} \arrow[r, "\mathcal{K}"'] & \mathcal{Y}
	\end{tikzcd} \ .
	\]
\end{lem}
\begin{defi}[The functor $\Wall$] \label{def::mur}  
	We define the functor 
	\[
		\mathcal{W}_n^{conn,or} : \Fin\op \longrightarrow \Fin\op
	\]
	by, for all non-empty set $S$, the fiber of $\mathcal{K}_S : \mathcal{W}\ord(S) \rightarrow \mathcal{Y}\ord(S)$ over the wall with one brick $\{S\}$, i.e. the subfunctor of $\mathcal{W}\ord_n$ giving by $\mathcal{W}^{conn,or}_n(S):=$
	\[
	\left\{ (W,<_W)\in\mathcal{W}_n\ord~~\Bigg|~~ 
	\begin{array}{l}
	\forall \alpha,\beta\in[\![1,n]\!], \exists W_\alpha=:W_{i_0},\ldots,W_{i_{m-1}},W_m:=W_\beta\\
	\mbox{s.t.}~\forall j \in [\![0,m-1]\!],~~W_{i_j}\in W,~ W_{i_j}\cap W_{i_{j+1}}\ne \varnothing ~~\\
	\mbox{and}~(W_{i_j}, W_{i_{j+1}})\mbox{ or }( W_{i_{j+1}},W_{i_j})\in \mathrm{Succ}(W)
	\end{array}
	~~\right\} .
	\]
	The natural surjection $\mathcal{W}\ord \twoheadrightarrow \mathcal{W}$ gives us the subfunctor 
	\[
	\begin{tikzcd}
	\mathcal{W}^{conn,or} \ar[d, two heads] \ar[r] & \mathcal{W}\ord \ar[d,two heads] \\
	\Wall \ar[r] &\mathcal{W}
 	\end{tikzcd}
	\]
	called \textit{the functor of connected walls}: an element of $\Wall(S)$ is called a \emph{connected} wall on $S$.
\end{defi}

\begin{rem}
	By the same arguments as in remark \ref{rem::sous_foncteur_combinatoire}, we have the natural injection of $\mathcal{X}^{\mathrm{conn}}$  in $\Wall$.
\end{rem}

\begin{prop}\label{prop::decompo_mur_connexe}
	Let $W$ be a wall in $\mathcal{W}(S)$. Then, there exist $n$ in $\N$ and $S_1\amalg \ldots \amalg S_n$ a unique non-ordered partition of $S$ such that
	\[
		W\in \mathrm{im}\Big(\mathcal{H} : \prod_{i=1}^n \Wall(S_i) \longrightarrow \mathcal{W}(S) \Big).
	\]
\end{prop}
\begin{proof}
	Let $S$ be a finite set and $W$ be in $\mathcal{W}(S)$, a wall over $S$. The partition $\mathcal{K}(W)$ in $\mathcal{Y}(S)$ gives the result. 
\end{proof}

\begin{rem}\emph{About the diagramatic representation of walls.} 
	\label{rem::representation_mur}
	The terminology introduced in  definition \ref{def::mur} comes from the diagramatic representation of the elements of $\mathcal{W}(S)$. Just as the combinatorics of operads is controlled by rooted trees (cf. \cite[Sect. 5.6]{LV12}), the combinatorics of protoperads is controlled by a stack of bricks: 
	
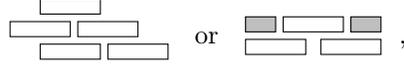
\[
	\begin{tikzpicture}[scale=0.4,baseline=1.4ex]
	\draw (0,0) rectangle (2,0.5);
	\draw (2.25,0) rectangle (4.25,0.5);
	\draw (0,1.5) rectangle (2,2);
	\draw (1.25,0.75) rectangle (3.25,1.25);
	\draw (-1,0.75) rectangle (1,1.25);
	\end{tikzpicture}
	~~\mbox{ or }~~
	\begin{tikzpicture}[scale=0.4,baseline=1ex]
	\draw (0,0) rectangle (2,0.5);
	\draw (2.5,0) rectangle (4.5,0.5);
	\draw (1.25,0.75) rectangle (3.25,1.25);
	\draw[fill= lightgray] (0,0.75) rectangle (1,1.25);
	\draw[fill= lightgray] (3.50,0.75) rectangle (4.5,1.25);
	\end{tikzpicture}~~ ,
	\]
	the gray colour indicates a brick that is not connected in this representation.
	We have some examples.
	\begin{itemize}	
		\item First, we consider the set $S=[\![1,4]\!]$ and the wall $W=\{W_a,W_b,W_c\}$ in $\mathcal{W}([\![1,4]\!])$ over $S$ with the three bricks
		\[
		W_a=\{1,2\}, W_b=\{3,4\} \mbox{ et } W_c=\{1,2\} 
		\]
		with the partial order $W_a<W_c$. We represent this wall by 
		\begin{equation}\label{fig::diagramme_mur}
				\begin{tikzpicture}[scale=0.4,baseline=1ex]
				\draw[dotted] (0.4,-0.2) -- (0.4,1.55);
				\draw (0.4,1.5) node[above] {\tiny{$1$}};
				\draw[dotted] (1.6,-0.2) -- (1.6,1.55);
				\draw (1.6,1.5) node[above] {\tiny{$2$}};
				\draw[dotted] (2.9,-0.2) -- (2.9,1.55);
				\draw (2.9,1.5) node[above] {\tiny{$3$}};
				\draw[dotted] (4.1,-0.2) -- (4.1,1.55);
				\draw (4.1,1.5) node[above] {\tiny{$4$}};
				\draw[fill=white] (0,0) rectangle (2,0.5);
				\draw (1,0.20) node {\scriptsize{a}};
				\draw[fill=white] (2.5,0) rectangle (4.5,0.5);
				\draw (3.5,0.25) node {\scriptsize{b}};
				\draw[fill=white] (0,0.75) rectangle (2,1.25);
				\draw (1,1) node {\scriptsize{c}};
				\end{tikzpicture}~.
		\end{equation}
		This wall is not connected: $W$ is in the image of the product 
		\[
		\mathcal{H} :  \Wall([\![1,2]\!])\times \Wall([\![3,4]\!]) \rightarrow \mathcal{W}([\![1,4]\!]).
		\]
		Note that this graphical representation of the wall $W$ \emph{depends on the choice of an order on $S$}: so, the diagrams 
		
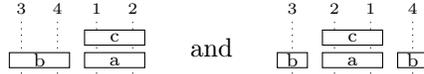
\[
		\begin{tikzpicture}[scale=0.4,baseline=1ex]
			\draw[dotted] (0.4,-0.2) -- (0.4,1.55);
			\draw (0.4,1.5) node[above] {\tiny{$1$}};
			\draw[dotted] (1.6,-0.2) -- (1.6,1.55);
			\draw (1.6,1.5) node[above] {\tiny{$2$}};
			\draw[dotted] (-2.1,-0.2) -- (-2.1,1.55);
			\draw (-2.1,1.5) node[above] {\tiny{$3$}};
			\draw[dotted] (-0.9,-0.2) -- (-0.9,1.55);
			\draw (-0.9,1.5) node[above] {\tiny{$4$}};
			\draw[fill=white] (0,0) rectangle (2,0.5);
			\draw (1,0.20) node {\scriptsize{a}};
			\draw[fill=white] (-2.5,0) rectangle (-0.5,0.5);
			\draw (-1.5,0.25) node {\scriptsize{b}};
			\draw[fill=white] (0,0.75) rectangle (2,1.25);
			\draw (1,1) node {\scriptsize{c}};
		\end{tikzpicture}
		~~~~\mbox{ and }~~~~
		\begin{tikzpicture}[scale=0.4,baseline=1ex]
			\draw[dotted] (0.4,-0.2) -- (0.4,1.55);
			\draw (0.4,1.5) node[above] {\tiny{$2$}};
			\draw[dotted] (1.6,-0.2) -- (1.6,1.55);
			\draw (1.6,1.5) node[above] {\tiny{$1$}};
			\draw[dotted] (-1,-0.2) -- (-1,1.55);
			\draw (-1,1.5) node[above] {\tiny{$3$}};
			\draw[dotted] (3,-0.2) -- (3,1.55);
			\draw (3,1.5) node[above] {\tiny{$4$}};
			\draw[fill=white] (0,0) rectangle (2,0.5);
			\draw (1,0.20) node {\scriptsize{a}};
			\draw[fill=white] (2.5,0) rectangle (3.5,0.5);
			\draw[fill=white] (-1.5,0) rectangle (-0.5,0.5);
			\draw (3,0.25) node {\scriptsize{b}};
			\draw (-1,0.25) node {\scriptsize{b}};
			\draw[fill=white] (0,0.75) rectangle (2,1.25);
			\draw (1,1) node {\scriptsize{c}};
		\end{tikzpicture}
		\]
		represent the same wall $W$. In the second diagram, the brick $W_b$ is not connected. The choice of the natural order on $[\![1,4]\!]$ corresponds to   \Cref{fig::diagramme_mur}. In cases where there is no ambiguity, we can omit the elements of $S$ and the names of bricks to obtain the following diagram
		\[
			\begin{tikzpicture}[scale=0.4,baseline=1ex]
			\draw[fill=white] (0,0) rectangle (2,0.5);
			\draw[fill=white] (2.5,0) rectangle (4.5,0.5);
			\draw[fill=white] (0,0.75) rectangle (2,1.25);
			\end{tikzpicture}~~ .
		\]
		\item We consider the connected wall with four bricks $W=\{W_a,W_b,W_c,W_d\}$ in  $\mathcal{W}^{\mathrm{conn}}([\![1,4]\!])$ with
		\[
		W_a=\{1,2\}, W_b=\{3,4\},W_c=\{2,3\} \mbox{ and } W_d=\{1,4\} 
		\]
		and the partial order $W_a<W_c$, $W_a<W_d$, $W_b<W_c$ and $W_b<W_d$. We represent this wall by
		\[
		\begin{tikzpicture}[scale=0.4,baseline=1ex]
			\draw (0,0) rectangle (2,0.5);
			\draw (1,0.20) node {\scriptsize{a}};
			\draw (2.5,0) rectangle (4.5,0.5);
			\draw (3.5,0.20) node {\scriptsize{b}};
			\draw (1.25,0.75) rectangle (3.25,1.25);
			\draw (0,0.75) rectangle (1,1.25);
			\draw (3.50,0.75) rectangle (4.5,1.25);
			\draw (2.25,0.95) node {\scriptsize{c}};
			\draw (0.5,1) node {\scriptsize{d}};
			\draw (4.15,1) node {\scriptsize{d}};
		\end{tikzpicture} ~~.
		\]
	\end{itemize}
\end{rem}

\section{Products on \texorpdfstring{$\mathfrak{S}$}{S}-modules}\label{sect::product_on_Smod}

\subsection{\texorpdfstring{$\mathfrak{S}$}{S}-modules}

Recall that the category $\Fin$ is a groupoid which gives us the equivalence of categories $\Fin \cong \Fin\op$ by passage to the inverse. One of the key points of the  constructions of this section is that $(\Fin,\amalg)$  is a symmetric monoidal category. We denote by  $\mathfrak{S}$ its skeleton i.e. the category where objects are natural numbers, i.e. $\Ob\mathfrak{S}=\N$ and where morphisms are given by $\Hom_\mathfrak{S}(n,n)=\mathfrak{S}_n$ for $n\ne 0$ and $\Hom(0,0)=\{\mathrm{id}\}$, and which is equivalent to $\Fin$.

\begin{defi}[$\mathfrak{S}$-module, $\mathfrak{S}$-bimodule]\label{def::Smod_loc_fini}
	A (right) \emph{$\mathfrak{S}$-module} is an object of  $\Smod_k\overset{\mathrm{not.}}{:=}\Func (\Fin\op,\Ch_k)$, the category of contravariant functors from $\Fin$  to the category of $k$-chain complexes.
	A \emph{$\mathfrak{S}$-bimodule} is an object of  $\Sbimod_k\overset{\mathrm{not.}}{:=}\Func (\Fin\times\Fin\op,\Ch_k)$. 
	
	Let $V$ be a $\mathfrak{S}$-module (respectively $W$ a $\mathfrak{S}$-bimodule). We say that $V$ (resp. $W$) is \emph{locally finite} if, for all finite sets $S$, $V(S)$ has finite total dimension (resp. for all finite sets $S,E$, the chain complex $W(S,E)$ has finite dimension).
\end{defi}
As  $\mathfrak{S}$ is the skeletal category of $\Fin$, we can view an  $\mathfrak{S}$-module $M$ as a collection $\big(M(n)\big)_{n\in\N^*}$ of chain complexes indexed by natural numbers, where the the group $\mathfrak{S}_n$ acts (on the left) on  $M(n)$, for $n\ne 0$. Similarly, an $\mathfrak{S}$-bimodule $P$ is a collection of chain complexes $\big(P(m,n)\big)_{m,n\in\N}$ indexed by pairs of integers where $P(m,n)$ has an  action of  $\mathfrak{S}_m$ on the left and an action of $\mathfrak{S}_n$ on the right, or equivalently, has an action of the group $\mathfrak{S}_m\times\mathfrak{S}_n\op$ on the left.

\begin{defi}[Reduced $\mathfrak{S}$-(bi)module]
	A $\mathfrak{S}$-module (resp. $\mathfrak{S}$-bimodule) $P$ which satisfies $P(\varnothing)=0$ (resp. $P(\varnothing,S)=0$ and $P(S,\varnothing)=0$ for all finite set $S$) is called \emph{reduced}.  We respectively note by  $\Smodred_k$ and $\Sbimodred_k$, the full subcategories of $\Smod_k$ and $\Sbimod_k$ of reduced $\mathfrak{S}$-modules and $\mathfrak{S}$-bimodules. 
\end{defi}
\begin{rem}
	We have the equivalence of categories $\Smod_k \cong\mathfrak{S}\op\mbox{-}\mod_k$, induced by taking the inverse of elements in symmetric groups. We use this equivalence without mention. 
\end{rem}

\subsection{Composition and concatenation products on \texorpdfstring{$\Smod_k$}{S-mod}}\label{subsect::prod_Smod}
In this subsection, we recall the classical constructions of composition and concatenation product of $\mathfrak{S}$-modules.
The \emph{composition product} (or \emph{vertical product}) is the bifunctor
\[
- \square - : \Smodred_k \times \Smodred_k \longrightarrow \Smodred_k
\]
defined, for $P$ and $Q$ two reduced $\mathfrak{S}$-modules and $S$ a non-empty finite set, by
\[
\big(P\:\square\: Q\big) (S) :=  P(S) \otimes Q(S).
\]
This bi-additive bifunctor gives the category $\Smodred$ a symmetric monoidal structure, with the identity $I_\square$, defined, for all non-empty sets $S$, by $I_\square(S):=k$ concentrated in degree $0$. In the litterature or algebraic operads (cf. \cite[Sect. 5.1.12]{LV12}), the composition product of $\mathfrak{S}$-modules is also called the \emph{Hadamard product}.
The \emph{concatenation product} is the bifunctor 
\[
	-\otimes^{\mathrm{conc}}- : \Smodred_k \times \Smodred_k \longrightarrow \Smodred_k
\]
defined, for all finite set $S$ and all reduced $\mathfrak{S}$-modules $P$ and $Q$, by:
\[
	\big(P\otimes^{\mathrm{conc}} Q\big) (S) := \underset{\{S',S''\}\in \mathcal{Y}\ord_2(S)}{\bigoplus} P(S') \otimes Q(S'').
\]
This product has no identity.

\begin{rem}
	This product is called the \emph{concatenation product} because it corresponds to a concatenation of operations. It is functorial: it is a particular case of Day's convolution product. For $P$ and $Q$ two reduced $\mathfrak{S}$-modules, we have
	\[
	P\otimes^{\mathrm{conc}} Q (-) := \int^{(S',S'')\in \Ob(\Fin\op)^{\times 2}} k\big[\Hom_{\Fin}(S'\amalg S'',-)\big] \otimes P(S')\otimes Q(S'').
	\]
\end{rem}
\begin{prop}\label{prop::symetrie_concatenation}
	The concatenation product is symmetric, i.e. for all reduced $\mathfrak{S}$-modules $P$ and $Q$, we have the following isomorphism of $\mathfrak{S}$-modules
	\[
		P\otimes^{\mathrm{conc}} Q \cong Q\otimes^{\mathrm{conc}} P.
	\]
\end{prop}	
\begin{proof}
	If $(S',S'')$ is an element of $\mathcal{Y}_2\ord(S)$, then $(S'',S')$ too.  As the monoidal product $\otimes_k$ of the category $\Ch_k$ is symmetric, we have the natural equivalence $\tau$ defined as follows: for all finite sets $S$,
	\[
	\tau_S~:\big(P\otimes^{\mathrm{conc}} Q\big) (S) \overset{\cong}{\longrightarrow} \big(Q\otimes^{\mathrm{conc}} P\big) (S)
	\]
	given by $\tau_S(p\otimes q)=(-1)^{|p||q|} q\otimes p$ for $p\otimes q$ in  $P(S')\otimes^{\mathrm{conc}}Q(S'')$.
\end{proof}
We can extend the concatenation product:
\begin{equation}\label{eq::extension_concatenation} 
-\otimes^{\mathrm{conc}}-:\Smod_k \times \Smodred_k \longrightarrow \Smodred_k.
\end{equation}
This extension is induced by the equivalence of categories 
\[
\Smod_k\cong \Ch_k\times\Smodred_k,
\]
by the injection $(-)^{\mathfrak{S}}:\Ch_k  \hookrightarrow  \Smod_k $ defined, for all chain complexes $C$ and all finite sets $S$, by
\[
	(C)^{\mathfrak{S}}(S):=\left\{
	\begin{array}{cc}
	C & \mbox{ if } S=\varnothing, \\
	0 & \mbox{ sinon ;}
	\end{array}\right.
\]
and by the action of the category $\Ch_k $ on $\Smodred_k$ defined, for all chain complexes $C$ and all finite sets $S$, by 
\[
	\big(C\otimes^{\mathrm{conc}} V\big)(S):= C \otimes V(S).
\]
This extension allows us to define the suspension of a $\mathfrak{S}$-module.
\begin{defi}[Suspension and desuspension of a $\mathfrak{S}$-module]\label{def::suspension_Smod}
	Let $\Sigma$ (respectively $\Sigma^{-1}$) be the chain complex $k$ concentrated in degree $1$ (resp. in degree $-1$). For $V$ a reduced $\mathfrak{S}$-module, the \emph{suspension of $V$} (resp. \emph{desuspension of $V$}) is the reduced $\mathfrak{S}$-module  $\Sigma V \overset{not.}{:=} \Sigma\otimes^{\mathrm{conc}} V$ (resp. $\Sigma^{-1} V \overset{not.}{:=} \Sigma^{-1}\otimes^{\mathrm{conc}} V$).
\end{defi}

\subsubsection{Free monoids associated to \texorpdfstring{$\otimes^{\mathrm{conc}}$}{otimes}}\label{sect::free_monoid}
Recall that the bifunctor  $-\otimes^{\mathrm{conc}}-$is linear in each of its inputs.
We have the functor:
\[
	\mathbb{T}_\otimes (-): \Smodred_k \longrightarrow \As(\Smodred_k, \otimes^{\mathrm{conc}})
\]
defined, for all finite sets $S$ and all reduced $\mathfrak{S}$-modules $P$, by
\[
	\big(\mathbb{T}_\otimes P\big)(S):=  \underset{r\in\N^*}{\bigoplus}\big(\mathbb{T}^r_\otimes P\big)(S)
\]
with
\[
	\big(\mathbb{T}^r_\otimes P\big)(S):=P^{\otimes^{\mathrm{conc}} r}(S)=\underset{I\in \mathcal{Y}\ord_r(S)}{\bigoplus} P(I_1)\otimes \ldots \otimes P(I_r).
\]
As $\otimes^{\mathrm{conc}}$ is symmetric, we have also the \emph{commutative free monoid functor}
\[
	\mathbb{S}(-) : \Smodred_k \longrightarrow \Com(\Smodred_k,\otimes^{\mathrm{conc}}).
\]
defined, for all reduced $\mathfrak{S}$-modules $P$, by:
\[
	\mathbb{S}(P):= \bigoplus_{b\in\N^*} \mathbb{S}^b(P) \ \mbox{with}  \ \mathbb{S}^b(P):=\big(\overline{\mathbb{T}}^b_\otimes(P)\big)_{\mathfrak{S}_b}
\]	
where the action of $\mathfrak{S}_b$ is given by the symmetry $\tau$ of the  product $\otimes^{\mathrm{conc}}$: explicitly, for 	$p_1 \otimes \cdots \otimes p_b$ an element of $P(I_1)\otimes^{\mathrm{conc}} \cdots\otimes^{\mathrm{conc}} P(I_b)$  which is included in $\left(\mathbb{T}^b_\otimes P\right)(S)$, and $\sigma$, a permutation in $\mathfrak{S}_b$, we have 
\[
	\sigma\cdot(p_1 \otimes \cdots \otimes p_b):= (-1)^{|\sigma(m)|} p_{\sigma^{-1}(1)} \otimes \ldots\otimes p_{\sigma^{-1}(b)} 
\]
which lives in $P(I_{\sigma^{-1}(1)})\otimes^{\mathrm{conc}} \!\ldots\!\otimes^{\mathrm{conc}} P(I_{\sigma^{-1}(b)})$ with  $(-1)^{|\sigma(m)|}$ the Koszul sign induced by $\tau$.~~\\
\begin{nota}
	Let $S$ be a finite set and $P$ be a reduced $\mathfrak{S}$-module, as in \cite[Sect. 5.1.14]{LV12}, we use the following notation
	\[
		\underset{I\in \mathcal{Y}_r(S)}{\bigoplus} ~\underset{\alpha\in A}{\bigotimes} P(I_\alpha) \overset{not.}{:=} \Big( \underset{I\in \mathcal{Y}\ord_r(S)}{\bigoplus} P(I_1)\otimes \ldots \otimes P(I_r) \Big)_{\mathfrak{S}_r}.
	\]
\end{nota}
Let $S$ be a finite set and $P$ and $Q$ be two reduced $\mathfrak{S}$-modules. Since
\[
	\big(\mathbb{S}^r P\big)(S) := ~\Big( \big(\mathbb{T}^rP\big)(S) \Big)_{\mathfrak{S}_r} 
	\cong  ~ \Big( \underset{I\in \mathcal{Y}\ord_r(S)}{\bigoplus} P(I_1)\otimes \ldots \otimes P(I_r) \Big)_{\mathfrak{S}_r} 
\]
then,  we have
\[
	\big(\mathbb{S}P\big)(S) \cong  \underset{ \{I_\alpha\}_{\alpha\in A}\in \mathcal{Y}(S)}{\bigoplus} ~\underset{\alpha\in A}{\bigotimes} P(I_\alpha).
\]
We also have  the following isomorphism:
\[
	\big(\mathbb{S}P \:\square\: \mathbb{S}Q\big)(S) \cong  \underset{(\{I_\alpha\}_{\alpha\in A},\{J_\beta \}_{\beta\in B})\in \mathcal{X}(S)}{\bigoplus} ~\underset{\alpha\in A}{\bigotimes} P(I_\alpha)\otimes~ \underset{\beta\in B}{\bigotimes} Q(J_\beta).
\]
Moreover, as the bifunctor $-\otimes^{\mathrm{conc}}-$ is biadditive, the functor $\mathbb{S}$ has the exponential property:
\begin{equation}\label{prop::prop_exp_S}
	\mathbb{S}(P\oplus Q) \cong \mathbb{S}(P) \oplus \mathbb{S}(Q) \oplus \mathbb{S}(P)\otimes^{\mathrm{conc}} \mathbb{S}(Q).
\end{equation}

\subsection{Connected composition product of  \texorpdfstring{$\mathfrak{S}$}{S}-modules}\label{sect::produit_connexe_Smod}

In this section, we define the new monoidal structure on the category of reduced $\mathfrak{S}$-modules, which is called the connected composition product. This monoidal structure is the analogue of the product defined by Vallette in \cite{Val03,Val07}.

\begin{defi}[Connected composition product of $\mathfrak{S}$-modules] \label{def::prod_connexe_Smod}
	The \emph{connected composition product} of reduced $\mathfrak{S}$-modules is the bifunctor 
	\[
		- \boxtimes_c - : \Smodred_k \times \Smodred_k \longrightarrow \Smodred_k
	\]
	defined, for all reduced $\mathfrak{S}$-modules $P$ and $Q$ and for all non-empty finite set $S$, by:	  
	\[
		~P \boxtimes_c Q (S):= \underset{(I,J)\in \mathcal{X}^{\mathrm{conn}}(S)}{\bigoplus}~\underset{\alpha}{\bigotimes}~P(I_\alpha) \otimes \underset{\beta}{\bigotimes}~Q(J_\beta).
	\]
	Denote by $\Ibox$, the $\mathfrak{S}$-module given by
	\[
	\Ibox(S) \overset{\mathrm{def}}{:=} 
	\left\{\begin{array}{cl}
	k & \mbox{ if } |S|=1 \\
	0 & \mbox{otherwise}
	\end{array}\right. ,
	\]
	which is the unit of the product $\boxtimes_c$.
\end{defi}

Below, we require a description of elements of $P\boxtimes_c Q(S)$: we represent by
\[
	(p_1\otimes\ldots\otimes p_r)\otimes(q_1\otimes\ldots\otimes q_s)
\]
a class of $\bigotimes_\alpha~P(I_\alpha) \otimes \bigotimes_\beta~Q(J_\beta)$ with $(I,J)$ in $\mathcal{X}^{\mathrm{conn}}(S)$. Recall that we identify $(p_1\otimes\ldots\otimes p_r)\otimes(q_1\otimes\ldots\otimes q_s)$ with 
\[
	(-1)^{|\sigma(p)|+|\tau(q)|}(p_{\sigma^{-1}(1)}\otimes\ldots\otimes p_{\sigma^{-1}(r)})\otimes(q_{\tau^{-1}(1)}\otimes\ldots\otimes q_{\tau^{-1}(s)})
\]
for all permutations $\sigma$ in $\mathfrak{S}_r$, $\tau$ in $\mathfrak{S}_s$ and  $(-1)^{|\sigma(p)|}, (-1)^{|\tau(q)|}$, the Koszul signs induced by these permutations.

\begin{lem}\label{lem::Smod_ana_scinde}
	The product $\boxtimes_c$ is associative and, for all reduced $\mathfrak{S}$-modules $A$ and $B$, the endofunctor
	\[	
	\begin{array}{rccc}
	\Phi_{A,B} : & \Smodred_k& \longrightarrow &  \Smodred_k\\
	& X & \longmapsto & A \boxtimes_c X \boxtimes_c B
	\end{array}
	\]
	is split analytic (in the sense of \cite{Val09}).
\end{lem}
\begin{proof}
	The associativity of the product $\boxtimes_c$ follow from the associativity of $\mathcal{K}_S$: for $P$,$Q$ and $R$ three reduced $\mathfrak{S}$-modules, and $S$ a non-empty set, we have the following isomorphism of chain complexes:
	\[ 	
	\begin{aligned}
	\big((P \boxtimes_c Q)&\boxtimes_c R\big)(S) 
	=~ \underset{(I,J)\in \mathcal{X}^{\mathrm{conn}}(S)}{\bigoplus}~ \underset{\alpha}{\bigotimes}~\big(P\boxtimes_c Q\big)(I_\alpha) \otimes \underset{\beta}{\bigotimes}R(J_\beta) \\
	=~&\underset{(I,J)\in \mathcal{X}^{\mathrm{conn}}(S)}{\bigoplus} \underset{\alpha}{\bigotimes}
	\underset{{\substack{
		(K^\alpha,L^\alpha)\\
		\in \mathcal{X}^{\mathrm{conn}}(I_\alpha)
	}}}{\bigoplus}~  \underset{\gamma}{\bigotimes} P(K_\gamma^\alpha) \otimes \underset{\delta}{\bigotimes} Q(L_\delta^\alpha)    \otimes \underset{\beta}{\bigotimes}R(J_\beta)\\ 
	\cong~&\underset{(I,J)\in \mathcal{X}^{\mathrm{conn}}(S)}{\bigoplus} 
	\underset{{\substack{ 
		\{(K^\alpha,L^\alpha)\}\\
		\in \amalg_{\alpha\in A}\mathcal{X}^{\mathrm{conn}}(I_\alpha)
	}}}{\bigoplus}  
	\underset{\gamma,\alpha}{\bigotimes} P(K_\gamma^\alpha) \otimes \underset{\delta,\alpha}{\bigotimes} Q(L_\delta^\alpha)    \otimes \underset{\beta}{\bigotimes}R(J_\beta)~~\\
	\cong~&\underset{(I,J)\in \mathcal{X}^{\mathrm{conn}}(S)}{\bigoplus}~ \underset{{\substack{(K,L)\in \mathcal{Y}(S)\times \mathcal{Y}(S)\\ \K_S(K,L)=I}}}{\bigoplus}~  \underset{\gamma,\alpha}{\bigotimes} P(K_\gamma^\alpha) \otimes \underset{\delta,\alpha}{\bigotimes} Q(L_\delta^\alpha)    \otimes \underset{\beta}{\bigotimes}R(J_\beta)~~\\
	\cong~&\underset{{\substack{(K,L,J)\in \mathcal{Y}(S)^{\times 3} \\ \K_S(\K_S(K,L),J) = (S) }}}{\bigoplus}~  \underset{\gamma}{\bigotimes} P(K_\gamma) \otimes \underset{\delta}{\bigotimes} Q(L_\delta)    \otimes \underset{\beta}{\bigotimes}R(J_\beta)~~\\
	\underset{\scriptscriptstyle{\eqref{lem::K_associative}}}{\cong}~&\underset{{\substack{(K,L,J)\in \mathcal{Y}(S)^{\times 3} \\ \K_S(K,\K_S(L,J)) = (S) }}}{\bigoplus}~  \underset{\gamma}{\bigotimes} P(K_\gamma) \otimes \underset{\delta}{\bigotimes} Q(L_\delta)    \otimes \underset{\beta}{\bigotimes}R(J_\beta)~~\\
	\cong~&\underset{{\substack{ 
		(K,I)\\
		\in \mathcal{X}^{\mathrm{conn}}(S)
	}}}{\bigoplus} \underset{{\substack{(J,L)\in \mathcal{Y}(S)\times \mathcal{Y}(S)\\ \K_S(J,L)=I}}}{\bigoplus}~  \underset{\gamma}{\bigotimes} P(K_\gamma) \otimes \underset{\delta,\alpha}{\bigotimes} Q(L_\delta^\alpha)    \otimes \underset{\beta,\alpha}{\bigotimes}R(J_\beta^\alpha)~~\\
	\cong~& \underset{(K,I)\in \mathcal{X}^{\mathrm{conn}}(S)}{\bigoplus}~ \underset{\alpha}{\bigotimes}~P(K_\gamma) \otimes \underset{\alpha}{\bigotimes}\big(Q \boxtimes_c R\big)(I_\alpha) \\
	=~& \big(P \boxtimes_c (Q\boxtimes_c R)\big)(S).
	\end{aligned}
	\]
	Also, for all reduced $\mathfrak{S}$-modules $A$ and $B$, the endofunctor $\Phi_{A,B}$ is well defined by
	\begin{align*}
	\Phi_{A,B}(X) \cong~&	\underset{{\substack{(K,L,J)\in \mathcal{Y}(S)^{\times 3} \\ \K_S(\K_S(K,L),J) = (S) }}}{\bigoplus}~  \underset{\gamma}{\bigotimes} A(K_\gamma) \otimes \underset{\delta}{\bigotimes} X(L_\delta)    \otimes \underset{\beta}{\bigotimes}B(J_\beta)~~\\
	\cong~& \underset{n\in\N}{\bigoplus}~\underset{ L\in \mathcal{Y}_n(S)}{\bigoplus}~\underset{{\substack{(K,J)\in \mathcal{Y}(S)^{\times 2} \\ \K_S(\K_S(K,L),J) = (S) }}}{\bigoplus}~  \underset{\gamma}{\bigotimes} A(K_\gamma) \otimes \underset{\delta}{\bigotimes} X(L_\delta)    \otimes \underset{\beta}{\bigotimes}B(J_\beta)~~\\
	=:~& \underset{n\in\N}{\bigoplus}~(\Phi_{A,B})_n(X).
	\end{align*}
	where $(\Phi_{A,B})_n$ is an homogeneous polynomial functor of degree $n$; so, for all reduced $\mathfrak{S}$-modules $A$ and $B$, the functor $\Phi_{A,B}$ is a split analytic functor (in the sense of \cite[Sect. 4]{Val09}).
\end{proof}
\begin{prop}\label{prop::prop_fond_du_produit_connexe}
	The category $\big(\Smodred, \boxtimes_c, \tau, \Ibox\big)$ is an abelian symmetric monoidal category that preserves reflexive coequalizers and sequential colimits. 
\end{prop}
\begin{proof}
	Let $P$ and $Q$ be two reduced $\mathfrak{S}$-modules, we have, for all non-empty finite sets $S$, the isomorphism of $\mathfrak{S}_{|S|}$-modules
	\begin{align*}	
	P \boxtimes_c Q (S):=&~ \underset{(I,J)\in \mathcal{X}^{\mathrm{conn}}(S)}{\bigoplus}~\underset{\alpha}{\bigotimes}~P(I_\alpha) \otimes \underset{\beta}{\bigotimes}~Q(J_\beta)~~\\
	\cong& ~ \underset{(I,J)\in \mathcal{X}^{\mathrm{conn}}(S)}{\bigoplus}~\underset{\beta}{\bigotimes}~Q(J_\beta)\otimes \underset{\alpha}{\bigotimes}~P(I_\alpha) 
	\intertext{by symmetry of $\otimes$ of the category $\Ch_k$;  since $(I,J)$ is in $\mathcal{X}(S)$ if, and only if, $(J,I)$ is in $\mathcal{X}(S)$, we have the isomorphism}
	P \boxtimes_c Q (S) \cong& ~ \underset{(J,I)\in \mathcal{X}^{\mathrm{conn}}(S)}{\bigoplus}~\underset{\beta}{\bigotimes}~Q(J_\beta)\otimes \underset{\alpha}{\bigotimes}~P(I_\alpha) \cong ~Q \boxtimes_c P (S).
	\end{align*}
	We denote this isomorphism $\tau_{P,Q}(S) : P \boxtimes_c Q(S) \longrightarrow Q \boxtimes_c P(S)$, which gives us the symmetry of the product.
	The rest of the proof is similar to \cite[Prop 13]{Val09}.
\end{proof}

We have the following compatibility between these products:
\begin{prop}\label{prop::S_permute_prod_connexe}
	Let $P$ and $Q$ be two reduced $\mathfrak{S}$-modules. We have the following natural isomorphism of $\mathfrak{S}$-modules: 
	\[
		\mathbb{S}(P \boxtimes_c Q) \cong \mathbb{S}P \:\square\: \mathbb{S}Q.
	\]
\end{prop}
\begin{proof}
	Let $S$ be a finite set, we have 
	\begin{align*}
	\mathbb{S}\big(P\boxtimes_c Q\big)(S) =~& \underset{\Lambda\in \mathcal{Y}(S)}{\bigoplus} ~\underset{\gamma}{\bigotimes} \big(P\boxtimes_c Q \big)(\Lambda_\gamma) \\
	\underset{\scriptscriptstyle{\ref{def::prod_connexe_Smod}}}{=}~& \underset{\Lambda\in \mathcal{Y}(S)}{\bigoplus} ~\underset{\gamma}{\bigotimes} ~\underset{(I^\gamma,J^\gamma)\in \mathcal{X}^{\mathrm{conn}}(\Lambda_\gamma)}{\bigoplus} \underset{\alpha}{\bigotimes} P(I^\gamma_\alpha) \otimes \underset{\beta}{\bigotimes} Q(J^\gamma_\beta) \\
	\cong~& \underset{\Lambda\in \mathcal{Y}(S)}{\bigoplus} ~\ \underset{\{(I^\gamma,J^\gamma)\}\in\coprod_{\gamma} \mathcal{X}^{\mathrm{conn}}(\Lambda_\gamma)}{\bigoplus}~ \bigg(\underset{\alpha,\gamma}{\bigotimes} P(I^\gamma_\alpha) \otimes \underset{\beta,\gamma}{\bigotimes} Q(J^\gamma_\beta)\bigg) \\
	\underset{}{\cong}~& \underset{\Lambda\in \mathcal{Y}(S)}{\bigoplus} ~ \underset{(\widetilde{I},\widetilde{J})\in \K_S^{-1}(\Lambda)}{\bigoplus}~ \bigg(\underset{a}{\bigotimes} P(\widetilde{I}_a) \otimes \underset{b}{\bigotimes} Q(\widetilde{J}_b)\bigg) \\
	\underset{}{\cong}~& \underset{(\widetilde{I},\widetilde{J})\in X(S)}{\bigoplus}\bigg(\underset{\alpha}{\bigotimes} P(\widetilde{I}_\alpha) \otimes \underset{\beta}{\bigotimes} Q(\widetilde{J}_\beta)\bigg)~=~ \Big(\mathbb{S}P\:\square\:\mathbb{S}Q\Big)(S)
	\end{align*}
	which conclude the proof.
\end{proof}

\section{Connected product on \texorpdfstring{$\mathfrak{S}$}{S}-bimodules}
\label{sect::bimodule}
Just as in the $\mathfrak{S}$-modules, we give a description of the three monoidal structures on the category of (reduced) $\mathfrak{S}$-bimodules which are the equivalent of the three previous ones defined in $\Smodred_k$. We start with a section on the functors that encode the combinatorics of our monoidal products. We express the combinatorics of the different monoidal structures on the $\Sbimod_k$ category, using the same formalism as in the previous section. 

\subsection{Combinatorics of the connected composition of \texorpdfstring{$\mathfrak{S}$}{S}-bimodules}

\begin{defi}[Functors $\mathbb{Y}\ord$ and $\mathbb{Y}$]
	Let $S$ and $E$ be two finite sets. We define
	\begin{enumerate}
		\item the bifunctor $\mathbb{Y}\ord(-,-): \Fin \times \Fin\op \rightarrow \Set$ given on objects by $\mathbb{Y}\ord(S,E)=\amalg_{r\in\N^*} \mathbb{Y}\ord_r(S,E)$, with $\mathbb{Y}\ord_r(S,E):= \mathcal{Y}\ord_r(S)\times \mathcal{Y}\ord_r(E)\cong$
		\[
			\left\{(I,K)=\big((I_j,K_j)\big)_{j\in [\![1,r]\!] } \ \left| \ \begin{array}{c}
			\amalg_{j=1}^{r} I_j = S; \ \amalg_{j=1}^{r} K_j = E; \\
			\forall j\in [\![1,r]\!], I_j\ne\varnothing, K_j\ne\varnothing 
			\end{array}
			\right. \right\} \ , 
		\]
		where the elements of $\mathbb{Y}\ord_r(S,E)$ are ordered sets of pairs of sets. 
		
		\item $\mathbb{Y}(-,-): \Fin \times \Fin\op \rightarrow \Set$ the bifunctor given by $\mathbb{Y}(S,E)=\amalg_{r\in\N^*} \mathbb{Y}_r(S,E)$ with $\mathbb{Y}_r(S,E)=$
		\[
			\left\{ \{I,K\}\overset{\mathrm{not}}{:=}\{(I_\alpha,K_\alpha)\}_{\alpha\in A} \ \left| \ 
			\begin{array}{c}
			 A\in\Ob\Fin, \ |A|=r \\
			\amalg_{\alpha\in A} I_\alpha = S,~\amalg_{\alpha\in A} K_\alpha = E \\ 
			\forall \alpha\in A, I_\alpha \ne \varnothing,~K_\alpha\ne\varnothing \end{array} \right.\right\} \ ,
		\]
		where the elements of $\mathbb{Y}_r(S,E)$ are non ordered sets of pairs of sets. 
	\end{enumerate}
\end{defi}

Note that  $\mathbb{Y}_r(S,E) \not\cong \mathcal{Y}(S)\times\mathcal{Y}(E)$. As in the case of $\mathcal{Y}\ord(-)$, we have a free action of $\mathfrak{S}_r$ on $\mathbb{Y}\ord_r(S,E)$ given, for all $ (I,K)$ in $\mathbb{Y}^{\mathrm{ord}}_r(S,E)$ and all permutation $\sigma$ in $\mathfrak{S}_r$, by
\[
\sigma\cdot (I,K)=\big((I_{\sigma^{-1}(j)},K_{\sigma^{-1}(j)})\big)_{j\in [\![1,r]\!]}
\]
which induces the surjection $\mathbb{Y}^{\mathrm{or}}_r(S,E)\twoheadrightarrow \mathbb{Y}_r(S,E)$.

As in the case of $\mathfrak{S}$-modules, we define a new bifunctor, denoted by $\mathbb{X}^{\mathrm{conn}}$ which encodes connectedness. For $r$ and $s$ in $\N^*$ and for $S$ and $E$ two finite sets, the bifunctor $\mathbb{X}^{n,\mathrm{conn}}_{r,s}(S,E)$ is equal to
\[ 
	\left\{\left. 
	\begin{array}{c}
	\big(\{I,K'\},\{K'',J\} \big)\\
	\in \mathbb{Y}_r(S,[\![1,n]\!])\times \mathbb{Y}_s([\![1,n]\!],E)
	\end{array}
	\right| (K',K'')\in \mathcal{X}^{\mathrm{conn}}([\![1,n]\!])
	\right\}\bigg/_{\mathfrak{S}_n},
\]
where, for all $n$ in $\N^*$, by the functoriality of $\mathcal{X}^{\mathrm{conn}}$, the quotient by $\mathfrak{S}_n$ which identifies
\[
	\big(\{I,K'\}\cdot \sigma,\{K'',J\}\big) \sim \big(\{I,K'\},\sigma^{-1}\cdot\{K'',J\}\big) 
\]
is well defined. We also have the functors
\[
	\mathbb{X}^{n,\mathrm{conn}}(S,E):=\coprod_{r,s\in\N^*} \mathbb{X}^{n,\mathrm{conn}}_{r,s}(S,E)
	\ \mbox{and} \ 
	\mathbb{X}^{\mathrm{conn}}(S,E):=\coprod_{n\in\N^*} \mathbb{X}^{n,\mathrm{conn}}(S,E).
\]

\subsection{Monoidal products of the category \texorpdfstring{$\Sbimod_k$}{S-bimod} }

\subsubsection{Composition product of $\mathfrak{S}$-bimodules $\square$}
\begin{defi}[The composition product of $\mathfrak{S}$-bimodules] \label{def::prod_compo_bimod}
The \emph{the composition product of $\mathfrak{S}$-bimodules} is the bifunctor of
\[
	-\:\square\:- : \Sbimodred_k\times\Sbimodred_k \longrightarrow \Sbimodred_k
\]
defined, for all reduced $\mathfrak{S}$-bimodules $P$ and $Q$, and for all finite sets $S$ and $E$, by 
\begin{align*}
	\big( P\:\square\: Q\big)(S,E) :=~ &\underset{n\in\N^*}{\bigoplus}~\Big(\underset{(\{I,K'\},\{K'',J\})\in \mathbb{X}^n_{1,1}(S,E)}{\bigoplus} P(I,K')\underset{}{\otimes} Q(K'',J)\Big)\Big/_{\mathfrak{S}_n} \\
	\cong~ &\underset{n\in\N^*}{\bigoplus} P(S,	[\![1,n]\!])\underset{\mathfrak{S}_n}{\otimes} Q([\![1,n]\!],E)~;
\end{align*}
where the action of $\mathfrak{S}_n$ is induced by the action on $[\![1,n]\!]$.
\end{defi}

\begin{rem}\label{rem::prod_composition}
The  composition product $\square$ is defined by a coend. In fact, the tensorial product of the category $\Ch_k$ gives us the external product
\[
	\begin{array}{ccc}
	\Func(\Fin\times \Fin\op,\Ch_k)^{\times 2} & \longrightarrow & \Func(\Fin\times \Fin\op\times \Fin\times \Fin\op,\Ch_k) \\
	(P,Q) & \longmapsto & \big\{(S_1,E_1,S_2,E_2) \mapsto P(S_1,E_1)\otimes Q(S_2,E_2)\}
	\end{array} \ ,
\]
and, by taking the coend of the functors $P(S_1,-)\otimes Q(-, E_2) : \Fin\op\times \Fin \rightarrow \Fin$ for $S_1$ and $E_2$ two finite sets, we have
\[
	\begin{array}{rccc}
	-\:\square\: - :&  \Func(\Fin\times \Fin\op,\Ch_k)^{\times 2}& \longrightarrow & \Func(\Fin\times \Fin\op,\Ch_k) \\
	& (P,Q) &\longmapsto & \int^{S\in \Fin}  P(-,S)\otimes Q(S,-)
	\end{array}.
\]
As the category $\mathfrak{S}$ is a skeleton of the category $\Fin$, we finally have
\[
	\begin{array}{rccc}
	-\:\square\: - :&  \Func(\Fin\times \Fin\op,\Ch_k)^{\times 2} & \longrightarrow & \Func(\Fin\times \Fin\op,\Ch_k) \\
	& (P,Q) &\longmapsto & \displaystyle{\bigoplus}_{n\in\N}  P(-,[\![1,n]\!])\underset{\mathfrak{S}_n}{\otimes} Q([\![1,n]\!],-)
	\end{array}.
\]
\end{rem}

\begin{rem}
Let $P$ and $Q$ be two reduced $\mathfrak{S}$-modules. For all $m$ and $n$ in $\N^*$, we have the following isomorphism of $\mathfrak{S}_m\times\mathfrak{S}_n\op$-bimodules:
\[
	P\:\square\: Q(m,n)\cong \underset{N\in\N^*}{\bigoplus}\: P(m,N)\underset{k[\mathfrak{S}_N]}{\otimes} Q(N,n) .
\]
\end{rem}

\begin{prop}
	The category $\big(\Sbimod_k,\square,I_\square\big)$ with  $I_\square(S,E)=k[\mathrm{Aut}(S)]$ for $S\cong E$ and $0$ otherwise, is a monoidal category.
\end{prop}
\begin{proof}
Let $P$, $Q$ and $R$ be three reduced $\mathfrak{S}$-bimodules and $S$ and $E$ be two non-empty finite sets. We note by $e$ the cardinal of the set $E$. By definition of  $I_\square$, we have the following isomorphisms:
\begin{eqnarray*}
	P\:\square\:I_\square(S,E) & \cong &\bigoplus_{n\in\N}  P(S,[\![1,n]\!])\underset{\mathfrak{S}_n}{\otimes} I_\square([\![1,n]\!],E) \\
	& \cong & P(S,[\![1,e]\!])\underset{\mathfrak{S}_e}{\otimes} I_\square([\![1,e]\!],E) \\
	& \cong & P(S,[\![1,e]\!])\underset{\mathfrak{S}_e}{\otimes} k[\mathrm{Aut}([\![1,e]\!])] \\
	& \cong & P(S,E).
\end{eqnarray*}
By \Cref{rem::prod_composition}, we also have the isomorphisms:
\begin{eqnarray*}
	\big( P\:\square \:Q\big)\:\square\: R(S,E) & = & \int^{U\in\Fin}\big( P\:\square \:Q\big)(S,U)\otimes R(U,E) \\
	& \cong & \int^{U\in\Fin}\int^{V\in\Fin} P(S,V)\otimes Q(V,U)\otimes R(U,E) \\
	& \underset{\mathrm{Fubini}}{\cong} & \int^{(U,V)\in\Fin^{\times 2}} P(S,V)\otimes Q(V,U)\otimes R(U,E) \\
	& \cong & 	 P\:\square \big( \:Q\:\square\: R\big)(S,E)
\end{eqnarray*}
\end{proof}

\subsubsection{The concatenation product $\otimes^{\mathrm{conc}}$}
As in the case of $\mathfrak{S}$-modules, we define the concatenation product  of two reduced $\mathfrak{S}$-bimodules. 
\begin{defi}[Concatenation product of $\mathfrak{S}$-bimodules] \label{def::prod_conca_bimod}
	The \textit{concatenation product} is the bifunctor
	\[
		-\otimes^{\mathrm{conc}}- : \Sbimodred_k \times \Sbimodred_k \longrightarrow \Sbimodred_k,
	\]
	defined, for $P$ and $Q$ two reduced $\mathfrak{S}$-bimodules and for all finite sets $S$ and $E$, by
	\[
		\big( P \otimes^{\mathrm{conc}} Q\big)(S,E):= \underset{(I,K) \in \mathbb{Y}_2^{\mathrm{ord}}(S,E)}{\bigoplus} ~P(I_1,K_1)\otimes Q(I_2,K_2) .
	\]
\end{defi}

\begin{rem}
This product is a particular case of a Day convolution product: let $P$ and $Q$ be two reduced $\mathfrak{S}$-bimodules, the bifunctor $P\otimes^{\mathrm{conc}} Q (-_1,-_2)$ is given by	
\[
	\int^{(I_1,I_2,J_1,J_2)} k\left[\Hom_{\Fin\op\times\Fin}\big((I_1\amalg I_2,J_1\amalg J_2), (-_1,-_2)\big)\right]\otimes P(I_1,J_1)\otimes Q(I_2,J_2).
\]
\end{rem}


The two product $\square$ and $\otimes^{\mathrm{conc}}$ satisfy the interchanging law.
\begin{prop}[Interchanging law]\label{prop::injection_compatibilite} 
	Let $A,B,C,D,E$ and $F$ be reduced $\mathfrak{S}$-bimodules. We have the natural  injection of $\mathfrak{S}$-bimodules 
	\[
		\Phi_{A,B,C,D} : (A\:\square\:B)\otimes^{\mathrm{conc}}(C\:\square\:D)\hookrightarrow (A\otimes^{\mathrm{conc}}C)\:\square\:(B\otimes^{\mathrm{conc}}D),
	\]
	which is associative, i.e. we have
	\[
		\Phi_{A\otimes C, B\otimes D, E,F}\left((\Phi_{A,B,C,D})\otimes (E\square F)\right) 
		= \Phi_{A,B,C\otimes E,D\otimes F}\left( (A\square B)\otimes \Phi_{C,D,E,F}\right).
	\]
\end{prop}

\begin{proof}
Let $A,B,C$ and $D$ be reduced $\mathfrak{S}$-bimodules. We have
\begin{align*}
	&(A\:\square\:B) \otimes^{\mathrm{conc}} (C\:\square\:D)(-,-)\\
	\overset{}{=}~&\int^{(I_1,I_2,J_1,J_2)\in\Fin^{\times 4}} 
	{\substack{k\left[\Hom\left((I_1\amalg I_2,J_1\amalg J_2),(-,-)\right)\right] \\
			\otimes (A\:\square\:B)(I_1,J_1) \otimes (C\:\square\:D)(I_2,J_2)}} \\
	\overset{}{=}~&\int^{(I_1,I_2,J_1,J_2)\in\Fin^{\times 4}} {\scriptstyle{k\left[\Hom\left((I_1\amalg I_2,J_1\amalg J_2),(-,-)\right)\right]}} \\
	& \quad\quad \otimes \int^{S\in\Fin} {\scriptstyle{A(I_1,S)\otimes B(S,J_1)}} 
	\otimes \int^{T\in\Fin} \scriptstyle{C(I_2,T)\otimes D(T,J_2)} \\
	\underset{\mathrm{Fubini}}{\cong}~& \int^{(I_1,I_2,J_1,J_2,S,T)\in\Fin^{\times 6}}
	{\substack{ k\left[\Hom\left((I_1\amalg I_2,J_1\amalg J_2),(-,-)\right)\right] \\ \otimes  A(I_1,S) \otimes C(I_2,T)\otimes B(S,J_1)\otimes D(T,J_2)}}. 
\end{align*}
The natural injections
\begin{eqnarray*}
	&&	\scriptstyle{k\big[\Hom\big((I_1\amalg I_2,J_1\amalg J_2),(-,-)\big)\big]}  \\
	 & \hookrightarrow &
	\displaystyle{\coprod_S } \scriptstyle{k\big[\Hom\big((I_1\amalg I_2,J_1\amalg J_2),(-,S)\big)\big]\otimes k\big[\Hom\big((I_1\amalg I_2,J_1\amalg J_2),(S,-)\big)\big] } \\
	& \hookrightarrow &	\displaystyle{\coprod_{{\substack{S\\ U\rightarrow S\\ S\rightarrow V}}} } \scriptstyle{k\big[\Hom\big((I_1\amalg I_2,U),(-,S)\big)\big]\otimes k\big[\Hom\big((V,J_1\amalg J_2),(S,-)\big)\big] } 
\end{eqnarray*}
and
\begin{eqnarray*}
	\scriptstyle{A(I_1,S) \otimes C(I_2,T)\otimes B(S,J_1)\otimes D(T,J_2)} & \hookrightarrow  &
	\displaystyle{ \coprod_{{\substack{J_1,J_2,K_1,K_2}}} }\scriptstyle{A(I_1,J_1) \otimes C(I_2,J_2)\otimes B(K_1,L_1)\otimes D(K_2,L_2)}
\end{eqnarray*}

give us the natural injection
\begin{align*}
	&(A\:\square\:B) \otimes^{\mathrm{conc}} (C\:\square\:D)(-,-)\\
	\hookrightarrow~& \int^{(S,I_1,I_2,J_1,J_2,K_1,K_2,L_1,L_2)}
	{\substack{
	k\big[\Hom\big((I_1\amalg I_2,J_1\amalg J_2),(-,S)\big)\big]\\
	\otimes k\big[\Hom\big((K_1\amalg K_2,L_1\amalg L_2),(S,-)\big)\big] \\ \qquad \otimes A(I_1,J_1) \otimes C(I_2,J_2)\otimes B(K_1,L_1)\otimes D(K_2,L_2) }}\\
	\underset{\mathrm{Fubini}}{\cong}~& \int^{S} \int^{I_i,J_i}
	{\substack{ k\big[\Hom\big((I_1\amalg I_2,J_1\amalg J_2),(-,S)\big)\big]\\
			\otimes A(I_1,J_1) \otimes C(I_2,J_2) }}
		 \otimes \int^{K_i,L_i} 
	{\substack{k\big[\Hom\big((K_1\amalg K_2,L_1\amalg L_2),(S,-)\big)\big]\\
			\otimes B(K_1,L_1)\otimes D(K_2,L_2) }}\\
	= ~& \int^{S} \big(A\otimes^{\mathrm{conc}} C\big)(-,S) \otimes \big(B\otimes^{\mathrm{conc}}D\big)(S,-) \\
	= ~& \big(A\otimes^{\mathrm{conc}} C\big)\:\square\: \big(B\otimes^{\mathrm{conc}}D\big)(-,-).
\end{align*}
Then we have the injective natural transformation between bifunctors: 
\[
	\Phi_{A,B,C,D} : (A\:\square\:B)\otimes^{\mathrm{conc}}(C\:\square\:D)\hookrightarrow (A\otimes^{\mathrm{conc}}C)\:\square\:(B\otimes^{\mathrm{conc}}D).
\]
Moreover, $\Phi$ is associative, because the following diagram is commutative:
\[
	\xymatrix@C=-7pc@R=0.7cm{
	\Big((A\:\square\:B)\otimes^{\mathrm{conc}}(C\:\square\:D)\Big)\otimes^{\mathrm{conc}}(E\:\square\:F) \ar[d]_{\Phi_{A,B,C,D}\otimes 1} \ar[rr]^{\cong}
	& \ar@{}[dd]|(0.4)\commu & (A\:\square\:B)\otimes^{\mathrm{conc}}\Big( (C\:\square\:D)\otimes^{\mathrm{conc}}(E\:\square\:F) \Big) \ar[d]^{1\otimes \Phi_{C,D,E,F}} \\
	\Big((A\otimes^{\mathrm{conc}}C)\:\square\:(B\otimes^{\mathrm{conc}}D)\Big)\otimes^{\mathrm{conc}}(E\:\square\:F) \ar@/_1pc/[dr]|(0.25){\Phi_{A\otimes C,B\otimes D,E,F}}
	& & (A\:\square\:B)\otimes^{\mathrm{conc}}\Big( (C)\otimes^{\mathrm{conc}}E)\:\square\:(D\otimes^{\mathrm{conc}}F) \Big) \ar@/^1pc/[dl]|(0.25){\Phi_{A,B,C\otimes E,D\otimes F}} \\
	& (A\otimes^{\mathrm{conc}}C \otimes^{\mathrm{conc}} E)\:\square\:(B\otimes^{\mathrm{conc}}D\otimes^{\mathrm{conc}} F) &  
	}
\]
\end{proof}

\begin{cor} 
	The categories of monoids  $\big(\As(\Sbimod_k,\otimes^{\mathrm{conc}}),\square,I_\square \big)$ and commutative monoids $\big(\Com(\Sbimod_k,\otimes^{\mathrm{conc}}),\square,I_\square \big)$ are monoidal. In other words, if $(P,c_P)$ and $(Q,c_Q)$ are two monoids in the symmetric monoidal category (without unit) $\big(\Sbimod_k,\otimes^{\mathrm{conc}}\big)$, then
	\begin{enumerate}
	\item  $(P\:\square\:Q,c_{P\:\square\:Q})$ is a monoid in  $\big(\Sbimod_k,\otimes^{\mathrm{conc}}\big)$; 	
	\item if $P$ and $Q$ are commutatives monoids, then $P\:\square\:Q$ is too.
	\end{enumerate}
\end{cor}

\begin{proof}
\begin{enumerate}
	\item The injection  $\Phi$ of the \Cref{prop::injection_compatibilite} gives us the product $c_{P\:\square\:Q}$ :
	\[
	\xymatrix{
		(P\:\square\:Q )\otimes^{\mathrm{conc}} (P\:\square\:Q ) \ar@{.>}@/_1pc/[rd]_{c_{P\:\square\:Q}}\ar[r]^{\Phi} & (P\otimes^{\mathrm{conc}} P)\:\square\:(Q\otimes^{\mathrm{conc}} Q) \ar[d]^{c_P\square c_Q} \\
		& P\:\square\: Q~
	}
	\]	
	and the associativity of the product $\square$ gives us the associativity of the product $c_{P\:\square \:Q}$.
	\item The commutativity of $c_{P\square Q}$ follows from the commutativity of $c_P$ and $c_Q$.
\end{enumerate}
\end{proof}

\begin{defi}[Free monoids]
Let $P$ be a reduced $\mathfrak{S}$-bimodule and $S$ and $E$ be two finite sets.
\begin{enumerate}
	\item The free associative monoid without unit on $P$ is defined by
	\[
		\mathbb{T}^rP(S,E):= \underset{(I,K) \in \mathbb{Y}_r^{\mathrm{ord}}(S,E)}{\bigoplus} ~  P(I_1,K_1)\otimes \ldots \otimes P(I_r,K_r).
	\]
	\item The free commutative monoid without unit on $P$ is defined by the quotient of the free associative monoid by the action of the symmetric groups; namely
	\[
		\mathbb{S}^rP(S,E):= \Big(\mathbb{T}^rP(S,E)\Big)_{\mathfrak{S}_r} \cong \underset{\{I,K\} \in \mathbb{Y}_r(S,E)}{\bigoplus} ~\underset{\alpha}{\bigotimes} P(I_\alpha,K_\alpha).
	\]
\end{enumerate}
\end{defi}

\subsection{Connected composition product of \texorpdfstring{$\mathfrak{S}$}{S}-bimodules}\label{sect::produit_connexe_Sbimod}
We define the connected composition product of $\mathfrak{S}$-bimodules. The product was defined the first time by Vallette in his PhD thesis \cite{Val03}, for studying the homotopic comportment of algebraic structures with several inputs and outputs. Our definition is not the original one, but we show (cf. proposition \ref{prop::equivalence_produits_connexes2}) that they are equivalent.

\begin{defi}[Product of connected composition $\boxtimes^b_c$] \label{def::prod_compo_conn_bimod}
	The \textit{product of connected composition} of reduced $\mathfrak{S}$-bimodules is the bifunctor
	\[
	-\boxtimes^b_c - : \Sbimodred_k \times \Sbimodred_k \longrightarrow \Sbimodred_k
	\]
	defined,  for all reduced $\mathfrak{S}$-bimodules $P$ and $Q$ and all pairs $(S,E)$ of finite sets, by $\big(P\boxtimes^b_c Q\big)(S,E) :=$
	\[
	 \underset{n\in\N^*}{\bigoplus}~\Big(\underset{(\{I,K'\},\{K'',J\})\in\mathbb{X}^{n,\mathrm{conn}}(S,E)}{\bigoplus} ~\underset{\alpha}{\bigotimes} P(I_\alpha,K'_\alpha) \underset{}{\otimes} ~\underset{\beta}{\bigotimes} Q(K''_\beta, J_\beta)\Big)\Big/_{\mathfrak{S}_n}\ ,
	\]
	where the quotient by $\mathfrak{S}_n$ identifies, for all $\sigma$ in $\mathfrak{S}_n$, the terms
	\[
	\underset{\alpha}{\bigotimes}	P(I_\alpha,K'_\alpha) \underset{}{\otimes} ~\underset{\beta}{\bigotimes} Q(K''_\beta, J_\beta)
	\sim
	\Big(\underset{\alpha}{\bigotimes}P(I_\alpha,K'_\alpha) \Big)\cdot \sigma^{-1}\underset{}{\otimes} ~\sigma\cdot\Big(\underset{\beta}{\bigotimes} Q(K''_\beta, J_\beta)\Big) .
	\] 
\end{defi}
\begin{rem}
	This construction is functorial, because $(P\boxtimes^b_c Q)(-,-)$ is a sub-bifunctor of $\big(\mathbb{S}P\:\square\:\mathbb{S}Q\big)(-,-)$ (see \Cref{lem::SPsquareSQ}).
\end{rem}
\begin{nota}\label{not::prod_connexe}
	We denote by $\big/_\mathfrak{S}$, the quotient by symmetric groups for
	\[
	\big(P\boxtimes^b_c Q\big)(S,E) :=
	\Big(\underset{{\substack{ 
	 	(\{I,K'\},\{K'',J\})\\ 
	 	\in\mathbb{X}^{\mathrm{conn}}(S,E)
 	}}}{\bigoplus} ~\underset{\alpha}{\bigotimes} P(I_\alpha,K'_\alpha) \underset{}{\otimes} ~\underset{\beta}{\bigotimes} Q(K''_\beta, J_\beta)\Big)\Big/_{\mathfrak{S}}
	\]
\end{nota}
The following proposition says that our definition  of the connected composition product of $\mathfrak{S}$-bimodules is equivalent to that of Vallette in \cite{Val03, Val07}.
First, recall the notion of connected permutation.

\begin{defi}[Connected permutation -- \textup{\cite[Sect 1.3]{Val07}}]
	Let $a,b$ and $N$ be three integers with $a$ and $b$ in $\N^*$, let  $\bar{\alpha}=(\alpha_1,\ldots,\alpha_a)$ in $(\N^*)^a$ be an $a$-tuple and  $\bar{\beta}=(\beta_1,\ldots,\beta_b)$ in $(\N^*)^b$ be a $b$-tuple such that $|\bar{\alpha}|=N=|\bar{\beta}|$. A \emph{$(\bar{\alpha},\bar{\beta})$-connected permutation} $\sigma$ of $\mathfrak{S}_N$ is a permutation of $\mathfrak{S}_N$ such that the graph of a geometric representation of $\sigma$ is connected if one gathers the inputs labelled by $\alpha_1+\ldots+\alpha_i+1,\ldots,\alpha_1+\ldots+\alpha_i+\alpha_{i+1}$ for $0\leqslant i \leqslant a-1$, and the outpus labelled by $\beta_1+\ldots+\beta_i+1,\ldots,\beta_1+\ldots+\beta_i+\beta_{i+1}$ for $0\leqslant i \leqslant b-1$. The set of $(\bar{\alpha},\bar{\beta})$-connected permutations is denoted by $S_c^{\bar{\beta},\bar{\alpha}}$.
\end{defi}

\begin{lem}\label{lem::combinatoire_connexite}
	Let  $r,s$ and $N$ be in $\N^*$, and let $\bar{k}$ be an $r$-tuple in $(\N^*)^r$ and  $\bar{j}$ be a $s$-tuple in $(\N^*)^s$ such that $\sum_{\alpha=1}^rk_\alpha=N=\sum_{\beta=1}^s j_\beta$. The map
	\[	
		\begin{array}{rccc}
			\phi : &\left\{ (K,J)\in \mathcal{X}^{\mathrm{conn,ord}}_{r,s}(N) \left|
			{\substack{(|K_1|,\ldots,|K_r|)=\bar{k},\\
					(|J_1|,\ldots,|J_s|)=\bar{j}
			}} \right.\right\}
			& \longrightarrow & S_c^{\bar{k},\bar{j}} \\
			& (K,J) & \longmapsto & \sigma_K^{-1}\sigma_J
		\end{array}
	\]
	is surjective.
\end{lem}

\begin{proof}
	By the definition of the connectedness of a pair $(K,J)$ in $\mathcal{X}^{\mathrm{conn,ord}}_{r,s}(N)$, if we consider the graph of a geometric representation of the permutation $\sigma_K^{-1}\sigma_J$, where we gather the inputs and the outputs as in the definition of connected permutation, then there exists a path between every input labelled by $i$ and every output labelled by $j$. Then, we have $\sigma_K^{-1}\sigma_J$ in  $S_c^{\bar{k},\bar{j}}$. 
	
	Let $\sigma$ be a permutation in $S_c^{\bar{k},\bar{j}}$. We consider
	\[
		[\![1,N]\!]_{\bar{j}}\overset{not}{:=}\Big \{ [\![1,j_1]\!],[\![j_1+1,j_1+j_2]\!],\ldots,[\![j_1+\ldots+j_{s-1}+1,j_1+\ldots+j_{s}]\!] \Big\},
	\]
	and we denote by $\left( [\![1,N]\!]\cdot \sigma\right)_{\bar{k}}$, the following set
	\[
		\Big \{\! \{\sigma^{-1}(1),\sigma^{-1}(2),\ldots\sigma^{-1}(k_1)\},\!\ldots\!,\{\sigma^{-1}(k_1+\ldots+k_{r-1}+1),\!\ldots\!,\sigma^{-1}(k_1+\ldots+k_{r})\} \!\Big\};
	\]
	the pair $( [\![1,N]\!]\cdot \sigma)_{\bar{k}},[\![1,N]\!]_{\bar{j}})$ is an element of $\mathcal{X}^{\mathrm{conn,ord}}_{r,s}(N)$ by the connectedness of the permutation $\sigma$. So the application 
	\[
	\begin{array}{rccc}
		\psi :& S_c^{\bar{k},\bar{j}} & \longrightarrow & 
		\left\{ (K,J)\in \mathcal{X}^{\mathrm{conn,ord}}_{r,s}(N) \left|
			{\substack{(|K_1|,\ldots,|K_r|)=\bar{k},\\
			(|J_1|,\ldots,|J_s|)=\bar{j}
			}} \right.\right\} \\
		&\sigma & \longmapsto &		\big(([\![1,N]\!]\cdot\sigma)_{\bar{k}},[\![1,N]\!]_{\bar{j}}\big)
	\end{array}
	\]
	is well-defined and gives us a section of $\phi$, so that $\phi$ is surjective.
\end{proof}

\begin{prop}\label{prop::equivalence_produits_connexes2}
	For all integers $m$ and $n$ in $\N^*$, we have the isomorphism of $\mathfrak{S}_m\times \mathfrak{S}_n\op$-modules:
	\[
		\big(P\boxtimes^b_c Q\big)([\![1,m]\!],[\![1,n]\!]) \cong \big(P\boxtimes_c^{\mathrm{Val}} Q\big)(m,n).
	\]
\end{prop}
\begin{proof}
	Let $m$ and $n$ be two integers in $\N^*$. We have
	\begin{align*}
	\big(&P\boxtimes^b_c Q\big)([\![1,m]\!],[\![1,n]\!]) :=~ \\
	& \bigg(\underset{N,r,s\in\N}{\bigoplus}~\underset{{\substack{(\{I,K'\},\{K'',J\})\\ \in\mathbb{X}^{N,\mathrm{conn}}([\![1,m]\!],[\![1,n]\!])}}}{\bigoplus} ~\underset{\alpha}{\bigotimes} P(I_\alpha,K'_\alpha) \underset{}{\otimes} ~\underset{\beta}{\bigotimes} Q(K''_\beta, J_\beta)\bigg)\bigg/_\mathfrak{S}.\\
	\cong~&
	\bigg(\underset{N\in\N}{\bigoplus}
	\bigg(\underset{{\substack{
				\{I,K'\}\in  \mathbb{Y}\ord_r([\![1,m]\!],[\![1,N]\!])\\ \{K'',J\}\in\mathbb{Y}\ord_s([\![1,N]\!],[\![1,n]\!])\\
				(K',K'')\in \mathcal{X}^{\mathrm{conn}}_{r,s}([\![1,N]\!])
			}}}{\bigoplus} 
	\Big(\bigotimes_{\alpha=1}^{r} P(I_\alpha,K'_\alpha)\Big)_{\mathfrak{S}_r} \underset{}{\otimes} 
	\Big(\bigotimes_{\beta=1}^{s} Q(K''_\beta, J_\beta)\Big)_{\mathfrak{S}_s}\bigg)\bigg)\bigg/_\mathfrak{S} \\
	\cong~& \bigg(\underset{N\in\N}{\bigoplus}~\bigg(~\underset{{\substack{\{I,K'\}\in  \mathbb{Y}\ord_r([\![1,m]\!],[\![1,N]\!])\\ \{K'',J\}\in\mathbb{Y}\ord_s([\![1,N]\!],[\![1,n]\!])\\ (K',K'')\in \mathcal{X}^{\mathrm{conn}}_{r,s}([\![1,N]\!])}}}{\bigoplus} ~ \bigotimes_{\alpha=1}^{r} P(I_\alpha,K'_\alpha) \underset{}{\otimes} ~\bigotimes_{\beta=1}^{s} Q(K''_\beta, J_\beta)\bigg)_{\mathfrak{S}_r\times\mathfrak{S}_s}\bigg)\bigg/_\mathfrak{S}~~\\
	\end{align*}
	\begin{align*} 
	\cong~& \bigg(\underset{N\in\N}{\bigoplus}
	\bigg(\underset{{\substack{
				\bar{l}\in(\N^*)^r, |\bar{l}|=m\\ 
				\bar{i}\in(\N^*)^s, |\bar{i}|=n\\ 
				(K',K'')\in \\
				X^{\mathrm{conn},\mathrm{ord}}_{r,s}([\![1,N]\!])
			}}}{\bigoplus} k[\mathfrak{S}_m]
		\underset{\prod \mathfrak{S}_{l_\alpha}}{\otimes} 
		\bigotimes_{\alpha=1}^{r} P(l_\alpha,K'_\alpha) \otimes \bigotimes_{\beta=1}^{s} Q(K''_\beta, i_\beta)
		\underset{\prod\mathfrak{S}_{i_\beta}}{\otimes} k[\mathfrak{S}_n] \bigg)_{\mathfrak{S}_r\times\mathfrak{S}_s}\bigg)\bigg/_\mathfrak{S}\\
	\overset{\phi}{\longrightarrow}~& \underset{N\in\N}{\bigoplus}~\bigg(~\underset{{\substack{\bar{l},\bar{k}\in(\N^*)^r, |\bar{l}|=m\\ \bar{i},\bar{j}\in(\N^*)^s, |\bar{i}|=n\\ |\bar{k}|=N=|\bar{j}|}}}{\bigoplus} ~k[\mathfrak{S}_m]\underset{\prod \mathfrak{S}_{l_\alpha}}{\otimes} \Big(\bigotimes_{\alpha=1}^{r} P(l_\alpha,k_\alpha)\Big) \underset{\mathfrak{S}_{\bar{k}}}{\otimes} k[S_c^{\bar{k},\bar{j}}]\underset{\mathfrak{S}_{\bar{j}}}{\otimes} ~\Big(\bigotimes_{\beta=1}^{s} Q(j_\beta, i_\beta) \Big)\underset{\prod\mathfrak{S}_{i_\beta}}{\otimes} k[\mathfrak{S}_n] \bigg)_{\mathfrak{S}_r\times\mathfrak{S}_s}\\
	\end{align*}
	where $\phi$ sends, for $(K',K'')\in \mathcal{X}^{\mathrm{conn}}_{r,s}([\![1,N]\!])$ fixed, the component
	\[
		\bigg(k[\mathfrak{S}_m]\underset{\prod \mathfrak{S}_{l_\alpha}}{\otimes} \bigotimes_{\alpha=1}^{r} P(l_\alpha,K'_\alpha)  \underset{}{\otimes}  ~\bigotimes_{\beta=1}^{s} Q(K''_\beta, i_\beta)\Big)\underset{\prod\mathfrak{S}_{i_\beta}}{\otimes} k[\mathfrak{S}_n]\bigg/_{\mathfrak{S}_N}
	\]
	to the following $\mathfrak{S}_m\times\mathfrak{S}_n^{\mathrm{op}}$-module
	\begin{align*}
	&k[\mathfrak{S}_m]\underset{\prod \mathfrak{S}_{l_\alpha}}{\otimes} \bigotimes_{\alpha=1}^{r} P\big(l_\alpha,\big[\!\!\big[\sum_{j=1}^{\alpha-1}|K'_j|+1,\sum_{j=1}^{\alpha}|K'_j|\big]\!\!\big]\big)
	\underset{\mathfrak{S}_{\bar{k}}}{\otimes} \sigma_{K'}^{-1}
	\underset{\mathfrak{S}_N}{\otimes}\sigma_{K''} \\
	& \qquad \qquad  \underset{\mathfrak{S}_{\bar{j}}}{\otimes}  ~\bigotimes_{\beta=1}^{s} Q\big(\big[\!\!\big[\sum_{j=1}^{\beta-1}|K''_j|+1,\sum_{j=1}^{\beta}|K''_j|\big]\!\!\big], i_\beta\big) \underset{\prod\mathfrak{S}_{i_\beta}}{\otimes} k[\mathfrak{S}_n] \ ,
	\end{align*}
	which is isomorphic to
	\begin{align*}
	\cong & k[\mathfrak{S}_m]\underset{\prod \mathfrak{S}_{l_\alpha}}{\otimes} \bigotimes_{\alpha=1}^{r} P\big(l_\alpha,\big[\!\!\big[\sum_{j=1}^{\alpha-1}|K'_j|+1,\sum_{j=1}^{\alpha}|K'_j|\big]\!\!\big]\big) 
	\underset{\mathfrak{S}_{\bar{k}}}{\otimes}  \sigma_{K'}^{-1}\sigma_{K''} \underset{\mathfrak{S}_{\bar{j}}}{\otimes}  \\
	& \qquad ~\bigotimes_{\beta=1}^{s} 
	Q\big(\big[\!\!\big[\sum_{j=1}^{\beta-1}|K''_j|+1,\sum_{j=1}^{\beta}|K''_j|\big]\!\!\big], i_\beta\big) 
	\underset{\prod\mathfrak{S}_{i_\beta}}{\otimes} k[\mathfrak{S}_n] \ .
	\end{align*}
	However, by  \Cref{lem::combinatoire_connexite}, the morphism $\phi$ is surjective, and the quotient by the group $\mathfrak{S}_{\bar{k}}\times\mathfrak{S}_{\bar{j}}$ gives us the injectivity.
\end{proof}

\begin{prop}[{\cite[Lem. 49]{Val03}}]
	The category $\big(\Sbimod_k, \boxtimes_c\Val, \Ibox\Val\big)$ where the unit $I$ is defined, for all pairs $(S,E)$ of finite sets, by:
	\[
	\Ibox\Val(S,E):= \left\{
	\begin{array}{rl}
	k & \mbox{ if } |S|=1=|E| \\
	0 & \mbox{ otherwise}
	\end{array}\right. ,
	\]
	is abelian monoidal and  preserves coequalizers and sequential colimits.
\end{prop}

The $\mathfrak{S}$-bimodule $P \boxtimes\Val_c Q$ appears as the  indecomposables for the product $\otimes^{\mathrm{conc}}$ of $\mathbb{S}P \:\square\:\mathbb{S}Q$.
\begin{lem}\label{lem::SPsquareSQ}
	Let $P$ and $Q$ be two reduced $\mathfrak{S}$-bimodules. We have, for all finite sets $S$ and $E$, the following isomorphism
	\[
		\big(\mathbb{S}P\:\square\:\mathbb{S}Q\big)(S,E) \cong \underset{n\in\N^*}{\bigoplus}
		\Big(\underset{{\substack{ 
			(\{I,K'\},\{K'',J\}) \\ 
			\in\mathbb{X}^n(S,E)
		}}}{\bigoplus}
		\underset{\alpha}{\bigotimes} P(I_\alpha,K'_\alpha) \underset{}{\otimes} \underset{\beta}{\bigotimes} Q(K''_\beta,J_\beta)\Big)\Big/_{\mathfrak{S}_n}.
	\]
\end{lem}
\begin{proof}
	Let $S$ and $E$ be two finite sets, then
	\begin{align*}
	\big(\mathbb{S}P\:\square\:&\mathbb{S}Q\big)(S,E) \underset{}{\cong}~ \underset{n\in\N^*}{\bigoplus}~\Big(\underset{(\{A,B\},\{C,D\})\in\mathbb{X}_{1,1}^n(m,n)}{\bigoplus} \mathbb{S}P(A,B)\underset{}{\otimes} \mathbb{S}Q(C,D)\Big)\Big/_{\mathfrak{S}_n} \\
	\cong~& \underset{n\in\N^*}{\bigoplus}~\Big(
	\underset{{\substack{
				(\{A,B\},\{C,D\})\\ 
				\in\mathbb{X}_{1,1}^n(S,E)
			}}}{\bigoplus} 
	~\underset{{\substack{ 
				\{I,K'\}\in \mathbb{Y}(A,B) \\ \{K'',J\}\in \mathbb{Y}(C,D) }}}{\bigoplus} 
	~\underset{\alpha}{\bigotimes} P(I_\alpha,K'_\alpha) \underset{}{\otimes} ~\underset{\beta}{\bigotimes} Q(K''_\beta, J_\beta) \Big)\Big/_{\mathfrak{S}_n}\\
	\cong~& \underset{n\in\N^*}{\bigoplus}~\Big(\underset{(\{I,K'\},\{K'',J\})\in\mathbb{X}^n(S,E)}{\bigoplus} ~\underset{\alpha}{\bigotimes} P(I_\alpha,K'_\alpha) \underset{}{\otimes} ~\underset{\beta}{\bigotimes} Q(K''_\beta, J_\beta)\Big)\Big/_{\mathfrak{S}_n} \ .
	\end{align*}
\end{proof}

\begin{prop}\label{prop::S_permute_prod_connexe_bi}
	Let $P$ and $Q$ be two reduced $\mathfrak{S}$-bimodules. We have the natural isomorphism
	\[
		\mathbb{S}\big( P\boxtimes\Val_c Q\big) \cong \mathbb{S}P\:\square\:\mathbb{S}Q.
	\]
\end{prop}
\begin{proof}
	By \Cref{lem::SPsquareSQ}, for all finite sets $S$ and $E$, we have the isomorphism
	\[
		\big(\mathbb{S}P\:\square\:\mathbb{S}Q\big)(S,E) \cong \underset{n\in\N^*}{\bigoplus}
		\Big(
		\underset{{\substack{ 
			(\{I,K'\},\{K'',J\})\\ \in\mathbb{X}^n(S,E)
		}}}{\bigoplus}
		\underset{\alpha}{\bigotimes} P(I_\alpha,K'_\alpha) \underset{}{\otimes} \underset{\beta}{\bigotimes} Q(K''_\beta,J_\beta)\Big)\Big/_{\mathfrak{S}_n}.
	\]
\end{proof}

\section{Induction functor}\label{sect::induction_functor}
We describe the functor $\Ind$, and its right adjoint, the restriction functor  $\mathrm{Res}$, which establishes the link between the two previous sections, since $ \Ind : \Smod_k \rightarrow \Sbimod_k$ is strong monoidal for the different products introduced in the sections \ref{sect::product_on_Smod} and \ref{sect::bimodule}.

\subsection{Adjunction Ind-Res}
For $\mathcal{C}$ a groupoid, we note $\Delta_{\mathcal{C}}$ the functor
\[
\Delta_{\mathcal{C}}:=(\mathrm{inv}_{\mathcal{C}},\mathrm{id}_{\mathcal{C}}) : \mathcal{C} \longrightarrow \mathcal{C}\op \times \mathcal{C}
\]
where $\mathrm{inv}_{\mathcal{C}} :\mathcal{C} \rightarrow \mathcal{C}\op $ is the equivalence of categories given by the passage to the inverse. The \textit{restriction functor}, denoted by $\mathrm{Res}$, is given by the following composition: 
\[
\begin{array}{rccc}
\mathrm{Res} : & \Sbimod_k & \longrightarrow & \Smod_k\\
& P & \longmapsto & P\circ \Delta_{\Fin\op}
\end{array}~~.
\]
This functor is exact and has a left adjoint functor, the \emph{induction functor}. 
\begin{defi}[Functor $\Ind$]\label{defi::foncteur Ind}
	The \emph{induction functor} is given by:
	\[
		\begin{array}{rccc}
		\Ind :& \Smod_k & \longrightarrow &  \Sbimod_k \\
		& V & \longmapsto & \big(\Ind V \big)(S,E):= \underset{\Hom_\Fin(E,S)}{\bigoplus} V(E)
		\end{array}
	\]
	where an element $f$ in $\mathrm{Aut}(S)$ acts on the left by
	\[
		f\cdot\Big( \underset{\phi\in\Hom_\Fin(E,S)}{\bigoplus} V(E)\Big) = \underset{f\phi\in\Hom_\Fin(E,S)}{\bigoplus} V(E)
	\]
	and an element $g$ in  $\mathrm{Aut}(E)$ acts on the right by
	\[
		\Big( \underset{\phi\in\Hom_\Fin(E,S)}{\bigoplus} V(E)\Big)\cdot g =  \underset{\phi g\in\Hom_\Fin(E,S)}{\bigoplus} V\big(f^{-1}(E)\big).
	\]
\end{defi}
\begin{rem}\label{rem::Description_Ind}
	Let $V$ be a reduced $\mathfrak{S}$-module and $S$ and $E$ be two finite sets. If  $S$ and $E$ are not isomorphic (i.e. $|S|\ne|E|$) then $\big(\Ind V \big)(S,E) =0$. Finally, we have:
	\begin{align*}
		\big(\Ind\;V \big)(S,E)\cong & 
		\left\{ \begin{array}{cl}
		0 & \mbox{ if } S\not\cong E \\ 
		k[\mathrm{Aut}(S)]\otimes V(S)  & \mbox{ otherwise.}
		\end{array} \right.
	\end{align*}
\end{rem}

\begin{prop}\label{prop::adj_ind_res}
	We have the adjunction
	\[
		\Ind : \Smod_k \rightleftarrows \Sbimod_k : \mathrm{Res}.
	\]
\end{prop}
\begin{proof}
	By the classical result \cite[Th.13]{Ser71}.
\end{proof}
One of the fundamental properties of  $\Ind$ is the following.
\begin{prop}\label{prop::Ind_exact}
	The functor $\Ind$ is exact, preserves quasi-isomorphisms and commutes with colimits.
\end{prop}
\begin{proof}
	The result is induced by the remark  \ref{rem::Description_Ind}; more generally, for $G$ a finite group and $H$ a subgroup of $G$, $k[G]$ is a free $k[H]$-module. As the functor $\Ind$ has a right adjoint, it commutes  with colimits.
\end{proof}	

\subsection{Compatibilities with products}
In this subsection, we show compatibilities of the functor  $\Ind$ with products defined in previous sections. First, we recall the following proposition, about compatibility between monoidal structures and adjunction.

\begin{prop}\label{prop::adjonction_lax_monoidal}
	Let $(\C,\otimes,I_\C)$ and  $(\mathsf{D},\odot,I_{\mathsf{D}})$ be two monoidal categories with the following adjunction:
	\[
	\begin{tikzcd}
		L \ : \ \C \arrow[r, harpoon, shift left=-0.2ex, "\perp"', bend left=15]
		& \arrow[l, harpoon,  shift left=-0.5ex, bend left=15] \mathsf{D} \ : \ R \ 
	\end{tikzcd}
	\]
	such that the left adjoint is a strong monoidal functor by the following natural  equivalence $\mu_L : L(-_1)\odot L(-_2) \overset{\cong}{\rightarrow} L(-_1\otimes -_2) $ and the natural isomorphism $
	e_L : I_\mathsf{D} \overset{\cong}{\rightarrow} L(I_\C)$. Then the right adjoint is a Lax monoidal functor with the natural transformation $\mu_R$ and the morphism $\epsilon_R$ given by
	\[
	\xymatrix@C=2cm{
		\ar@{.>}[d]_{\mu_R:=}  R(-_1)\otimes R(-_2) \ar[r]^(0.45){\eta(R(-_1)\otimes R(-_2)} & RL\big( R(-_1)\otimes R(-_2)\big) \ar[d]^{R\big(\mu_L^{-1}(R(-_1),R(-_2))\big)} \\
		R(-_1\otimes -_2) & R\big( LR(-_1)\odot LR(-_2)\big) \ar[l]^(0.55){R(\epsilon(-_1)\odot \epsilon(-_2))}
	}
	\]
	and 
	$
	\xymatrix{
		e_R : I_\C \ar[r]^{\eta(I_\C)} & RL(I_\C) \ar[r]^{R(e_L^{-1})} & R(I_\mathsf{D}) .
	}
	$
\end{prop}
Now, we can study the case of the adjunction given by the functors $\Ind$ and $\Res$.
\begin{nota}
	We note $\Sbimod_k^{\Ind}$, the essential image of the functor $\Ind$.
\end{nota}
\begin{prop}\label{prop_ind_mon_sym_comp}
	 Then $(\Sbimod_k^{\Ind},\:\square\:)$ is a symmetric monoidal category and the induction functor 
	\[
	\Ind:(\Smod_k,\square)\longrightarrow(\Sbimod_k^{\Ind},\square)
	\]
	is symmetric monoidal.
\end{prop}
\begin{proof}
	Let $P$ and $Q$ be two $\mathfrak{S}$-modules and, $S$ and $E$ be two finite sets. We have
	\begin{align*}
		\big(\Ind P \:\square\:\Ind Q\big)(S,E)\cong~&
		\int^{T\in \Fin} \Ind P (S,T)\otimes \Ind Q(T,E)\\
		\cong~&\int^{T\in\Fin} \bigoplus_{\Hom_{\Fin}(T,S)} P(T) \otimes \bigoplus_{\Hom_{\Fin}(E,T)}Q(E) \\
		\cong~&\int^{T\in \Fin}\bigoplus_{\Hom_{\Fin}(T,S)\times\Hom_{\Fin}(E,T)} P(T)\otimes Q(E)~~.
	\end{align*}
	Also, we have the following commutative diagram:
	\[
	\xymatrix{
		\underset{\scriptscriptstyle{\Hom(E,S)}}{\displaystyle{\bigoplus}} P(E)\otimes Q(E) \ar[d]^{T=E} 
		\ar@/^2pc/[dr] \ar@/_3pc/@<-3pc>[dd]_{\mathrm{id}}& \\
		\underset{T}{\displaystyle{\coprod}} \underset{{\substack{\scriptscriptstyle{\Hom(T,S)}\\ \scriptscriptstyle{\times \Hom(E,T)}}}}{\displaystyle{\bigoplus}}
		P(T)\otimes Q(E) \ar[r] \ar[d]^{\phi}&
		\displaystyle{\int}^{T}\underset{{\substack{\scriptscriptstyle{\Hom(T,S)}\\ \scriptscriptstyle{\times \Hom(E,T)}}}}{\bigoplus}
		P(T)\otimes Q(E) \ar@/^2pc/@[..>][ld]^\psi\\
		\underset{\scriptscriptstyle{\Hom(E,S)}}{\displaystyle{\bigoplus}} P(E)\otimes Q(E)& ,
	}
	\]
	where $\phi$ is given by the composition of functions and  $\psi$ exists by the universal property of the coend and the commutativity of the following diagram:
	\[
	\xymatrix{
		\underset{\scriptscriptstyle{T_1\rightarrow T_2}}{\displaystyle{\coprod}}
		\underset{{\substack{
					\scriptscriptstyle{\Hom(T_2,S)}\\
					\scriptscriptstyle{ \times \Hom(E,T_1)
				}}}}{\displaystyle{\bigoplus}} 
		P(T_2) \otimes Q(E) 
		\ar@<2pt>[r]^{\Phi_1}\ar@<-2pt>[r]_{\Phi_2} \ar@/_2pc/[rd]|{p\Phi_1=p\Phi_2}
		& \underset{\scriptscriptstyle{T}}{\displaystyle{\coprod}} \underset{{\substack{\scriptscriptstyle{\Hom(T,S)}\\ \scriptscriptstyle{\times \Hom(E,T})}}}{\displaystyle{\bigoplus}} P(T)\otimes Q(E) \ar[r] \ar[d]^{p} 
		& \displaystyle{\int}^{\scriptscriptstyle{T}}\underset{{\substack{\scriptscriptstyle{\Hom_{\Fin}(T,S)}\\\scriptscriptstyle{ \times\Hom_{\Fin}(E,T)}}}}{\displaystyle{\bigoplus}} P(T)\otimes Q(E) 
		\ar@/^2pc/@{..>}[ld]^\psi
		\\
		& \underset{\scriptscriptstyle{\Hom(E,S)}}{\displaystyle{\bigoplus}} P(E)\otimes Q(E)&
	}~,
	\] 
	which gives us the natural isomorphism
	\[
		\int^{T\in \Fin}\underset{{\substack{\Hom_{\Fin}(T,S)\\ \times\Hom_{\Fin}(E,T)}}}{\bigoplus} P(T)\otimes Q(E) \cong~
		\underset{\Hom(E,S)}{\bigoplus} P(E)\otimes Q(E) \cong~ \Ind\big( P\:\square\:Q\big)(E) \ .
	\]
\end{proof}
\begin{cor}
	Let $M$ be a reduced $\mathfrak{S}$-module. We have the following natural isomorphism of $\mathfrak{S}$-bimodules  $\mathbb{T}_{\square}\big(\Ind(M)\big) \cong\Ind\big(\mathbb{T}_\square(M)\big)$ where $\mathbb{T}_{\square}$ is the free unitary associative monoid for the product $\square$.
\end{cor}
\begin{cor}
	The functor $\Res : \big(\Sbimodred_k, \square\big) \rightarrow \big(\Smodred_k,\square \big)$ is Lax-monoidal.
\end{cor}
\begin{proof}
	By adjunction of functors $\Ind$ and $\Res$, and \Cref{prop::adjonction_lax_monoidal}.
\end{proof}

\begin{rem}
	The functor $\Res$ is not strongly monoidal with respect to $\square$. For example, if we consider $P$ and $Q$, the $\mathfrak{S}$-bimodules defined by
	\[
	P(S,E):= \left\{		
	\begin{array}{cl}
	k &\quad \mbox{ if }|S|=1\mbox{ and }|E|=1 \mbox{ or } 2~~\\
	0 &\quad \mbox{ otherwise}
	\end{array}\right.
	\]
	and
	\[
		Q(S,E):= \left\{		
		\begin{array}{cl}
		k &\quad \mbox{ if }|E|=1\mbox{ and }|S|=1 \mbox{ or } 2~~\\
		0 &\quad \mbox{ otherwise}
		\end{array}\right.
	\]
	then $\Res(P\:\square\:Q)(\{*\},\{*\})=k^2$ and $\big(\Res P\:\square\:\Res Q\big)(\{*\},\{*\})=k$.
\end{rem}
\begin{rem}
	We have the monoidal adjunction
	\[
	\Ind : \big(\Smodred_k,\square\big) \rightleftarrows \big(\Sbimodred_k,\square\big) : \mathrm{Res}.
	\]
\end{rem}
The functor $\Ind$ is also compatible with the concatenation product.
\begin{prop}\label{prop_ind_mon_sym_conc}
	The functor of induction  
	\[
		\Ind:(\Smodred_k,\otimes^{\mathrm{conc}}) \longrightarrow (\Sbimodred_k,\otimes^{\mathrm{conc}})
	\]
	is symmetric monoidal. 
\end{prop}
\begin{proof}
	Let $P$ and $Q$ be two reduced $\mathfrak{S}$-modules and $S$ and $E$ be two finite sets. We have the isomorphisms of chain conplexes:
	\begin{align*}
	\big(\Ind P\otimes^{\mathrm{conc}}\Ind Q\big)(S,E) \cong ~&
	\bigoplus_{(I,J)\in\mathbb{Y}\ord_2(S,E)} \bigoplus_{\Hom_{\Fin}(J_1,I_1)\times\Hom_{\Fin}(J_2,I_2)} P(J_1)\otimes Q(J_2) \\
	\cong~& \bigoplus_{J\in \mathcal{Y}\ord_2(E)} \bigoplus_{\Hom_{\Fin}(J_1\amalg J_2,S)} P(J_1)\otimes Q(J_2)\\
	\cong~& \Ind\big(P\otimes^{\mathrm{conc}} Q\big)(S,E).
	\end{align*}
\end{proof}

\begin{prop}
	The functor $\Res : \big(\Sbimodred_k,\otimes^{\mathrm{conc}}\big) \rightarrow \big(\Smodred_k,\otimes^{\mathrm{conc}}\big)$ is Lax-monoidal.
\end{prop}
\begin{proof}
	Let $P$ and  $Q$ be two $\mathfrak{S}$-bimodules. We have the natural injections
	\begin{align*}
	\underset{I\in\mathcal{Y}\ord_2(-)}{\coprod}& \Hom_{\Fin\op}(I_1\amalg I_2,-)\otimes \Res P(I_1)\otimes \Res Q(I_2)\\
	~=~& \underset{I\in\mathcal{Y}\ord_2(-)}{\coprod} \Hom_{\Fin\op}(I_1\amalg I_2,-)\otimes P(I_1,I_1)\otimes Q(I_2,I_2) ~~\\
	\hookrightarrow~&\underset{I\in\mathcal{Y}\ord_2(-)}{\coprod} \Hom_{\Fin\times\Fin\op}\big((I_1\amalg I_2,I_1\amalg I_2),\Delta(-)\big)\otimes P(I_1,I_1)\otimes Q(I_2,I_2) ~~\\
	\hookrightarrow~&\underset{(I,J)\in\mathbb{Y}\ord_2(-)}{\coprod} \Hom_{Fin\times\Fin\op}\big((I_1\amalg I_2,J_1\amalg J_2),\Delta(-)\big)\otimes P(I_1,J_1)\otimes Q(I_2,J_2) 
	\end{align*}
	which imply the following natural injection
	\[
	\coprod_{I\in\mathcal{Y}\ord_2(-)} \Hom_{\Fin\op}(I_1\amalg I_2,-)\otimes \Res P(I_1)\otimes \Res Q(I_2)
	\hookrightarrow \Res\big(P\otimes^{\mathrm{conc}} Q\big).
	\]
	Also, we have the following commutative diagram
	\[
	\xymatrix{
		\underset{I_1\rightarrow J_1,  I_2\rightarrow J_2}{\coprod} \Hom_{\Fin\op}(J_1\amalg J_2,-)\otimes \Res P(I_1)\otimes \Res Q(I_2) \ar@<2pt>[d] \ar@<-2pt>[d] & \\
		\underset{I\in\mathcal{Y}\ord_2(-)}{\coprod} \Hom_{\Fin\op}(I_1\amalg I_2,-)\otimes \Res P(I_1)\otimes \Res Q(I_2) 		\ar[d] \ar@{^{(}->}[r]
		& \Res (P\otimes^{\mathrm{conc}} Q) \\
		\big(\Res P\big)\otimes\big(\Res Q\big) \ar@{..>}@/_1pc/[ur]_\exists  & 
	}
	\]
	so, by the universal property of the coend, we have the natural morphism
	\[
	\Res P \otimes^{\mathrm{conc}} \Res Q \longrightarrow \Res(P\otimes^{\mathrm{conc}} Q).
	\]
\end{proof}
\begin{rem}
	The functor $\Res$ is not strongly monoidal with respect to $\otimes^{\mathrm{conc}}$. Indeed, if we consider $P$ and $Q$, the $\mathfrak{S}$-bimodules defined by
	\[
		P(S,E):=\left\{
		\begin{array}{cl}
		k & \mbox{ if } |E|=2 \mbox{ and } |S|=1;\\
		0 & \mbox{ otherwise }
		\end{array}\right.
		\]
		and
		\[
		Q(S,E):=\left\{
		\begin{array}{cl}
		k & \mbox{ if } |S|=2 \mbox{ and } |E|=1;\\
		0 & \mbox{ otherwise }
		\end{array}\right.~,
	\]
	then $\Res P\otimes^{\mathrm{conc}} \Res Q =0$, while $\Res(P\otimes^{\mathrm{conc}} Q)(S)=k$ if the cardinal of $S$ is $3$.
\end{rem}
We have the following compatibility between functors
$\mathbb{S}(-)$ and $\Ind(-)$.
\begin{prop}\label{prop::ind_commute_S}
	Let $P$ be a reduced $\mathfrak{S}$-module. Then, we have the natural isomorphism of reduced $\mathfrak{S}$-bimodules:
	\[
	\Ind\big(\mathbb{S}(P)\big)\cong \mathbb{S}\big( \Ind(P) \big).
	\]
\end{prop}
\begin{proof}
	The functor $\Ind$ commutes with the concatenation product $\otimes^{\mathrm{conc}}$ and is compatible with the symmetry, by \Cref{prop_ind_mon_sym_conc}. We conclude by the exactness of the functor $\Ind$.
\end{proof}
One of the most important properties of the functor $\Ind$ is  that it is compatible with connected composition products. 

\begin{thm}\label{thm::Ind_monoidal_connexe}
	The functor 
	\[
	\Ind : \big(\Smodred_k,\boxtimes_c\big)  \rightarrow   \big(\Sbimodred_k,\boxtimes_c\Val\big)
	\]
	is monoidal.
\end{thm}
\begin{proof}
	Let $S$ and $E$ be two finite sets, then $(\Ind \,\Ibox)(S,E) = k$ if $|S|=1=|E|$ and $0$ otherwise. Then, the functor $\Ind$ respect the unit. Let $V$ and $W$ be two reduced $\mathfrak{S}$-modules and $S$ and $E$ be two finite sets.
	\begin{align*}
	\Big( \Ind V&\boxtimes_c\Val\Ind W \Big)(S,E) \\
	=~&\underset{n\in\N^*}{\bigoplus}~ \underset{{\substack{(\{I,K'\},\{K'',J\})\\ \in\mathbb{X}^{n,\mathrm{conn}}(S,E)}}}{\bigoplus} ~\underset{\alpha}{\bigotimes}~\Ind V(I_\alpha,K'_\alpha) \underset{\mathfrak{S}_n}{\otimes} ~\underset{\beta}{\bigotimes}~ \Ind W(K''_\beta, J_\beta) ~~\\
	\cong~& \underset{n\in\N^*}{\bigoplus}~ \underset{{\substack{(\{I,K'\},\{K'',J\})\\ \in\mathbb{X}^{n,\mathrm{conn}}(S,E)}}}{\bigoplus} ~\underset{\alpha}{\bigotimes}~\underset{\Hom_\Fin(K'_\alpha,I_\alpha)}{\bigoplus} V(K'_\alpha)\underset{\mathfrak{S}_n}{\otimes} ~\underset{\beta}{\bigotimes}~ \underset{\Hom_\Fin(J_\beta,K''_\beta)}{\bigoplus}W(J_\beta). ~
	\end{align*}
	Note that the right side is non zero if and only if  $I_\alpha\cong K'_\alpha$ for all $\alpha$ and $K''_\beta \cong J_\beta$ for all $\beta$. This implies that $(\Ind V\boxtimes_c\Ind W )(S,E)=0$ if $|S|\ne|E|$.
	\begin{align*}			
	\Big( \Ind V&\boxtimes_c\Val\Ind W \Big)(S,E) 	\\
	\cong  & \underset{n\in\N^*}{\bigoplus}~ \underset{{\substack{(\{I,K'\},\{K'',J\}) \\ \in\mathbb{X}^{n,\mathrm{conn}}(S,E)}}}{\bigoplus} ~
	\underset{{\substack{\prod_\alpha\Hom_\Fin(K'_\alpha,I_\alpha) \\ \times\prod_\beta\Hom_\Fin(J_\beta,K''_\beta)}}}{\bigoplus}~\underset{\alpha}{\bigotimes}~ V(K'_\alpha)\underset{\mathfrak{S}_n}{\otimes} ~\underset{\beta}{\bigotimes}~ W(J_\beta)~~\\
	\cong~& \underset{r,s,n\in\N^*}{\bigoplus}~ \underset{{\substack{I\in Y_r(S), J\in Y_s(E)\\ (K',K'')\in \mathcal{X}^{\mathrm{conn}}_{r,s}([\![1,n]\!]) \\ I_\alpha\cong K'_\alpha, K''_\beta\cong J_\beta}}}{\bigoplus} 
	~\underset{{\substack{
		\Hom_\Fin([\![1,n]\!],S) \\ \quad \underset{\mathfrak{S}_n}{\times}\Hom_\Fin(E,[\![1,n]\!])
	}}}{\bigoplus}\underset{\alpha}{\bigotimes}~ V(K'_\alpha)\underset{\mathfrak{S}_n}{\otimes} ~\underset{\beta}{\bigotimes}~ W(J_\beta)~~\\
	\cong~& \underset{r,s,n\in\N^*}{\bigoplus}~ \underset{{\substack{I\in Y_r(S), J\in Y_s(E)\\ (K',K'')\in \mathcal{X}^{\mathrm{conn}}_{r,s}([\![1,n]\!]) \\ I_\alpha\cong K'_\alpha, K''_\beta\cong J_\beta}}}{\bigoplus} ~\underset{\Hom_\Fin(E,S)}{\bigoplus}~\underset{\alpha}{\bigotimes}~ V(K'_\alpha)\underset{\mathfrak{S}_n}{\otimes} ~\underset{\beta}{\bigotimes}~ W(J_\beta)~~\\
	\cong~& \underset{\Hom_\Fin(E,S)}{\bigoplus}~\underset{r,s,n\in\N^*}{\bigoplus}~ \underset{{\substack{I\in Y_r(S), J\in Y_s(E)\\ (K',K'')\in \mathcal{X}^{\mathrm{conn}}_{r,s}([\![1,n]\!]) \\ I_\alpha\cong K'_\alpha, K''_\beta\cong J_\beta}}}{\bigoplus} ~\underset{\alpha}{\bigotimes}~ V(K'_\alpha)\underset{\mathfrak{S}_n}{\otimes} ~\underset{\beta}{\bigotimes}~ W(J_\beta)~~\\
	\cong~& \underset{\Hom_\Fin(E,S)}{\bigoplus}~\underset{r,s\in\N^*}{\bigoplus}~ \underset{(I,J)\in \mathcal{X}^{\mathrm{conn}}_{r,s}(E)}{\bigoplus} ~\underset{\alpha}{\bigotimes}~ V(I_\alpha)\underset{\mathfrak{S}_n}{\otimes} ~\underset{\beta}{\bigotimes}~ W(J_\beta)~~\\
	\cong~& \Ind\:\big( V \boxtimes_c W\big)(S,E).
	\end{align*}
\end{proof}
\begin{cor}
	The functor
	$
	\Res : \big(\Sbimodred_k,\boxtimes_c\Val\big)\longrightarrow\big(\Smodred_k,\boxtimes_c\big)
	$ 
	is Lax-monoidal.
\end{cor}	
\begin{proof}
	By the \Cref{prop::adjonction_lax_monoidal}.
\end{proof}

\section{Protoperads}\label{sect::protoperads}
In this section, we study monoids in the monoidal category $(\Smodred_k,\boxtimes_c,\Ibox)$.
\begin{defi}[Protoperad]\label{def::protoperade}
	A \textit{protoperad} is an unital monoid $(P,\mu,\eta)$ in the monoidal category $(\Smodred_k,\boxtimes_c)$: we note by 
	\[
		\Proto _k:=\UAs(\Smodred_k,\boxtimes_c,\Ibox),
	\] 
	the category of protoperads.
\end{defi}	

\begin{prop}
	The functor $\Ind$ induces the functor
	\[
	\Ind : \Proto _k \longrightarrow \Prope _k.
	\]
\end{prop}
\begin{proof}
	By the  \Cref{thm::Ind_monoidal_connexe}.
\end{proof}

\begin{rem}
	There exists the notion of \emph{prop} which is more general than the notion of properad: a prop is an object of the category
	\[
	\mathsf{props}_k:=\UAs\big(\Com(\Sbimod_k,\otimes^{\mathrm{conc}}),\square,I_\square \big),
	\]
	i.e. a $\mathfrak{S}$-bimodule with two products, a horizontal and a vertical ones, which satisfy the \emph{interchanging law} (see \cite{Mar08}). A  natural question is the following: what structure puts on a $\mathfrak{S}$-module $P$ such that $\Ind(P)$ is a prop? As the functor $\Ind$ is monoidal for $\square$ and $\otimes^{\mathrm{conc}}$, it induces the functor
	 \[
	 \Ind:\mathsf{protops}_k:=\UAs\big(\Com(\Smod,\otimes^{\mathrm{conc}}),\square,I_\square\big) 
	 \longrightarrow
	 \mathsf{props}_k.
	 \]
\end{rem}

We also have the dual notion.
\begin{defi}[Coprotoperad]\label{def::coprotoperade}
	A \emph{coprotoperad} is a co-unital comonoid  $(Q,\Delta,\epsilon)$ in the monoidal category  $(\Smodred_k,\boxtimes_c,\Ibox)$: we note  $\mathsf{coprotoperads}_k$, the category $\mathrm{co}\UAs(\Smodred_k,\boxtimes_c,\Ibox)$ of coprotoperads.
\end{defi}


\begin{nota}
	We note $\Smod^{\mathrm{gr}}_k$, the category $\Func(\Fin\op,\Ch^{\mathrm{gr}}_k)$, where $\Ch^{\mathrm{gr}}_k$ is the category of chain complexes with an 
	$\N$-grading called the \emph{weight}. All the previous constructions extend naturally to graded $\mathfrak{S}$-modules. Remark that this grading does not imply Koszul signs: the symmetry of the monoidal category $\Ch_k^{\mathrm{gr}}$ is given, for $C$ and $D$, two weight graded chain complexes,  by
	\[
	\begin{array}{cccc}
		\tau_{C,D} : & C\otimes  D& \longrightarrow & D\otimes C \\
		& c \otimes d & \longmapsto & (-1)^{|c||d|}d\otimes c
	\end{array} \ ,
	\]
	where, for $c$, an homogeneous element of $C$, $|c|$ is its homological degree.
\end{nota}
\begin{defi}[(Connected) Weight graded protoperad/coprotoperad]
	A protoperad (resp. coprotoperad) $\mathcal{P}$ is  \emph{weight graded}  if $\mathcal{P}$ is a  monoid (resp. comonoid) in the category $\Smod_k^{\mathrm{red,gr}}$ for the product $\boxtimes_c$. We denote this grading by $\mathcal{P}=\bigoplus_{i\in \N} \mathcal{P}^{[i]}$. We say that a weight graded (co)protoperad $\mathcal{P}=\bigoplus_{i\in \N} \mathcal{P}^{[i]}$ is \emph{connected} if $\mathcal{P}^{[0]}\cong \Ibox$.
\end{defi}

\subsection{Partial compositions}\label{sect::composition partielle}
One can describe (cf. \cite[Sect. 5.3.4]{LV12}) the operad structure on a $\mathfrak{S}$-module $\mathcal{O}$ just by giving the partial compositions maps $\circ_s : \mathcal{O}(S)\otimes \mathcal{O}(R) \rightarrow \mathcal{O}(R\amalg S\backslash\{s\})$. We have a similar property for protoperads
\begin{defi}[Partial compositions]
	Let $P$ be a reduced $\mathfrak{S}$-module equipped with a morphism of $\mathfrak{S}$-modules $\epsilon : \Ibox \hookrightarrow P$. Let $M,N$ and $S$ be three non-empty finite sets, with two diagrams of injections as follows:
	\[
		\phi:= \ \big(i:M\hookrightarrow S \hookleftarrow N :j\big) \ \mbox{and} \ \phi\op:= \ \big(j:N\hookrightarrow S \hookleftarrow M :i\big)
	\]  
	and such that
	\begin{equation}\label{eq::recouvrement_proto}
	\left\{
	\begin{array}{rl}
	\mathrm{im}(i)\cup \mathrm{im}(j) & = S \\
	\mathrm{im}(i)\cap\mathrm{im}(j) & \ne \varnothing
	\end{array}
	\right.~~.
	\end{equation}
	We say that $P$ has a \emph{partial composition system} if, for all diagrams $\phi$, we have a morphism of chain complexes
	\[
	\underset{\phi}{\circ} : P(M) \otimes P(N) \longrightarrow P(S)\ ,
	\]
	compatible with the action of $\mathrm{Aut}(S)$, i.e. for all $\sigma\in \mathrm{Aut}(S)$ with
	\[
	\sigma\cdot \phi := \Big( i':M':=\sigma\big(i(M)\big)\hookrightarrow S \hookleftarrow \sigma\big(j(N)\big)=:N' : j'\Big)
	\]
	we have the commutative diagram 
	\begin{equation}\label{eq::compa_action_AutS}
	\xymatrix@C=3cm@R=0.5cm{
		P(M) \otimes P(N) \ar[d]_{\underset{\phi}{\circ}} \ar[r]_{\cong}^{P(\sigma|_M)\otimes P(\sigma|_N)} & P(M') \otimes P(N') \ar[d]^{\underset{\sigma\cdot \phi}{\circ}} \\
		P(S) \ar[r]_{P(\sigma)} & P(S)~;
	}
	\end{equation}
	and which satisfies the following compatibility properties, for all commutative diagram of injections 
	\[
	\xymatrix@C=0.7cm@R=0.3cm{
		M \ar[dr]  & \ar@{}[d]|\phi & N \ar[dl] \ar[dr]  &\ar@{}[d]|\psi & U \ar[dl] \\
		& R \ar[dr] & \ar@{}[d]|\xi &  S \ar[dl]& \\
		& & T & &
	}
	\]
	with $\xi_L:= M \rightarrow T \leftarrow S$ and $\xi_R:= R \rightarrow T \leftarrow U$, such that the four pairs of arrows $\phi,\psi,\xi_L$ and $\xi_R$ satisfy the \Cref{eq::recouvrement_proto}. The partial composition satisfies the three  associativity axioms:
	\begin{description}
		\item[Axiom $\mathrm{H}$]~
		$\left(
		\begin{tikzpicture}[scale=0.2,baseline=1.5ex]
		\draw (0,1.5) rectangle (2,2);
		\draw (1.25,2.25) rectangle (3.25,2.75);
		\draw (-1,0.75) rectangle (1,1.25);
		\end{tikzpicture}\right)
		$	
		\[
		\xymatrix@C=2cm@R=0.5cm{
			P(M)\otimes P(N) \otimes P(U) \ar[r]^{1\otimes \underset{\psi}{\circ}}\ar[d]_{\underset{\phi}{\circ}\otimes 1} & P(M)\otimes P(S) \ar[d]^{\underset{\xi_L}{\circ}} \\
			P(R) \otimes P(U) \ar[r]_{\underset{\xi_R}{\circ}} & P(T)~;
		} 
		\] 
		\item[Axiom $\mathrm{V}$]~
		$\left(
		\begin{tikzpicture}[scale=0.2,baseline=1.2ex]
		\draw (0,1.5) rectangle (2,2);
		\draw (2.25,1.5) rectangle (4.25,2);
		\draw (1.25,0.75) rectangle (3.25,1.25);
		\end{tikzpicture}\right)
		$
		\[
		\xymatrix@C=2cm@R=0.5cm{
			P(N)\otimes P(M) \otimes P(U) \ar[r]^{(\underset{\psi}{\circ}\otimes 1)(1\otimes \tau)}\ar[d]_{\underset{\phi\op}{\circ}\otimes 1} & P(S)\otimes P(M) \ar[d]^{\underset{\xi_L\op}{\circ}} \\
			P(R) \otimes P(U) \ar[r]_{\underset{\xi_R}{\circ}} & P(T)~;
		} 
		\] 
		\item[Axiom $\mathrm{\Lambda}$]~
		$\left(
		\begin{tikzpicture}[scale=0.2,baseline=1.2ex]
		\draw (0,1.5) rectangle (2,2);
		\draw (1.25,0.75) rectangle (3.25,1.25);
		\draw (-1,0.75) rectangle (1,1.25);
		\end{tikzpicture}\right)
		$	
		\[		
		\xymatrix@C=2cm@R=0.5cm{
			P(M)\otimes P(U) \otimes P(N) \ar[r]^{1\otimes \underset{\psi\op}{\circ}}\ar[d]_{(1\otimes \underset{\phi}{\circ})(\tau\otimes 1)} & P(M)\otimes P(S) \ar[d]^{\underset{\xi_L}{\circ}} \\
			P(U) \otimes P(R) \ar[r]_{\underset{\xi_R\op}{\circ}} & P(T).
		} ~~.
		\]	
		The partial products also satisfy the following unital property for all diagrams of the form
		\[
		\iota := \big(i:\{*\} \hookrightarrow M \overset{\cong}{\leftarrow} M:\mathrm{id}\big),
		\]
		we have commutative diagrams
		\[
		\begin{tikzcd}
			\Ibox(\{*\}) \otimes P(M) \ar[rd, bend right=15, "\cong"'] \ar[r, hookrightarrow, "\epsilon \otimes \mathrm{id}"] & P(\{*\})\otimes P(M) \ar[d, "\underset{\iota}{\circ}"]  \\
			& P(M)
		\end{tikzcd}
		\]
		and
		\[
		\begin{tikzcd}
			P(M)\otimes \Ibox(\{*\}) \ar[rd, bend right=15, "\cong"']  \ar[r, hookrightarrow, "\mathrm{id}\otimes\epsilon"] & P(M)\otimes P(\{*\}) \ar[d, "\underset{\iota\op}{\circ}"]  \\
			& P(M)
		\end{tikzcd} \ .
		\]
	\end{description}
\end{defi}
\begin{prop}\label{prop::proto_def_partielle}
	A protoperad $\calP$ has canonically a partial compositions system. Conversely, a partial compositions system on a $\mathfrak{S}$-module $P$  canonically extends to a protoperad structure.
\end{prop}
\begin{proof}
	Let $(P,\mu)$ be a monoid in the category  $\big(\Smodred_k,\boxtimes_c,\Ibox\big)$. By the grading of $\boxtimes_c$  which is implied by the analycity of $\boxtimes_c$ (cf. \Cref{lem::Smod_ana_scinde}), the restriction of the product $\mu$ to $(P\boxtimes_c P)^{(2)_P}$ gives us directly all the partial compositions $\underset{\phi}{\circ}$ and the associativity and the unit of the product imply all diagrams of the definition hold. 
	
	Conversely, let $P$ be a reduced $\mathfrak{S}$-module, with an injection of $\mathfrak{S}$-modules $\Ibox\hookrightarrow P$ and a partial composition system. By the associativity of partial compositions $\underset{\phi}{\circ}$ for $P$, we define, for all $K\in \mathcal{Y}\ord_m(S)$, $L\in \mathcal{Y}\ord_n(S)$ with $\mathcal{K}(K,L)=\{S\}$, a morphism
	\[
	\tilde{\mu}_{K,L} : \bigotimes_{i=1}^m P(K_i) \otimes \bigotimes_{j=1}^n P(L_j) \longrightarrow P(S).
	\]
	The compatibility of the partial compositions with the action of the automorphism group of the target (cf. \Cref{eq::compa_action_AutS}) implies that the following morphism
	\[	
	\sum \tilde{\mu}_{K,L} : \bigoplus_{(\sigma,\tau)\in\mathfrak{S}_m\times\mathfrak{S}_n}\bigotimes_{i=1}^m P(K_{\sigma(i)}) \otimes \bigotimes_{j=1}^n P(L_{\tau(j)}) \longrightarrow P(S)
	\]
	passes to the quotient
	\[
	\mu_S : \bigoplus_{(K,L)\in\mathcal{X}^{\mathrm{conn}}(S)}\bigotimes_{\alpha\in A} P(K_\alpha) \otimes \bigotimes_{\beta\in B} P(L_\beta) \longrightarrow P(S),		
	\]
	which gives us a natural transformation $\mu : P\boxtimes_c P \longrightarrow P$. This natural transformation makes $P$ a unital associative monoid in $\boxtimes_c$, because the partial products satisfy the associativity and unital axioms.
\end{proof}

Vallette's works\cite{Val09}  gives us the construction of the free (co)monoid in a abelian monoidal category as $(\Smodred_k,\boxtimes_c,\Ibox)$. We briefly review this construction in the next section.

\subsection{Free monoid in abelian monoidal categories}\label{subsect::monoide_libre}
We briefly recall the construction of the free monoid $\scrF(-)$ by Vallette in \cite{Val03} for general abelian monoidal category and \cite{Val09}. 

\indent Let $(\mathsf{A},\odot,I)$ be an abelian monoidal category such that, for all objects $A$ in  $\mathsf{A}$, the endofunctors of $\mathsf{A}$ $R_A$ and $L_A$, given by $R_A(M):= M\odot A$ and $L_A(M)=A\odot M$ for all object $M\in \mathsf{A}$,  preserve reflexive coequalizors and sequential colimits (cf. \cite{Val09}). Fix  an object $V$ in the category $\mathsf{A}$. We consider the augmented object $V_+:= I\oplus V$ and we denote by  $\eta_V : I \hookrightarrow V_+$, the injection of $I$ in $V_+$; $\epsilon_V : V_+ \twoheadrightarrow I$, the projection of $V_+$ to $I$; $i_V : V \rightarrow V_+$, the injection of $V$ into $V_+$: 
\[
\begin{tikzcd}
	I \ar[r, hook, "\eta_V"] \ar[rd, bend right=15, "="'] & I\oplus V  \ar[d, two heads, "\epsilon_V"] & V \ar[l, hook', "\iota_V"'] \\
	& I &
\end{tikzcd} \ .
\]
For all natural numbers $n$, we denote $V_n:= (V_+)^{\odot n}$, and $\lambda_A : I\odot A \overset{\cong}{\rightarrow} A$ and $\rho_A : A \odot I \overset{\cong}{\rightarrow} A$, the structure isomorphisms of the monoidal category $(\mathsf{A},\odot,I)$. The injection $\eta_V$ induces degeneracy morphisms $\eta_{V,i} : V_n \rightarrow V_{n+1}$: for all integers $1\leqslant i\leqslant n$, we have
\[
\xymatrix@C=4pc{
	\eta_{V,i} : (V_+)^{\odot i}\odot I \odot (V_+)^{\odot (n-i)} \ar[r]^{V_i\odot \eta_V \odot V_{n-i}} & (V_+)^{\odot i}\odot (V_+) \odot (V_+)^{\odot (n-i)}
}.
\]
We define the morphism $\tau_V : V \rightarrow V_2$ as the following composition:  
\[
\xymatrix@C=5pc{
	V \ar[r]_(0.35){\lambda_V^{-1}+ \rho_V^{-1}} \ar@{.>}@/^1.5pc/[rr]^{=:\tau_V} & (I \odot V) \oplus (V\odot I) \ar[r]_(0.45){\eta_V\odot i_V - i_V\odot \eta_V} &  (I\oplus V)\odot(I\oplus V)=:V_2
}.
\]
For two objects $A$ and $B$ in $\mathsf{A}$, we define the object $ A \odot( \underline{V}\oplus V_2) \odot B $ as the cokernel of 
\[
\xymatrix@C=4pc{
	A \odot V_2 \odot B \ar[r]^(0.4){A\odot i_{V_2} \odot B}& A \odot ( V \oplus V_2) \odot B
},
\]
where $i_{V_2}$ is the canonical injection $V_2 \hookrightarrow V\oplus V_2$. By  \cite[Cor. 4]{Val03}, $ A \odot( \underline{V}\oplus V_2) \odot B $ is also the kernel of 
\[
\xymatrix@C=4pc{
	A \odot ( V \oplus V_2) \odot B\ar[r]^(0.55){A\odot \pi_{V_2} \odot B}&  A \odot V_2 \odot B, 
}
\]
where $\pi_{V_2}$ is the projection $V\oplus V_2 \rightarrow V_2$. So, the object $ A \odot( \underline{V}\oplus V_2) \odot B $ can be considered as a subobject of $ A \odot( V\oplus V_2) \odot B $. We define: 
\begin{itemize}
	\item $R_{A,B}$ as the image of the composition 
	\[
	\xymatrix@C=3pc{
		A \odot( \underline{V}\oplus V_2) \odot B \ar@{^{(}->}[r] & A \odot( V\oplus V_2) \odot B \ar[rr]^{A\odot(\tau+\mathrm{id}_{V_2})\odot B} & & A \odot V_2 \odot B
	};
	\]
	\item  $\widetilde{V}_n$ as the cokernel of
	\begin{equation}\label{eq::def_VnTilde}
	\bigoplus_{i=0}^{n-2} R_{V_i, V_{n-i-2}} \longrightarrow V_n.
	\end{equation} 
\end{itemize}
By \cite[Lem. 4]{Val09}, the morphisms $\eta_{V,i}: V_n \rightarrow V_{n+1}$, for $i$ in $[\![1,n+1]\!]$, induce the same morphism on the quotient: 
\[
	\widetilde{\eta}_V : \widetilde{V}_n \longrightarrow \widetilde{V}_{n+1}.
\]
\begin{defi}[The object {$\scrF(V)$} -- \textup{\cite[Sect. 3]{Val09}}]
	Let $V$  be an element of the category $\mathsf{A}$. The object $\scrF(V)$ associated to $V$, is the sequential colimit:
	\[
	\xymatrix@C=0.7cm{
		\widetilde{V}_0:= I \ar[r]^{\widetilde{\eta}_V} \ar@/_1pc/[rrrd]|(0.4){j_{V,0}}
		& \widetilde{V}_1 = V_1 = V_+ \ar[r]^{\widetilde{\eta}_V}\ar@/_1pc/[rrd]|(0.35){j_{V,1}}
		&  \widetilde{V}_2 \ar[r]^{\widetilde{\eta}_V}\ar@/_1pc/[rd]|(0.3){j_{V,2}}
		& ~~\ldots~~ \ar[r]^{\widetilde{\eta}_V} 
		& \widetilde{V}_n \ar[r]^{\widetilde{\eta}_V}  \ar@/^1pc/[ld]|(0.3){j_{V,n}}
		& ~\ldots~ \\
		& & & \scrF(V):= \underset{n\in \N}{\mathrm{Colim}} \widetilde{V}_n .
	}
	\]
\end{defi}
\begin{prop}[see \cite{Val09}]
	Let $V$ be an object of the category $\mathsf{A}$. The free monoid (resp. cofree comonoid) on $V$ is $(\scrF(V),\mu)$ (resp. $(\scrF^c(V),\Delta)$).
\end{prop}

\begin{prop}\label{prop::foncteur_free_monoide}
	Let $(\C,\odot_\C,I_\C)$ and $(\D,\odot_\D,I_\D)$ be two abelian monoidal categories which admit sequential colimits and such that the monoidal product preserves sequential colimits and reflexive coequalizors. Let $G: \C \rightarrow \D$ be a monoidal functor which commutes with colimits. Then $G$ commutes with the free monoids, i.e. we have the natural equivalence
	\[
		G\big( \scrF_\C(-)\big) \cong \scrF_D\big(G(-)\big).
	\]
\end{prop}
\begin{proof}
	Let $V$ be an object of $\C$. By hypothesis, the functor $G$ commutes with colimits, in particular, with coproducts:  so we have
	\[
		G(V_+) \cong G(V)\oplus G(I_\C) \cong G(V)\oplus I_\D = G(V)_+
	\]
	and, as $G$ is monoidal, we have that, for all integers $n$, $G(V_n)\cong G(V)_n$. 	Since $G$ preserves cokernels, for all objects $A$ and $B$ in the  category $\mathsf{A}$, we have
	\[
		G(R_{A,B}) \cong R_{G(A),G(B)}
	\]
	and, for all integers $n$, $G(\widetilde{V}_n)\cong \widetilde{G(V)}_n$. Finally, as $G$ commutes with colimits, we have 
	\[
		G\big( \scrF_\C(V)\big) \cong \scrF_D\big(G(V)\big).
	\]
\end{proof}

\subsection{A first description of the free protoperad}
\label{subsec::protopérade_libre}
We use the results of the previous section to describe the free protoperad $\scrF(V)$ over a $\mathfrak{S}$-module $V$. This first is not very explicit, in the sense that  the underlying combinatorics is hidden by the formalism.  \Cref{rem::description_F2} and \Cref{sect::functors of walls} will remedy this problem.

To a partition $K$ of a finite set $S$, i.e. $K$ is an element of $\mathcal{Y}(S)$ (see \Cref{rem::sous_foncteur_combinatoire}), we associate the non-ordered set $\Gamma(K)$ which labelled this partition: $S\cong \coprod_{\alpha\in \Gamma(K)} K_\alpha$. Recall that the functor $\mathbb{S}(-)$ is non-unitary and satisfies the exponential property (cf. \Cref{prop::prop_exp_S}): then, for two $\mathfrak{S}$-modules $V_1$ and $V_2$, we have the following isomorphism of $\mathfrak{S}$-modules:
\[
	\mathbb{S}\big( V_1\oplus V_2\big) \cong \mathbb{S}(V_1) \oplus \mathbb{S}(V_2)\oplus \mathbb{S}(V_1) \otimes^{\mathrm{conc}}\mathbb{S}(V_2).
\]
\begin{nota}
	Let $V$ be a $\mathfrak{S}$-module,  we denote by the exponent $(-)_V$, the weight-grading by the number of terms $V$.
\end{nota}
The functor $\mathbb{S}(-)$ is split analytic (cf. \cite{Val03,Val09} for the definition of split analytic functor), so that, for three $\mathfrak{S}$-modules $V_1$, $V_2$ and $V_3$, we have the weight-bigrading:
\begin{align*}
\mathbb{S}(V_1\oplus V_2)\:\square\:\mathbb{S}(V_3) \cong &~\mathbb{S}V_1\:\square\:\mathbb{S}V_3 \oplus \mathbb{S}V_2\:\square\:\mathbb{S}V_3 \oplus \big(\mathbb{S}V_1\otimes^{\mathrm{conc}}\mathbb{S}V_2\big)\:\square\:\mathbb{S}V_3 \\
\cong &~\underset{i,j\in \N^*}{\bigoplus} \mathbb{S}^jV_1\:\square\:\mathbb{S}V_3 \oplus \mathbb{S}^jV_2\:\square\:\mathbb{S}V_3 \oplus \big(\mathbb{S}^iV_1\otimes^{\mathrm{conc}}\mathbb{S}^jV_2\big)\:\square\:\mathbb{S}V_3 \\
\overset{not}{=:} &~\underset{i,j\in \N^*}{\bigoplus}\big(\mathbb{S}(V_1\oplus V_2)\:\square\:\mathbb{S}(V_3)\big)^{(i)_{V_1},(j)_{V_2}}
\end{align*}
by the bi-additivity of the bifunctors $-\square-$ and $- \otimes^{\mathrm{conc}} -$. 

This bigrading induces, via the injection  $(V_1\oplus V_2)\boxtimes_c V_3 \hookrightarrow \mathbb{S}(V_1\oplus V_2)\:\square\:\mathbb{S}(V_3)$ (cf. \Cref{prop::S_permute_prod_connexe}), the bigrading by weight of $V_1$ and $V_2$ on $(V_1\oplus V_2)\boxtimes_c V_3$ which is denoted by
\[
	(V_1\oplus V_2)\boxtimes_c V_3 =: \underset{i,j\in\N^*}{\bigoplus} \big((V_1\oplus V_2)\boxtimes_c V_3\big)^{(i)_{V_1},(j)_{V_2}}.
\]
By the symmetry of the product $\boxtimes_c$, we also have  the bigrading
\[
	V_3\boxtimes_c(V_1\oplus V_2)=: \underset{i,j\in\N^*}{\bigoplus} \big(V_3\boxtimes_c(V_1\oplus V_2)\big)^{(i)_{V_1},(j)_{V_2}},
\]
and we denote $	\big((V_1\oplus V_2)\boxtimes_c V_3\big)^{(j)_{V_2}} := \underset{i\in\N^*}{\bigoplus} \big((V_1\oplus V_2)\boxtimes_c V_3\big)^{(i)_{V_1},(j)_{V_2}}$.
\begin{rem}
	These gradings are natural: they arise from the split analytic property of the bifunctor $-\boxtimes_c -$.
\end{rem}
\begin{rem}
	In \cite{MV09I} and \cite{Val03}, for two $\mathfrak{S}$-bimodules $\mathcal{M}$ and $\mathcal{P}$, the weight grading on $\mathcal{M}$ is denoted by
	\[
	\underbrace{\mathcal{M}}_{r} \boxtimes_c \mathcal{P}.
	\]
	As the functor $\Ind$ is monoidal and commutes with the direct sum, we have
	\[
	\Ind\big(((V_1\oplus V_2) \boxtimes_c V_3)^{(r)_{V_2}} \big) = (\Ind(V_1)\oplus \underbrace{\Ind(V_2)}_{r}) \boxtimes_c\Ind(V_3).
	\]
\end{rem}
As we have seen in \Cref{subsect::monoide_libre}, the construction of  the free monoid generated by a $\mathfrak{S}$-module $V$ is based on the formal addition of the unit $\Ibox$ to $V$. We consider the $\mathfrak{S}$-module $V_+=V\oplus \Ibox$, so that $V_+(S)=V(S)$ for $|S|\ne 1$ and $V_+(\{*\})=V(\{*\})\oplus k$. We also need the weight bigrading of the $\mathfrak{S}$-module $V_+ \boxtimes_c W_+$ given by the weight grading on $V$ and the weight grading on $W$. For all finite sets $S$, this bigrading allows us to write the product $V_+\boxtimes_cW_+(S)$ as a direct sum of terms with $i$ copies of $V$ and $j$ copies of $W$. 
More precisely, the $\mathfrak{S}$-module $\mathbb{S}(V_+) \:\square\: \mathbb{S}(W_+)$ is bigraded by weights in $V$ and $W$, and, via the injection
\[
V_+ \boxtimes_c W_+ \hookrightarrow \mathbb{S}\big( V_+ \boxtimes_c W_+\big) \underset{\eqref{prop::S_permute_prod_connexe}}{\cong} \mathbb{S}(V_+) \:\square\: \mathbb{S}(W_+),
\]
the $\mathfrak{S}$-module $V_+ \boxtimes_c W_+$ naturally inherits a weight bigrading in $V$ and $W$. To express $V_+\boxtimes_cW_+(S)$ as a sum of terms indexed by the bigrading, we require the following notation.
\begin{nota} 
	Recall that, for all non-empty set $S$ and for all pairs $(I,J)$ in $\mathcal{Y}\ord_2(S)$, by definition, we have that $I$ and $J$ are also non-empty; we note:
	\[
		\mathcal{Y}^{\mathrm{or},+}_2(S)\overset{not}{:=}\mathcal{Y}\ord_2(S) \cup \{(S,\varnothing),(\varnothing,S)\}.
	\] 
	For a non-ordered partition $K\in\mathcal{Y}_n(S)$ with $n$ terms, we want  to distinguish the components of $V$ and these with the unit $\Ibox$  in $\bigotimes_{\alpha\in\Gamma(K)} V_+(K_\alpha)$ (where $\Gamma(K)$ is the non-ordered set labeling the partition $K$). So, we introduce the following functor:
	\[
		\mathfrak{Q}^V_{(R_1^K,R_2^K)} : \widetilde{\Gamma(K)} \longrightarrow \Ch_k,
	\]
	where $\widetilde{\Gamma(K)}$ is the discrete category on the set $\Gamma(K)$. For each $(R_1^K,R_2^K)$ in $\mathcal{Y}^{\mathrm{or},+}_2(\Gamma(K))$, the functor $\mathfrak{Q}^V_{(R_1^K,R_2^K)}$ associates, for all $K_\alpha$ with $\alpha$ in $\Gamma(K)$,  a chain complex as follows:
	\[	
	\mathfrak{Q}^V_{(R_1^K,R_2^K)}(K_\alpha):=
	\left\{\begin{array}{cc}
	V(K_\alpha) & \mbox{if } \alpha\in R_1^K \\
	\Ibox(K_\alpha) &  \mbox{if } \alpha\in R_2^K
	\end{array}\right.~~.
	\]
	Remark that, for $\alpha$ in $R_2^K$ such that $|K_\alpha|\geqslant 2$, the complex $\mathfrak{Q}^V_{(R_1^K,R_2^K)}(K_\alpha)$ is zero.
\end{nota}
So we decompose $V_+\boxtimes_cW_+$ as follow: for a finite set $S$, we have
\begin{align*}
V_+\boxtimes_cW_+(S) = &~ \big((V\oplus \Ibox) \boxtimes_c (W\oplus \Ibox)\big)(S)\\
\overset{\ref{def::prod_connexe_Smod}}{:=} &~ \underset{(K,L)\in \mathcal{X}^{\mathrm{conn}}(S)}{\bigoplus}~\underset{\alpha\in \Gamma(K)}{\bigotimes}~(V\oplus \Ibox)(K_\alpha) \otimes \underset{\beta\in \Gamma(L)}{\bigotimes}~(W\oplus \Ibox)(L_\beta) \\
\cong & 
\underset{{\substack{
			(K,L)\in \mathcal{X}^{\mathrm{conn}}(S)\\
			(R_1^K,R_2^K)\in \mathcal{Y}^{\mathrm{or},+}_2(\Gamma(K))\\
			(R_1^L,R_2^L)\in \mathcal{Y}^{\mathrm{or},+}_2(\Gamma(L))
		}}}{\bigoplus}	
	\underset{\alpha\in \Gamma(K)}{\bigotimes}\mathfrak{Q}^V_{(R_1^K,R_2^K)}(K_\alpha) \otimes \underset{\beta\in \Gamma(L)}{\bigotimes}\mathfrak{Q}^W_{(R_1^L,R_2^L)}(L_\beta).
\end{align*}
We collect terms by the number of copies of $V$ and $W$, then the $\mathrm{Aut}(S)$-module $V_+\boxtimes_cW_+(S)$ is isomorphic to
\begin{align*}
& 
\underset{{\substack{
			(r,s)\in \N^2 \\
			(K,L)\in \mathcal{X}^{\mathrm{conn}}(S)}}}{\bigoplus}
\underset{{\substack{
			(R_1^K,R_2^K)\in \mathcal{Y}^{\mathrm{or},+}_2(\Gamma(K))\\
			(R_1^L,R_2^L)\in \mathcal{Y}^{\mathrm{or},+}_2(\Gamma(L))\\ 
			|R_1^K|=r, |R^L_1|=s
		}}}{\bigoplus}
\underset{\alpha\in \Gamma(K)}{\bigotimes}\mathfrak{Q}^V_{(R_1^K,R_2^K)}(K_\alpha)\!
\otimes\!
\underset{\beta\in \Gamma(L)}{\bigotimes}\mathfrak{Q}^W_{(R_1^L,R_2^L)}(L_\beta) \\
 \cong &\Ibox(S)\oplus W(S) \oplus V(S) \\
 & 
 \oplus \underset{{\substack{
 			(r,s)\in (\N^*)^2 \\
 			(K,L)\in \mathcal{X}^{\mathrm{conn}}(S)
 		}}}{\bigoplus}
 \underset{{\substack{
 			((R_1^K,R_2^K),(R_1^L,R_2^L))\in \Xi_{K,L} \\ 
 			|R_1^K|=r, |R^L_1|=s 
 		}}}{\bigoplus}
 		\underset{\alpha\in R_1^K}{\bigotimes}V(K_\alpha) \otimes 
 		\underset{\beta\in R_1^L}{\bigotimes}W(L_\beta)
\end{align*}
where
\[
	\Xi_{K,L}:=\left\{\scriptstyle{
		\big((R_1^K,R_2^K),(R_1^L,R_2^L)\big)\in
		\left.
		\begin{array}{c} \scriptstyle{\big(\mathcal{Y}\ord_2(\Gamma(K))\cup(\Gamma(K),\varnothing)\big)}\\
		\scriptstyle{\times\big( \mathcal{Y}\ord_2(\Gamma(L))\cup (\Gamma(L),\varnothing)\big)}
		\end{array} \right| 
		\begin{array}{c}
		\forall \beta\in R_2^K, |K_\beta|=1, \\
		\forall \beta\in R_2^L, |L_\beta|=1
		\end{array}
	}\right\},
\]
which gives us the bigrading. This last isomorphism is given by distinguishing terms arising from the injections $\Ibox\boxtimes_c W \hookrightarrow V_+\boxtimes_c W_c$, $V\boxtimes_c \Ibox \hookrightarrow V_+\boxtimes_c W_c$ and $\Ibox \boxtimes_c \Ibox \hookrightarrow V_+\boxtimes_c W_c$. This gives the following.
\begin{lem} \label{lem::bigrading_proto}
	The isomorphism of $\mathrm{Aut}(S)$-modules between $V_+\boxtimes_c W_+(S)$ and 
	\[
	\underset{{\substack{
				(r,s)\in (\N^*)^2\\
				(K,L)\in \mathcal{X}^{\mathrm{conn}}(S)}
			}}{\bigoplus}
	\underset{{\substack{
				(R_1^K,R_2^K)\in \mathcal{Y}^{\mathrm{or},+}_2(\Gamma(K))\\
				(R_1^L,R_2^L)\in \mathcal{Y}^{\mathrm{or},+}_2(\Gamma(L))\\
				|R_1^K|=r, |R^L_1|=s
			}}}{\bigoplus}
	\underset{\alpha\in \Gamma(K)}{\bigotimes}\mathfrak{Q}^V_{(R_1^K,R_2^K)}(K_\alpha) \otimes \underset{\beta\in \Gamma(L)}{\bigotimes}\mathfrak{Q}^W_{(R_1^L,R_2^L)}(L_\beta)	
	\]
	gives the bigrading of $V_+\boxtimes_c W_+(S)= \bigoplus_{(r,s)\in \N^2} (V_+\boxtimes_c W_+)^{(r)_V,(s)_W}$ where, for integers $r$ and $s$ in $\N^*$, the term $(V_+\boxtimes_c W_+)^{(r)_V,(s)_W}(S)$ is isomorphic to
	\[
		\underset{(K,L)\in \mathcal{X}^{\mathrm{conn}}(S)}{\bigoplus}~\underset{{\substack{((R_1^K,R_2^K),(R_1^L,R_2^L))\in \Xi_{K,L}\\ |R_1^K|=r, |R^L_1|=s}}}{\bigoplus}~~
		\underset{\alpha\in R^K_1}{\bigotimes}~V(K_\alpha) \otimes \underset{\beta\in R^L_1}{\bigotimes}~W(L_\beta) 
	\]
 	with
	\[
	\Xi_{K,L}:=\left\{\scriptstyle{
		\big((R_1^K,R_2^K),(R_1^L,R_2^L)\big)\in
		\left.
		\begin{array}{c} \scriptstyle{\big(\mathcal{Y}\ord_2(\Gamma(K))\cup(\Gamma(K),\varnothing)\big)}\\
		\scriptstyle{\times\big( \mathcal{Y}\ord_2(\Gamma(L))\cup (\Gamma(L),\varnothing)\big)}
		\end{array} \right| 
		\begin{array}{c}
		\forall \beta\in R_2^K, |K_\beta|=1, \\
		\forall \beta\in R_2^L, |L_\beta|=1
		\end{array}
	}\right\},
	\]
	\begin{align*}
	\bigoplus_{s\in \N^*} (V_+\boxtimes_c W_+)^{(0)_V,(s)_W}(S)  \cong &~ (V_+\boxtimes_c W_+)^{(0)_V,(1)_W}(S)= W(S)~;\\
	\bigoplus_{r\in \N^*} (V_+\boxtimes_c W_+)^{(r)_V,(0)_W}(S) \cong &~(V_+\boxtimes_c W_+)^{(1)_V,(0)_W}(S)= V(S) \ \mbox{and} \\
	(V_+\boxtimes_c W_+)^{(0)_V,(0)_W}(S)= \Ibox(S).
	\end{align*}
\end{lem}
We describe $(V_+\boxtimes_cW_+)^{(1)_V,(1)_W}$ explicitly.
\begin{prop}[Case of $(V_+\boxtimes_cW_+)^{(1)_V,(1)_W}$] 					\label{rem::description_VboxtimesV}
	Let $V$ and $W$ be two reduced $\mathfrak{S}$-modules. For all finite sets $S$, we have the following isomorphism:
	\[
	(V_+\boxtimes_cW_+)^{(1)_V,(1)_W}(S) \cong \underset{{\substack{K,L\subset S \\ K\cup L= S \\ K\cap L \ne \varnothing}}}{\bigoplus}~V(K) \otimes W(L).
	\]
\end{prop}
\begin{proof}
	The $\mathrm{Aut}(S)$-module $(V_+\boxtimes_cW_+)^{(1)_V,(1)_W}(S) $ is isomorphic to
	\begin{align*}
		&~\underset{(K,L)\in \mathcal{X}^{\mathrm{conn}}(S)}{\bigoplus}~
		\underset{{\substack{(R_1^K,R_2^K)\in \mathcal{Y}^{\mathrm{or},+}_2(\Gamma(K))\\(R_1^L,R_2^L)\in \mathcal{Y}^{\mathrm{or},+}_2(\Gamma(L))\\ |R_1^K|=1, |R^L_1|=1}}}{\bigoplus}	\underset{\alpha\in \Gamma(K)}{\bigotimes}~\mathfrak{Q}^V_{(R_1^K,R_2^K)}(K_\alpha) \otimes \underset{\beta\in \Gamma(L)}{\bigotimes}~\mathfrak{Q}^W_{(R_1^L,R_2^L)}(L_\beta) \\
		\cong & ~ \underset{{\substack{(K,L)\in \mathcal{X}^{\mathrm{conn}}(S)\\ \exists a\in A, b\in B | \forall \alpha\in \Gamma(K)\backslash\{a\}, \beta\in \Gamma(L)\backslash\{b\} \\ K_\alpha\cong \{*\}\cong L_\beta}}}{\bigoplus}~	V(K_a) \otimes W(L_b)
	\end{align*}
	since, if there exists $\alpha$ in $A\backslash\{a\}$ such that $K_\alpha \not \cong \{*\}$, then $\Ibox(K_\alpha)=0$; likewise, for $B\backslash\{b\}$. Finally, we rewrite $(V_+\boxtimes_cW_+)^{(1)_V,(1)_W}(S)$ as follows:
	\[
		(V_+\boxtimes_cW_+)^{(1)_V,(1)_W}(S) \cong \underset{{\substack{K,L\subset S \\ K\cup L= S \\ K\cap L \ne \varnothing}}}{\bigoplus}~V(K) \otimes W(L).
	\]
\end{proof}
When we take $V=W$, the bigrading of $V_+\boxtimes_cW_+$ induces a weight grading of $V_+\boxtimes_cV_+$: for a finite set $S$, we have $V_+\boxtimes_cV_+(S)$, which is isomorphic to
\begin{align*}
 &  \underset{{\substack{
 			\rho\in \N^* \\
 			(K,L)\in \mathcal{X}^{\mathrm{conn}}(S)
 		}}}{\bigoplus}~
 \underset{{\substack{(R_1^K,R_2^K)\in \mathcal{Y}^{\mathrm{or},+}_2(\Gamma(K))\\(R_1^L,R_2^L)\in \mathcal{Y}^{\mathrm{or},+}_2(\Gamma(L))\\ |R_1^K|+ |R^L_1|=\rho}}}{\bigoplus}~~ \underset{\alpha\in \Gamma(K)}{\bigotimes}~\mathfrak{Q}^V_{(R_1^K,R_2^K)}(K_\alpha) \otimes \underset{\beta\in \Gamma(L)}{\bigotimes}~\mathfrak{Q}^V_{(R_1^L,R_2^L)}(L_\beta) \\
\cong &~\Ibox(S)\oplus \underbrace{V(S) \otimes \bigotimes_{S} \Ibox(*) \oplus \bigotimes_{S} \Ibox(*) \otimes V(S)}_{\overset{not}{=:}(V_+\boxtimes_cV_+)^{(1)_V}(S)} \\
& \oplus \underset{\rho\in \N\backslash\{0,1\}}{\bigoplus}~\underbrace{\underset{(K,L)\in \mathcal{X}^{\mathrm{conn}}(S)}{\bigoplus}~\underset{{\substack{((R_1^K,R_2^K),(R_1^L,R_2^L))\in \Xi_{K,L}\\ |R_1^K|+ |R^L_1|=\rho}}}{\bigoplus} ~~\underset{\alpha\in R_1^K}{\bigotimes}~V(K_\alpha) \otimes \underset{\beta\in R_1^L}{\bigotimes}~V(L_\beta)}_{\overset{not}{=:}(V_+\boxtimes_cV_+)^{(\rho)_V}(S)} 
\end{align*}
where we recall that  $\Xi_{K,L}:=\big\{\big((R_1^K,R_2^K),(R_1^L,R_2^L)\big)\in \big(\mathcal{Y}\ord_2(\Gamma(K))\cup(\Gamma(K),\varnothing)\big)\times\big( \mathcal{Y}\ord_2(\Gamma(L))\cup (\Gamma(L),\varnothing)\big) \big| \forall \beta\in R_2^K, |K_\beta|=1, \forall \beta\in R_2^L, |L_\beta|=1 \big\}$. More generally, the $\mathfrak{S}$-module $V_n:=(V_+)^{\boxtimes_c n}$ is weight-graded in $V$; for all finite set $S$, we have 
\begin{align*}
V_n& (S) \overset{def}{:=}~ (V_+)^{\boxtimes_c n} (S)\\
=&
\underset{{\substack{
			(J^1,\ldots,J^n)\in \mathcal{Y}^n(S)\\ 
			\mathcal{K}_S^{n-1}(J^1,\ldots,J^n)=S \\
			\big((R_1^{J^1},R_2^{J^1}),\ldots, (R_1^{J^n},R_2^{J^n}) \big)\\ \in \prod_{i\in [\![1,n]\!]} \mathcal{Y}^{\mathrm{or},+}_2(\Gamma(J^i))
		}}}{\bigoplus}~
\underset{\alpha\in\Gamma(J^1)}{\bigotimes} \mathfrak{Q}^{V}_{(R_1^{J^1},R_2^{J^1})}(J^1_\alpha) \otimes \ldots \otimes \underset{\alpha\in\Gamma(J^n)}{\bigotimes}
\mathfrak{Q}^{V}_{(R_1^{J^n},R_2^{J^n})}(J^n_\alpha)\\
\cong & 
\underset{\rho\in \N}{\bigoplus}
\underset{{\substack{
			(J^1,\ldots,J^n)\in \mathcal{Y}^n(S)\\ 
			\mathcal{K}_S^{n-1}(J^1,\ldots,J^n)=S \\
			\big((R_1^{J^i},R_2^{J^i})\big)_i\\ \in \prod_{i\in [\![1,n]\!]} \mathcal{Y}^{\mathrm{or},+}_2(\Gamma(J^i))\\  \sum_{i=1}^{n}|R_1^{J^i}|=\rho
		}}}{\bigoplus}
\underset{\alpha\in \Gamma(J^1)}{\bigotimes} \mathfrak{Q}^{V}_{(R_1^{J^1},R_2^{J^1})}(J^1_\alpha) \otimes \ldots \otimes \underset{\alpha\in \Gamma(J^n)}{\bigotimes}
\mathfrak{Q}^{V}_{(R_1^{J^n},R_2^{J^n})}(J^n_\alpha)
\end{align*}
and, if we note by $V_n^{(\rho)_V}$ the following:
\[
\underset{{\substack{
			l\in[\![1,n]\!]\\
			\{r_1,\ldots,r_l \}\subset [\![1,n]\!] \\ 
			r_1<r_2<\ldots<r_l \\
			(J^{r_1},\ldots,J^{r_l})\in \mathcal{Y}^s(S)\\ \mathcal{K}_S^{l-1}(J^{r_1},\ldots,J^{r_l})=S
}}}{\bigoplus} 
\underset{{\substack{
			((R_1^{J^{r_1}},R_2^{J^{r_1}}),\ldots, (R_1^{J^{r_l}},R_2^{J^{r_l}}))\\ \in\Xi_{(J^{r_1},\ldots ,J^{r_l})} \\ 
			\sum_{i\in\{r_1,\ldots,r_l \} }|R_1^{J^i}|=\rho
}}}{\bigoplus}	
\underset{\alpha\in R_1^{J^{r_1}}}{\bigotimes}~
V(J^{r_1}_\alpha) \otimes \ldots \otimes 
\underset{\alpha\in R_1^{J^{r_l}}}{\bigotimes} V(J^{r_l}_\alpha)
\]

where the set $\Xi_{(J^1,\ldots ,J^n)}$ is the following
\[
	\left\{
	\begin{array}{c} 
		\big((R_1^{J^1},R_2^{J^1}),\ldots, (R_1^{J^n},R_2^{J^n}) \big) \\
		\in \prod_{i\in [\![1,n]\!]} \big(\mathcal{Y}\ord_2(\Gamma(J^i))\cup (\Gamma(J^i),\varnothing)\big)
	\end{array}
	\left| \forall i\in[\![1,n]\!], \forall \beta \in R_2^{J^i}, |J^i_\beta |=1 
	\right.\right\},
\]
then we have
\[
V_n(S) \cong  \Ibox(S) \oplus  V_n^{(\rho)_V}.
\]

The isomorphisms $\lambda_V : \Ibox\boxtimes_c V \rightarrow$ and $\rho_V:V\boxtimes_c \Ibox \rightarrow V$ preserve the grading because $V\boxtimes_c \Ibox=(V\boxtimes_c \Ibox)^{(1)_V}$, and preserve all the constructions  which are in the construction of $\widetilde{V}_n$ (cf. \Cref{subsect::monoide_libre}), then the grading of $V_n$ carries on $\widetilde{V}_n$. Also, the injection $\widetilde{V}_n \hookrightarrow \widetilde{V}_{n+1}$ preserves the weight-grading, so, finally, the free monoid $\scrF(V)$ is weight-graded by the number of copies of $V$.

\begin{prop}[First description of the free protoperad  $\scrF(V)$]  \label{prop::grading_proto_libre}
	Let $S$ be a finite set and $V$ be a reduced $\mathfrak{S}$-module. We have the isomorphism of chain complexes:
	\[
		\scrF(V)(S)\!\cong \!\Ibox(S) \oplus \underset{n\in \N^*}{\bigoplus} 
		\underset{{\substack{
					(J^1,\ldots,J^n)\in \mathcal{Y}^n(S)\\ \mathcal{K}_S^{n-1}(J^1,\ldots,J^n)=S
				}}}{\bigoplus}
		\underset{\widetilde{\Xi}_{(J^1,\ldots,J^n)}}{\bigoplus}
		\underset{\alpha\in R_1^{J^1}}{\bigotimes} V(J^1_\alpha) \otimes \ldots \otimes
		\underset{\alpha\in R_1^{J^n}}{\bigotimes} V(J^n_\alpha),
	\]
	where
	\[
		\widetilde{\Xi}_{(J^1,\ldots,J^n)}:=\left\{
		\begin{array}{c}
		\scriptstyle{\big((R_1^{J^1},R_2^{J^1}),\ldots, (R_1^{J^n},R_2^{J^n}) \big)} \\
		\scriptstyle{\in \prod_{i\in [\![1,n]\!]} \big(\mathcal{Y}\ord_2(\Gamma(J^i))\cup (\Gamma(J^i),\varnothing)\big) }
		\end{array}
		\left|
		\begin{array}{c}
		\scriptstyle{\forall i \in[\![1,n-1]\!],\forall \beta \in R^{J^{i+1}}_1, }\\
		\scriptstyle{J^{i+1}_\beta \cap \coprod_{\alpha\in R^{J^i}_1} J_\alpha^i\ne \varnothing} \\
		\scriptstyle{\forall i \in[\![1,n]\!],\forall \beta \in R^{J^{i}}_2, |J^i_\beta|=1}
		\end{array}
		\right.\right\} .
	\]
	Moreover, the free protoperad is weight-graded:
	\[ 
	\scrF(V) \cong \underset{\rho\in\N}{\bigoplus} \scrF^{(\rho)}(V)
	\]
	with $\scrF^{(0)}(V)=\Ibox (V)$ and, for all integers $\rho$ in $\N^*$, the $\rho$-weighted part $\scrF^{(\rho)}(V)(S)$ is isomorphic to: 
	\[
	\underset{n\in \N^*}{\bigoplus} 
	\underset{{\substack{(J^1,\ldots,J^n)\in \mathcal{Y}^n(S)\\ \mathcal{K}_S^{n-1}(J^1,\ldots,J^n)=S}}}{\bigoplus}~
	\underset{{\substack{\widetilde{\Xi}_{(J^1,\ldots,J^n)}\\ \sum_{i=1}^n |R_1^{J^i}|=\rho }}}{\bigoplus}~
	\underset{\alpha\in R_1^{J^1}}{\bigotimes}~V(J^1_\alpha) \otimes \ldots \otimes \underset{\alpha\in R_1^{J^n}}{\bigotimes}~V(J^n_\alpha) .
	\]
\end{prop}
\begin{proof}
	This isomorphism corresponds to the choice of a representative for the quotient $V_n \twoheadrightarrow \widetilde{V}_n$, as we will see below. We define the morphism $\tau_V : V \rightarrow V_2$ as the following composition:
	\[
	\xymatrix@C=5pc{
		V \ar[r]_(0.35){\lambda_V^{-1}+ \rho_V^{-1}} \ar@{.>}@/^1.5pc/[rr]^{=:\tau} & \Ibox \boxtimes_c V \oplus V\boxtimes_c \Ibox \ar[r]_(0.4){\eta\boxtimes_c i_V - i_V\boxtimes_c \eta} &  (\Ibox\oplus V)\boxtimes_c(\Ibox\oplus V)=:V_2
	}
	\]
	with $\eta : \Ibox \hookrightarrow \Ibox\oplus V$, which  appears in the definition of $R_{A,B}$ for all reduced $\mathfrak{S}$-modules $A$ and $B$: the $\mathfrak{S}$-module $R_{A,B}$ is defined as the image of the composition:
	\[
	\begin{tikzcd}
		A \boxtimes_c( \underline{V}\oplus V_2) \boxtimes_c B 
		\ar[r, hook,"=: \iota_{A,B}"] 
		& A \boxtimes_c( V\oplus V_2) \boxtimes_c B \ar[rrr,"A\boxtimes_c(\tau+\mathrm{id}_{V_2})\boxtimes_c B"] 
		& & & A \boxtimes_c V_2 \boxtimes_c B
	\end{tikzcd} .
	\] 
	We also define the $\mathfrak{S}$-module $\widetilde{V}_n$ as the cokernel of the morphism 
	\[
	\bigoplus_{i=0}^{n-2} R_{V_i, V_{n-i-2}} \longrightarrow V_n.
	\] 
	The quotient $\widetilde{V}_n$ corresponds to the identification of the images of morphisms $\big(\mathrm{id}_{V_i}\boxtimes_c \eta \boxtimes_c i_V \boxtimes_c \mathrm{id}_{V_{n-i-2}}\big)\circ \iota_{V_i,V_{n-i-2}}$ and $\big(\mathrm{id}_{V_i}\boxtimes_c i_V \boxtimes_c \eta \boxtimes_c \mathrm{id}_{V_{n-i-2}}\big)\circ \iota_{V_i,V_{n-i-2}}$, for all $i\in[\![1,n]\!]$, in $V_n$. We choose to identify each class of $\widetilde{V}_n$ with an element of the image of $\Sigma_{i\in [\![0,n-2]\!]} \big(\mathrm{id}_{V_i}\boxtimes_c \eta \boxtimes_c i_V \boxtimes_c \mathrm{id}_{V_{n-i-2}}\big)\circ \iota_{V_i,V_{n-i-2}}$. Then, for all finite sets $S$, we have 
	\begin{align*}
	\widetilde{V}_n(S) \cong &
	\underset{h\in [\![1,n]\!]}{\bigoplus} \underset{{\substack{(J^1,\ldots,J^h)\in \mathcal{Y}^h(S)\\ \mathcal{K}_S^{h-1}(J^1,\ldots,J^h)=S}}}{\bigoplus}
	\underset{\widetilde{\Xi}_{(J^1,\ldots,J^h)}}{\bigoplus}
	\underset{\alpha\in R_1^{J^1}}{\bigotimes}~V(J^1_\alpha) \otimes \ldots \\
	& \qquad \otimes \underset{\alpha\in R_1^{J^h}}{\bigotimes}~V(J^h_\alpha)
	\underbrace{\otimes \bigotimes_S I(*) \otimes \ldots \otimes\bigotimes_S I(*)}_{n-h~\mathrm{terms}} \ ,
	\end{align*}
	where 
	\[
	\widetilde{\Xi}_{(J^1,\ldots,J^h)}:=
	\left\{
	\begin{array}{c}
	\scriptstyle{\big((R_1^{J^1},R_2^{J^1}),\ldots, (R_1^{J^h},R_2^{J^h}) \big)} \\
	\scriptstyle{\in \prod_{i\in [\![1,h]\!]} \big(\mathcal{Y}\ord_2(\Gamma(J^i))\cup (\Gamma(J^i),\varnothing)\big)} 
	\end{array}
	\bigg|
	\begin{array}{c}
	\scriptstyle{\forall i \in[\![1,h-1]\!],\forall \beta \in R^{J^{i+1}}_1,}\\
	\scriptstyle{J^{i+1}_\beta \cap \underset{\alpha\in R^{J^i}_1}{\coprod} J_\alpha^i\ne \varnothing \ ;}\\
	\scriptstyle{\forall i \in[\![1,h]\!],\forall \beta \in R^{J^{i}}_2, |J^i_\beta|=1}
	\end{array}
	\right\} .
	\]
	Then, for all integers $n>1$, we have
	\[
	\widetilde{V}_n(S) \cong \widetilde{V}_{n-1}(S)\oplus \underset{{\substack{(J^1,\ldots,J^n)\in \mathcal{Y}^n(S)\\ \mathcal{K}_S^{n-1}(J^1,\ldots,J^n)=S}}}{\bigoplus}
	\underset{\widetilde{\Xi}_{(J^1,\ldots,J^n)}}{\bigoplus}
	\underset{\alpha\in R_1^{J^1}}{\bigotimes}~V(J^1_\alpha) \otimes \ldots \otimes \underset{\alpha\in R_1^{J^n}}{\bigotimes}~V(J^n_\alpha),
	\]
	which exactly describes the injections $\widetilde{V}_{n-1} \hookrightarrow \widetilde{V}_n$. 
\end{proof}

\begin{rem}[Description of $\scrF^{(2)}(V)$] \label{rem::description_F2}
	We have the isomorphism 
	\[
	(V_+\boxtimes_cV_+)^{(2)_V}\cong (V_+\boxtimes_cV_+)^{(1)_V,(1)_V},
	\]
	then, by  \Cref{rem::description_VboxtimesV} and \Cref{prop::grading_proto_libre}, we have an explicit description of the sub-$\mathfrak{S}$-module of weight $2$ of the free protoperad:
	\[
	\scrF^{(2)}(V)\cong \underset{{\substack{K,L\subset S \\ K\cup L= S \\ K\cap L \ne \varnothing}}}{\bigoplus}~V(K) \otimes V(L) .
	\]
\end{rem}


By the definition of the free monoid in $(\Smodred_k, \boxtimes_c,\Ibox)$, the following proposition and corollary follow from \Cref{prop::foncteur_free_monoide} and \cite[Prop. 28]{Val03}. Thereby, the functor $\Ind$ sends a protoperad defined by generators $\mathcal{G}$ and relations $\mathcal{R}$ to a properad defined by generators $\Ind(\mathcal{G})$ and relations $\Ind(\mathcal{R})$  
\begin{prop}\label{prop::Ind_commute_F}
	\begin{enumerate}
		\item 	The functor $\Ind$ commutes with the functors $\scrF$ and $\scrF\Val$, i.e. 
		\[
		\scrF\Val\big(\Ind(-)\big) \cong \Ind\big(\scrF(-)\big),
		\]
		where $\scrF\Val$ is the functor of free properad (see \cite{Val03,Val07}).
		\item 	Let $V$ be a reduced $\mathfrak{S}$-module and $R$ be a sub-$\mathfrak{S}$-module of the free monoid $\F(V)$. Then, we have the isomorphism:
		\[
		\Ind\Big( \scrF(V)\big/ \langle R\rangle \Big) \cong \scrF\big(\Ind(V)\big)\Big/ \big\langle \Ind(R) \big\rangle.
		\]
	\end{enumerate}
\end{prop}
\begin{proof}
	This follows from  \Cref{thm::Ind_monoidal_connexe}, \Cref{prop::Ind_exact}, \cite[Prop. 28]{Val03} and the construction of the free monoid
\end{proof}

We have seen that the description of the free protoperad is rather complicated, so we  prefer a combinatorial description generalizing the identification of \Cref{rem::description_F2}. This is the purpose of the next subsection.

\subsection{Combinatorial version of the free protoperad}
We begin by defining the category $(\Fin\op)_{\mathcal{W}\ord_n}$ which has for objects the pairs $(S,x)$ with $S$, a finite set and $x$ an element of $\mathcal{W}\ord_n(S)$. A morphism from $(S,x)$ to $(T,y)$ is a morphism $\phi : S \rightarrow T$ in the category $\Fin\op$ such that $\mathcal{W}\ord_n(\phi)(x)=y$. We have the two canonical functors:
\begin{itemize}
	\item the functor $\For  : \big(\Fin\op\big)_{\mathcal{W}\ord_n} \rightarrow (\Fin\op)^{\times_n}$ is defined, for an object $(S,X)$, by $\For_n(S,x)=\bar{x}$ with $\bar{x}$ the $n$-tuple of sets underlying $x$ and, on morphisms, by $\For_n(\phi)=\mathcal{W}\ord_n(\phi)$;
	\item the projection functor $\pi : \big(\Fin\op\big)_{\mathcal{W}_n} \rightarrow \Fin\op$ defined by $\pi(S,x)=S$.
\end{itemize}
Fix a reduced $\mathfrak{S}$-module $P$: by the external product of $\mathfrak{S}$-modules, we define the functor
\[
	P^{\times n} : \big(\Fin\op\big)^{\times n} \longrightarrow \Ch_k.
\]
Then, we have the functors
\[
	\begin{array}{rccc}
	\hat{P} : & \big(\Fin\op\big)_{\mathcal{W}\ord_n} & \longrightarrow & \Ch_k \\
	& (S,x) & \longmapsto & P^{\times n} \circ \For_n(S,x)
	\end{array}~;
\]
and
\[
\begin{array}{rccc}
	P^{\mathcal{W}\ord_n} : & \Fin\op & \longrightarrow & \Ch_k \\
	& S & \longmapsto & \bigoplus_{x\in \mathcal{W}\ord_n(S)}P^{\times n} \circ \For_n(S,x)
\end{array}~.
\] 
Finally, for all finite sets $S$, we have 
\[
	P^{\mathcal{W}\ord_n}(S)=\bigoplus_{\big((K_1,\ldots,K_n),\leqslant\big)\in \mathcal{W}\ord_n(S)} P(K_1) \otimes \ldots \otimes P(K_n)~.
\]
The free action of $\mathfrak{S}_n$ on $P^{\times n}$ acts on  $P^{\mathcal{W}\ord_n}$, by the symmetry of $(\Ch_k,\otimes,\sigma)$: then, we obtain the functor
\begin{align*}
	\big(P^{\mathcal{W}\ord_n}\big)_{\mathfrak{S}_n}(S)= &~\Big(\bigoplus_{\big((K_1,\ldots,K_n),\leqslant\big)\in \mathcal{W}\ord_n(S)} P(K_1) \otimes \ldots \otimes P(K_n)\Big)_{\mathfrak{S}_n}~\\
	\cong~&  \bigoplus_{\big(\{K_\alpha\}_{\alpha\in A},\leqslant\big)\in \mathcal{W}_n(S)}  \Big( \bigoplus_{(K_1,\ldots,K_n)\in \{K_\alpha\}_{\alpha\in A}}P(K_1) \otimes \ldots \otimes P(K_n)\Big)_{\mathfrak{S}_n}
\end{align*}
which we denote as follows:
\[
	P^{\mathcal{W}_n}(S)\overset{\mathrm{not.}}{:=}\bigoplus_{\big(\{K_\alpha\}_{\alpha\in A},\leqslant\big)\in \mathcal{W}_n(S)} \bigotimes_{\alpha\in A} P(K_\alpha).
\]
The construction works for the functor of connected walls $ \mathcal{W}^{conn,ord}_n$, which gives us the functor
\[
P^{\Wall_n}(S):=\bigoplus_{\big(\{K_\alpha\}_{\alpha\in A},\leqslant\big)\in \Wall_n(S)} \bigotimes_{\alpha\in A} P(K_\alpha).
\]
We define the partial composition product. Let $\phi$ be a diagram of injections $i: S \hookrightarrow R \hookleftarrow T:j $ with $\mathrm{im}(i) \cap \mathrm{im}(j)\ne\varnothing$ and  $\mathrm{im}(i) \cup \mathrm{im}(j)=R$. We have, by  \Cref{prop::produits_sur_W}, the morphism
\begin{align*}
P^{\mathcal{W}\ord_m}(S)& \otimes P^{\mathcal{W}\ord_n}(T) =~
\Big(\bigoplus_{\big((K_1,\ldots,K_m),\leqslant_K\big)\in \mathcal{W}\ord_m(S)} P(K_1) \otimes \ldots \otimes P(K_m) \Big) \\
&\qquad\qquad\qquad\qquad \otimes \Big(\bigoplus_{\big((L_1,\ldots,L_n),\leqslant_L\big)\in \mathcal{W}\ord_n(T)} P(L_1) \otimes \ldots \otimes P(L_n)\Big) \\
\cong~ &\bigoplus_{{\substack{\big((K_1,\ldots,K_m),\leqslant_K\big)\in \mathcal{W}\ord_m(S) \\ \big((L_1,\ldots,L_n),\leqslant_L\big)\in \mathcal{W}\ord_n(T)}}} P(K_1) \otimes \ldots \otimes P(K_m)\otimes  P(L_1) \otimes \ldots \otimes P(L_n) \\
\rightarrow~&
\bigoplus_{{\substack{
	\big((i(K_1),\ldots,i(K_m), \\
	 j(L_1),\ldots,j(L_n)),
	\leqslant_{i(K)}^{j(L)}\big) \\
	\in \mathcal{W}\ord_{m+n}(R) 
}}}
 P(K_1) \otimes \ldots \otimes P(K_m)\otimes  P(L_1) \otimes \ldots \otimes P(L_n),
\end{align*}
where $\leqslant_{i(K)}^{j(L)}$ is defined as follows: we have on $\cup_a i(K_a)$ (resp. $\cup_b j(L_b)$) the partial order $\leqslant_{i(K)}$ (resp.$\leqslant_{j(L)}$) induced by that of $K$ (resp. $L$), i.e. $i(K_a) \leqslant_{i(K)}^{j(L)} i(K_b)$ if $K_a\leqslant_K K_b$ (resp.  $j(L_a) \leqslant_{i(K)}^{j(L)} j(L_b)$ if $L_a\leqslant_L L_b$) which gives to $(K_1,\ldots,\K_m,L_1,\ldots,L_n)$, the order  $\leqslant_{i(K)}^{j(L)}$, by \Cref{lem::ordre_partiel_canonique}. Thus we have the morphism
\[
	P^{\mathcal{W}\ord_m}(S) \otimes P^{\mathcal{W}\ord_n}(T)\longrightarrow ~ P^{\mathcal{W}_{m+n}}(R) =\bigoplus_{\big(\{H_\alpha\}_{\alpha\in A},\leqslant_H\big)\in \mathcal{W}_{m+n}(R) } \bigotimes_{\alpha\in A} P(H_\alpha)
\]
which factorizes through $P^{\mathcal{W}_m}(S) \otimes P^{\mathcal{W}_n}(T)$, giving the partial composition product 
\[
	\underset{\phi}{\circ} : P^{\mathcal{W}_m}(S) \otimes P^{\mathcal{W}_n}(T)\longrightarrow ~ P^{\mathcal{W}_{m+n}}(R).
\]
As $\mathrm{im}(i)\cap\mathrm{im}(j)\ne\varnothing$, this partial composition product can be restricted to the connected version: we have the partial composition product
\[
	\underset{\phi}{\circ} : P^{\Wall_m}(S) \otimes P^{\Wall_n}(T)\longrightarrow ~ P^{\Wall_{m+n}}(R).
\]
These partial products make $P^{\Wall}:= \coprod_{n}P^{\Wall_n}$, a  protoperad by \Cref{prop::proto_def_partielle}.

\begin{thm}[Description of the free protoperad]\label{prop::proto_libre}
	Let $V$ be a reduced $\mathfrak{S}$-module and $\rho$ be an integer in $\N^*$. We have, for all finite sets $S$, the isomorphism of (right) $\mathrm{Aut}(S)$-modules
	\[
		\scrF^\rho(V)(S) \cong \bigoplus_{{\substack{(\{K_\alpha\}_{\alpha\in A},\leqslant)\\ \in \mathcal{W}^{\mathrm{conn}}_\rho(S)}}} \bigotimes_{\alpha\in A} V(K_\alpha).
	\]
\end{thm}
\begin{proof}
	By \Cref{prop::grading_proto_libre}, for all $\rho$ in $\N^*$,  we have the isomorphism
	\[
		\scrF^{(\rho)}(V)(S)\cong  \underset{n\in \N^*}{\bigoplus} 
		\underset{{\substack{(J^1,\ldots,J^n)\in \mathcal{Y}^n(S)\\ \mathcal{K}_S^{n-1}(J^1,\ldots,J^n)=S}}}{\bigoplus}~
		\underset{{\substack{\widetilde{\Xi}_{(J^1,\ldots,J^n)}\\ \sum_{i=1}^n |R_1^{J^i}|=\rho }}}{\bigoplus}~
		\underset{\alpha\in R_1^{J^1}}{\bigotimes}~V(J^1_\alpha) \otimes \ldots \otimes \underset{\alpha\in R_1^{J^n}}{\bigotimes}~V(J^n_\alpha) .
	\]
	with 
	\[
		\widetilde{\Xi}_{(J^1,\ldots,J^n)}:=
		\left\{
		\begin{array}{c}
			\scriptstyle{\big((R_1^{J^1},R_2^{J^1}),\ldots, (R_1^{J^n},R_2^{J^n}) \big)} \\
			\scriptstyle{\in \prod_{i\in [\![1,n]\!]} \big(\mathcal{Y}\ord_2(\Gamma(J^i))\cup (\Gamma(J^i),\varnothing)\big)} 
		\end{array}
		\left|
		\begin{array}{c}
			\scriptstyle{\forall i \in[\![1,n-1]\!],\forall \beta \in R^{J^{i+1}}_1,}\\
			\scriptstyle{J^{i+1}_\beta \cap \coprod_{\alpha\in R^{J^i}_1} J_\alpha^i\ne \varnothing} \\
			\scriptstyle{\forall i \in[\![1,n]\!],\forall \beta \in R^{J^{i}}_2, |J^i_\beta|=1}
		\end{array}
		\right.
		\right\} .
	\]
	Let $(J^1,\ldots,J^n)$ in $\mathcal{Y}^n(S)$ such that $\mathcal{K}_S^{n-1}(J^1,\ldots,J^n)=S$ and $\widetilde{\Xi}_{(J^1,\ldots,J^n)}\ne \varnothing$; we associate to $(J^1,\ldots,J^n)$ the wall $W$ in $\Wall_\rho(S)$  with sets $\big\{ J^i_{\alpha^i}\subset S~|~i\in [\![1,n]\!],~\alpha^i \in R_1^{J^i}\big\}$ and the partial order induced by the relations $J^i_\alpha<J^j_\beta$ if $J^i_\alpha\cap J^j_\beta\ne\varnothing$, $\alpha\in R_1^{J^i}$, $\beta\in R_1^{J^j}$ and $i<j$. So we have the following morphism of $\mathrm{Aut}(S)$-modules:
	\[
		\Phi : \scrF^\rho(V)(S) \longrightarrow \bigoplus_{{\substack{(\{W_\alpha\}_{\alpha\in A},\leqslant)\\ \in \mathcal{W}^{\mathrm{conn}}_\rho(S)}}} \bigotimes_{\alpha\in A} V(W_\alpha).
	\]
	Conversely, to a connected wall $W=\big(\{W_\alpha~|~\alpha\in A\},\leqslant\big)$ with $\rho$ bricks, i.e. $W$ is in $\Wall_\rho(S)$, with $\mathrm{max}_{\alpha\in A}\big(\mathfrak{h}(W_\alpha)\big)=n$ (where $\mathfrak{h}:W \rightarrow \N \cup \{\infty\}$ is the height in the poset $W$, see \Cref{subsect::recalls on posets}), we associate an element $(J^1,\ldots,J^n)\in \mathcal{Y}^n(S)$ such that $\widetilde{\Xi}_{(J^1,\ldots,J^n)}\ne \varnothing$ as follows. We construct partitions $J^i$ as the sets $\{W_{\alpha^i}\in W~|~\mathfrak{h}(W_\alpha)=i\}$, extended to a partition by singletons: so we have
	\[
		J^i:=  \big\{W_{\alpha^i}\in W~|~\mathfrak{h}(W_\alpha)=i\big\} \amalg \big\{ \{s\}~\big|~s\notin \underset{\alpha^i}{\amalg}W_{\alpha^i} \big\}=\big\{J^i_\beta~|~\beta\in B=\Gamma(J^i) \big\}
	\]
	and the decomposition  $(R_1^{J^i},R_2^{J^i})\in  \mathcal{Y}\ord_2(\Gamma(J^i))\cup (\Gamma(J^i),\varnothing)$ is given by the definition of $J^i$ :
	\[
		\beta \in \left\{
		\begin{array}{cl}
		R_1^{J^i} & \ \mathrm{if} \ J^i_\beta\in \big\{W_{\alpha^i}\in W~|~\mathfrak{h}(W_\alpha)=i\big\}, \\
		R_2^{J^i} & \ \mathrm{otherwise}.
		\end{array}\right.
	\]
	The connectedness of the wall $W$ implies that the element $(J^1,\ldots,J^n)$ also satisfies the property of connectedness 
	\[
		\mathcal{K}_S^{n-1}(J^1,\ldots,J^n)=S.
	\]
	Finally, we have the following morphism of $\mathrm{Aut}(S)$-modules~:
	\[
		\Psi : \bigoplus_{{\substack{(\{W_\alpha\}_{\alpha\in A},\leqslant)\\ \in \mathcal{W}^{\mathrm{conn}}_\rho(S)}}} \bigotimes_{\alpha\in A} V(W_\alpha) \longrightarrow\scrF^\rho(V)(S)
	\]
	which satisfies $\Phi\circ \Psi=\mathrm{id}$ and $\Psi\circ \Phi= \mathrm{id}$.
\end{proof}

\section{Colours on walls}\label{sect::coloring}
In this section, we associate to a wall $W$ in $\Wall$, a chain complex, called the \emph{colouring complex}: it is determined by the combinatorics of the wall $W$. This is a combinatorical introduction to some results of \cite{Ler18ii}: the colouring complex of a connected wall $W$ over $S$ encodes a part of the differential of the bar construction of a free protoperad (see \cite[Sect. 2]{Ler18ii}).

\subsection{Coloured bricks}
In this subsection, we define the notion of a colouring of a wall: throughout this section, we consider $S$, a non-empty finite set.

\begin{defi}[Colouring] 
Let $(W,\leqslant_W)$ in $\mathcal{W}(S)$ be a wall over $S$. A (connected) $C$-\emph{colouring} of $W$ is  a surjective morphism of sets $\phi : W \twoheadrightarrow C$, where $C$ is called \emph{the set of colours}, satisfying the following assertions: 
\begin{enumerate}
	\item the binary relation $\leqslant_\phi$ induced on $C$ by the partial order of $W$, defined, for all $c_1,c_2$ in $C$, by 
	\[
		c_1\leqslant_{\phi} c_2 \mbox{ if } \exists k_1 \in\phi^{-1}(c_1),k_2 \in\phi^{-1}(c_2) \mbox{ such that } k_1\leqslant_W k_2;
	\]
	is a partial order;
	\item the fibers of $\phi$ are connected, i.e. for each colour $c$ in $C$, the set  
	$ \phi^{-1}(c)$ belongs to  $\Wall(S_c)$ with $S_c:= \bigcup_{W_\alpha\in \phi^{-1}(c)} W_\alpha$.
\end{enumerate}

We denote by $\mathrm{Succ}(\phi)$ or $\mathrm{Succ}(C)$ the set of successive colours. Two colouring  $\phi : W \twoheadrightarrow C $ and $\psi: W \twoheadrightarrow D$ of a wall $W$ are isomorphic  if there exists an isomorphism of posets $\Phi : (C, \leqslant_{\phi}) \rightarrow (D,\leqslant_{\psi})$ such that the following diagram commutes: 
\[
	\begin{tikzcd}
		& W \ar[rd, two heads, bend left=15,  "\psi"] \ar[ld, two heads, bend right=15, "\phi"'] & \\
		C \ar[rr, "\cong", "\Phi"'] & & D.
	\end{tikzcd}
\]	
We denote by $\Colo(W)$ the set of isomorphism classes of colourings of $W$: 
\[
	\Colo(W):=\big\{ \phi : W \rightarrow C |\phi~\mbox{a colouring}\big\}/_{\cong},
\]
which is graded by the number of colours:
\[
	\Colo_\bullet(W)= \coprod_{n\in\N^*} \Colo_n(W) ~\mbox{ with }~\Colo_n(W):=\big\{\phi\in \Colo(W)~|~|\phi(W)|=n\big\}.
\]
\end{defi}
\begin{rem}\label{rem::decompo_colo}
	\begin{enumerate}
		\item As any colouring $\phi : W \rightarrow C$ is a surjective map, for any colour $c$ in $C$, the set of $C$-coloured bricks is non empty, i.e.
		$\phi^{-1}(c)\ne\varnothing$. Furthermore, as $W$ is a wall over $S$, i.e. $W \in \mathcal{W}(S)$, we have $\bigcup_{c\in C} S_c =S$ with
		\[
			S_c:= \bigcup_{W_\alpha\in \phi^{-1}(c)} W_\alpha.
		\]
		\item $\Colo_{|W|}(W)$ is reduced to a unique element.
		\item If $n>|W|$, then $\Colo_n(W)=\varnothing$, as in example \ref{exmp::top_coloriage}.
		\item \label{rem_iii::decompo_colo} Let $W$ be a non-connected wall over S. The decomposition into connected parts (cf. proposition  \ref{prop::decompo_mur_connexe}) of $W=W^1\amalg \ldots \amalg W^l$ implies that we have the following graded product: 
	\[
		\Colo_\bullet(W)=\prod_{i=1}^l \Colo_\bullet(W^i).
	\]
	\end{enumerate}
\end{rem}

\begin{exmp}[Top-colouring]\label{exmp::top_coloriage} 
For any wall $W$, the identity morphism $W\rightarrow W$ defines the \emph{top-colouring}, in which each brick of $W$ has a different colour. For example, we represent the top-colouring of the following wall
\[
	\begin{tikzpicture}[scale=0.4,baseline=3ex]
		\draw (0,1.5) rectangle (2,2);
		\draw (1.25,0.75) rectangle (3.25,1.25);
		\draw (-1,0.75) rectangle (1,1.25);
	\end{tikzpicture} ~~
\]
by the coloured diagram:
\[
	\begin{tikzpicture}[scale=0.4,baseline=3ex]
		\draw (0,1.5) rectangle (2,2);
		\draw[fill=lightgray] (1.25,0.75) rectangle (3.25,1.25);
		\draw[fill=black] (-1,0.75) rectangle (1,1.25);
	\end{tikzpicture}~~.
\]
\end{exmp}
\begin{exmp}[Bot-colouring]~\label{exmp::bot_coloriage} For any connected wall, the projection to a point $W\rightarrow \{*\}$ defines a colouring called the \emph{bot-colouring}, denoted by $bot_W$, which colours all the bricks of $W$ the same colour; for example, we represent diagrammatically the bot-colouring of the previous wall by:
\[
	\begin{tikzpicture}[scale=0.4,baseline=3ex]
		\draw[fill=black] (0,1.5) rectangle (2,2);
		\draw[fill=black] (1.25,0.75) rectangle (3.25,1.25);
		\draw[fill=black] (-1,0.75) rectangle (1,1.25);
	\end{tikzpicture}~~.
\]
	If $W\in\mathcal{W}(S)$ is a non connected wall over $S$, then, by the proposition \ref{prop::decompo_mur_connexe}, we have its decomposition in connected component  $W^1 \amalg \ldots \amalg W^n$, and the \emph{bot-colouring} of $W$, also denoted by $bot_W$, is given by $bot_W=bot_{W_1} \amalg \ldots \amalg bot_{W_n}$; for example:
\[
	\begin{tikzpicture}[scale=0.4,baseline=3ex]
		\draw[fill=black] (0,1.5) rectangle (2,2);
		\draw[fill=black] (1.25,0.75) rectangle (3.25,1.25);
		\draw[fill=black] (-1,0.75) rectangle (1,1.25);
		\draw[fill=gray] (3.5,0.75) rectangle (5.5,1.25);
		\draw[fill=gray] (3.5,1.5) rectangle (5.5,2);
	\end{tikzpicture}~~.
\]
\end{exmp}
\begin{nonexmp}\label{ctrexmp::colo1}
We consider the wall  $W=\{W_a,W_b,W_c,W_d\}$ in $\mathcal{W}^{\mathrm{conn}}([\![1,4]\!])$ over $S=[\![1,4]\!]$ with 
\[
	W_a=\{1,2\}, W_b=\{3,4\},W_c=\{2,3\} \mbox{ and } W_d=\{1,4\} 
\]
and the partial order given by $W_a<W_c$, $W_a<W_d$, $W_b<W_c$ and $W_b<W_d$. We consider the surjective map $f: W \rightarrow \{w,b\}$ which maps $W_a$ and  $W_c$ to $b$ and $W_b$ and $W_d$ to $w$: we diagrammatically represent $f$ by
\[
	\begin{tikzpicture}[scale=0.4,baseline=1ex]
		\draw[fill=black] (0,0) rectangle (2,0.5);
		\draw[white] (1,0.20) node {\scriptsize{a}};
		\draw (2.5,0) rectangle (4.5,0.5);
		\draw (3.5,0.20) node {\scriptsize{b}};
		\draw[fill=black] (1.25,0.75) rectangle (3.25,1.25);
		\draw (0,0.75) rectangle (1,1.25);
		\draw (3.50,0.75) rectangle (4.5,1.25);
		\draw[white] (2.25,1) node {\scriptsize{c}};
		\draw (0.5,1) node {\scriptsize{d}};
		\draw (4.15,1) node {\scriptsize{d}};
	\end{tikzpicture}~~.
\]
However, $f$ does not define a colouring of $W$, because the binary relation $\leqslant_f$ induced by the order of $W$ is not a partial order. For the same reason, the following coloured diagram:
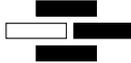
\[
	\begin{tikzpicture}[scale=0.4]
		\draw[fill=black] (0,0) rectangle (2,0.5);
		\draw[fill=black] (0,1.5) rectangle (2,2);
		\draw[fill=black] (1.25,0.75) rectangle (3.25,1.25);
		\draw (-1,0.75) rectangle (1,1.25);
	\end{tikzpicture} 
\]
	is not the diagram of a colouring. The diagram
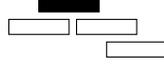
\[
	\begin{tikzpicture}[scale=0.4]
		\draw (2.25,0) rectangle (4.25,0.5);
		\draw[fill=black] (0,1.5) rectangle (2,2);
		\draw (1.25,0.75) rectangle (3.25,1.25);
		\draw (-1,0.75) rectangle (1,1.25);
	\end{tikzpicture} 
\]
is not the diagram of a colouring, because the white sub-wall is not connected. 
\end{nonexmp} 

\begin{lem}\label{lem::coloriage_quotient}
Let $S$ be a non-empty finite set, $W$ in $\mathcal{W}(S)$, a wall over $S$ and $\phi : W \twoheadrightarrow C$, a colouring of $W$ with $\mathrm{Succ}(C)\ne\varnothing$ (which implies $|C|>1$). For any pair $(c_1<c_2)\in \mathrm{Succ}(C)$ of successive colours, the composition
\[
	\begin{tikzcd}
	W \ar[r, two heads, "\phi"] \ar[rd, dotted, bend right=15,"\exists !\tilde{\phi}"']& C \ar[d, two heads, "\pi_{c_1}^{c_2}"] \\
	& C/_{c_1\sim c_2}
	\end{tikzcd}
\]
where $\pi_{c_1}^{c_2}$ identified the two colours $c_1$ and $c_2$, define a colouring  $\tilde{\phi}$ of $W$.
\end{lem}
\begin{proof}
 	By \Cref{prop::projection_ordre}:
\end{proof}
\begin{exmp}	
	We consider the wall $W$ with five bricks represented by   
	
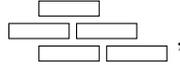
\[
	\begin{tikzpicture}[scale=0.4,baseline=1.4ex]
			\draw (0,0) rectangle (2,0.5);
			\draw (2.25,0) rectangle (4.25,0.5);
			\draw (0,1.5) rectangle (2,2);
			\draw (1.25,0.75) rectangle (3.25,1.25);
			\draw (-1,0.75) rectangle (1,1.25);
		\end{tikzpicture}~,
	\]
	and the colouring $\phi : W \twoheadrightarrow \{w,b,g\}$ (with $w$ for white, $g$ for gray and  $b$ for black), diagrammatically represented by:
	
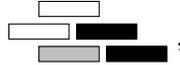
\[
		\begin{tikzpicture}[scale=0.4,baseline=1.4ex]
			\draw[fill=lightgray] (0,0) rectangle (2,0.5);
			\draw[fill=black] (2.25,0) rectangle (4.25,0.5);
			\draw[fill=black] (1.25,0.75) rectangle (3.25,1.25);
			\draw (-1,0.75) rectangle (1,1.25);
			\draw (0,1.5) rectangle (2,2);
		\end{tikzpicture}~,
	\]
	with $g\leqslant b \leqslant w$. As $g$ and $b$ are two successive colours, $\phi$ induces a colouring $\tilde{\phi} : W \twoheadrightarrow C/_{g\sim b}$ with two colours, which is diagrammatically represented by: 
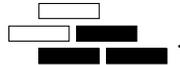
\[
		\begin{tikzpicture}[scale=0.4,baseline=1.4ex]
		\draw[fill=black] (0,0) rectangle (2,0.5);
		\draw[fill=black] (2.25,0) rectangle (4.25,0.5);
		\draw[fill=black] (1.25,0.75) rectangle (3.25,1.25);
		\draw (-1,0.75) rectangle (1,1.25);
		\draw (0,1.5) rectangle (2,2);
	\end{tikzpicture}~.
\]
The colouring $\phi$ induces another colouring $\widehat{\phi} : W \twoheadrightarrow C/_{b\sim w}$ represented by:
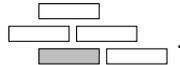
\[
	\begin{tikzpicture}[scale=0.4,baseline=1.4ex]
		\draw[fill=lightgray] (0,0) rectangle (2,0.5);
		\draw (2.25,0) rectangle (4.25,0.5);
		\draw (1.25,0.75) rectangle (3.25,1.25);
		\draw (-1,0.75) rectangle (1,1.25);
		\draw (0,1.5) rectangle (2,2);
	\end{tikzpicture}~.
\]
On the other hand, as $G$ and $W$ are not successive, the map $\bar{\phi} : W \twoheadrightarrow C/_{g\sim w}$ does not define a colouring of  $W$: the diagram 
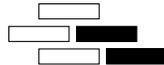
\[
	\begin{tikzpicture}[scale=0.4,baseline=1.4ex]
		\draw (0,0) rectangle (2,0.5);
		\draw[fill=black] (2.25,0) rectangle (4.25,0.5);
		\draw[fill=black] (1.25,0.75) rectangle (3.25,1.25);
		\draw (-1,0.75) rectangle (1,1.25);
		\draw (0,1.5) rectangle (2,2);
	\end{tikzpicture}
\]
is not a colouring diagram. 
\end{exmp}
\begin{lem}\label{lem::quotient_de_mur_par_coloriage}
	Let $S$ be a non-empty finite set, $W$ in $\mathcal{W}(S)$ a wall (resp. $W$ in $\Wall(S)$ a connected wall) over $S$ and $\phi : W \twoheadrightarrow C$ a colouring of $W$. We note $\sim_\phi$, the equivalence relation of $W$ induced  by  $\phi$, i.e. for $k$ and $l$, two elements of $W$, we have $ k\sim_\phi l $ if $\phi(k)=\phi(l)$. Then $W/_{\sim_\phi}$ is a wall (resp. $W/_{\sim_\phi}$ is a connected wall) over $S$.
\end{lem}
\begin{proof}
	By \Cref{prop::projection_ordre} and the definition of a colouring.
\end{proof} 

\subsection{The colouring complex}
\begin{defiprop}
Consider $(S,<_S)$, a finite \emph{totally} ordered set: for a wall  $(W,<_W)$ in $\mathcal{W}(S)$ over $S$, we can extend as follows the partial order $<_W$ to a total order $\prec_W$ on $W$, induced by that of $S$. For $W_a$ and $W_b$  in $W$, we have $W_a\prec_W W_b$, if:
\begin{itemize}
	\item $W_a\cap W_b\ne\varnothing$ and $W_a<_W W_b$ (because $W_a\cap W_b\ne\varnothing$ implies that $W_a$ and $W_b$ are comparable for $<_W$);
	\item $W_a\cap W_b=\varnothing$ and $\mathfrak{h}(W_a)<_\N \mathfrak{h}(W_b)$ with $\mathfrak{h}(W_\alpha)$ the height of the brick $W_\alpha$ in the wall  $W$ (cf. \Cref{def::poset::hauteur});
	\item $W_a\cap W_b=\varnothing$, $\mathfrak{h}(W_a)= \mathfrak{h}(W_b)$ and  $\mathrm{min}(W_a)<_S \mathrm{min}(W_b)$.
\end{itemize}
\end{defiprop}
Let $\phi : W \rightarrow C$ be a colouring of $W$, a wall over $S$. As the order $<_W$ induces a partial order $<_\phi$ on $C$, by definition of a colouring, the total order $\prec_W$ induces a total order on $C$ denoted by $\prec_\phi$:
\[
c_1\prec_{\phi} c_2 \mbox{ if } \exists k_1 \in\phi^{-1}(c_1),k_2 \in\phi^{-1}(c_2) \mbox{ such that } k_1\prec_W k_2;
\] 
\begin{lem}
	For a connected wall $W$ in $\Wall(S)$ and a colouring $\phi$ in $\Colo(W)$, the set of pairs of successive colours $\mathrm{Succ}(\phi)$ has a total order $\prec_\phi$ defined as follows:  for $c=(c_1,c_2)$ and $d=(d_1,d_2)$, two elements of $\mathrm{Succ}(\phi)$, we have $c\prec_\phi d$ if 
	\[
		\mathrm{min}_{\prec_W}\Big(\phi^{-1}(c_1)\cup\phi^{-1}(c_2) \Big)\prec_W \mathrm{min}_{\prec_W}\Big(\phi^{-1}(d_1)\cup\phi^{-1}(d_2)\Big).
	\]
\end{lem}
By this lemma, we index the projections $\pi_c^d$ by integers: if   $(c<d)\in\mathrm{Succ}(\phi)$ is the $i$-th element (for the total order $\prec_\phi$), we note $\partial_i:=\pi_c^d$. Furthermore, we observe that $ \partial_i\partial_j =\partial_{j-1}\partial_i$ for all $i<_\N j$, as in a semi-simplicial set. However, we will see (cf.\Cref{ex::colouring_complex}) that the set of colouring of a wall $W$ is \emph{not} a semi-simplicial set, but we can still associate a chain complex to the poset of colourings of a wall.

\begin{defiprop}
	Let $S$ be a finite totally ordered set. For a connected wall $W$ over $S$, the $\Z$-linearization of the graded set $\Colo_\bullet(W)$ gives a chain complex, called the \emph{colouring complex}, denoted by $C^{\Colo}_\bullet(W)$, where the differential is given by 
	\[
	(W,\phi) \overset{\partial^{\Colo}}{\longmapsto} \sum_{(c<d)\in\mathrm{Succ}(\phi)} (-1)^\Lambda (W,\pi_c^d\circ\phi)
	\]
	with 	
	\begin{align*}
		\Lambda:=& \#\big\{x \in \phi(W)~|~x \prec_\phi d \mbox{ and } x\ne c\big\} \\
		& +\#\big\{x \in \phi(W)~|~c \prec_\phi x \prec_\phi d \mbox{ and } (x<d)\in \mbox{Succ}(\phi)\big\}~.
	\end{align*}
\end{defiprop}
\begin{rem}
	The condition $c \prec_\phi x \prec_\phi d$ implies $x\ne c$ and $x\ne d$: the inequalities are strict.
\end{rem}
\begin{proof}
	We need to prove that $\partial^{\Colo}\circ \partial^{\Colo}=0$. We just need to understand what are the signs of terms $\pi_a^b\pi_c^d$ and $\pi_c^d\pi_a^b$ in $\partial^{\Colo}\circ \partial^{\Colo}$. We consider two pairs of successive colours $(a<b) \prec_\phi (c<d)$ for $\phi$, a colouring of $W$. 
	\begin{itemize}
		\item We start with two pairs of successive colours $(a<b) \prec_\phi (c<d)$ such that $b\ne d$. Then, in the composition $\partial^{\Colo}\circ \partial^{\Colo}$, we have the contribution $(-1)^{\Lambda_1}\pi_a^b\pi_c^d + (-1)^{\Lambda_2}\pi_c^d\pi_a^b$, with 
		\begin{align*}
		\Lambda_1=& \#\big\{x \in \phi(W)~|~x \prec_\phi d \mbox{ and } x\ne c\big\} \\
			& +\#\big\{x \in \phi(W)~|~c \prec_\phi x \prec_\phi d \mbox{ and } (x<d)\in \mbox{Succ}(\phi)\big\} \\
			& +\#\big\{x \in \pi_c^d\circ\phi(W)~|~x \prec_{\pi_c^d\circ\phi} b \mbox{ and } x\ne a\big\} \\
			& +\#\big\{x \in \pi_c^d\circ\phi(W)~|~a \prec_{\pi_c^d\circ\phi} x \prec_{\pi_c^d\circ\phi} b \mbox{ and } (x<b)\in \mbox{Succ}(\pi_c^d\circ\phi)\big\} \\
			= & \#\big\{x \in \phi(W)~|~x \prec_\phi d \mbox{ and } x\ne c\big\} \\
			& +\#\big\{x \in \phi(W)~|~c \prec_\phi x \prec_\phi d \mbox{ and } (x<d)\in \mbox{Succ}(\phi)\big\} \\
			& +\#\big\{x \in \phi(W)~|~x \prec_\phi b \mbox{ and } x\ne a\big\} \\
			& +\#\big\{x \in \phi(W)~|~a \prec_\phi x \prec_\phi b \mbox{ and } (x<b)\in \mbox{Succ}(\phi)\big\}
		\end{align*}
		 and 
		\begin{align*}
			\Lambda_2=& \#\big\{x \in \phi(W)~|~x \prec_\phi b \mbox{ and } x\ne a\big\} \\
			& +\#\big\{x \in \phi(W)~|~a \prec_\phi x \prec_\phi b \mbox{ and } (x<b)\in \mbox{Succ}(\phi)\big\} \\
			& + \#\big\{x \in \pi_a^b\circ\phi(W)~|~x \prec_{\pi_a^b\circ\phi} d \mbox{ and } x\ne c\big\} \\
			& +\#\big\{x \in \pi_a^b\circ\phi(W)~|~c \prec_{\pi_a^b\circ\phi} x \prec_{\pi_a^b\circ\phi} d \mbox{ and } (x<d)\in \mbox{Succ}(\pi_a^b\circ\phi)\big\} \\
			= & \#\big\{x \in \phi(W)~|~x \prec_\phi b \mbox{ and } x\ne a\big\} \\
			& +\#\big\{x \in \phi(W)~|~a \prec_\phi x \prec_\phi b \mbox{ and } (x<b)\in \mbox{Succ}(\phi)\big\} \\
			& + \#\big\{x \in \phi(W)~|~x \prec_{\phi} d \mbox{ and } x\ne c\big\} -1\\
			& +\#\big\{x \in \phi(W)~|~c \prec_{\phi} x \prec_{\phi} d \mbox{ and } (x<d)\in \mbox{Succ}(\phi)\big\}
		\end{align*}
		so the contribution $(-1)^{\Lambda_1}\pi_a^b\pi_c^d + (-1)^{\Lambda_2}\pi_c^d\pi_a^b$ is null.
		\item  We consider the case which we have $(a<c) \prec_\phi (b<c)$. The contribution $(-1)^{\Lambda_1}\pi_a^c\pi_b^c + (-1)^{\Lambda_2}\pi_b^c\pi_a^c$ have the signs given by:
		\begin{align*}
		\Lambda_1=& \#\big\{x \in \phi(W)~|~x \prec_\phi c \mbox{ and } x\ne b\big\} \\
			& +\#\big\{x \in \phi(W)~|~b \prec_\phi x \prec_\phi c \mbox{ and } (x<c)\in \mbox{Succ}(\phi)\big\} \\
			& +\#\big\{x \in \pi_b^c\circ\phi(W)~|~x \prec_{\pi_b^c\circ\phi} c \mbox{ and } x\ne a\big\} \\
			& +\#\big\{x \in \pi_b^c\circ\phi(W)~|~a \prec_{\pi_b^c\circ\phi} x \prec_{\pi_b^c\circ\phi} c \mbox{ and } (x<c)\in \mbox{Succ}(\pi_b^c\circ\phi)\big\} \\
		 = &\#\big\{x \in \phi(W)~|~x \prec_\phi c \mbox{ and } x\ne b\big\} \\
			& +\#\big\{x \in \phi(W)~|~b \prec_\phi x \prec_\phi c \mbox{ and } (x<c)\in \mbox{Succ}(\phi)\big\} \\
			& +\#\big\{x \in \phi(W)~|~x \prec_\phi c \mbox{ and } x\ne a\big\} -1\\
			& +\#\big\{x \in \phi(W)~|~a \prec_\phi x \prec_\phi c \mbox{ and } (x<c)\in \mbox{Succ}(\phi)\big\}-1
		\end{align*}
		 and 
		\begin{align*}
			\Lambda_2= &\#\big\{x \in \phi(W)~|~x \prec_\phi c \mbox{ and } x\ne b\big\} -1\\
			& +\#\big\{x \in \phi(W)~|~b \prec_\phi x \prec_\phi c \mbox{ and } (x<c)\in \mbox{Succ}(\phi)\big\} \\
			& +\#\big\{x \in \phi(W)~|~x \prec_\phi c \mbox{ and } x\ne a\big\} \\
			& +\#\big\{x \in \phi(W)~|~a \prec_\phi x \prec_\phi c \mbox{ and } (x<c)\in \mbox{Succ}(\phi)\big\}
		\end{align*}
		so the contribution of $(-1)^{\Lambda_1}\pi_a^c\pi_b^c + (-1)^{\Lambda_2}\pi_b^c\pi_a^c$ is null.
	\end{itemize}
	Then, we have $\partial^{\Colo}\circ \partial^{\Colo}=0$.
\end{proof}

\begin{lem}\label{rem::decompo_cpx_colo}
	 If $W\in\mathcal{W}(S)$ is a non connected wall over a totally ordered $S$, and $W^1\amalg\ldots\amalg W^l$ is the decomposition in connected component of $W$, we have  
	\[
		C_\bullet^{\Colo}(W) \cong \bigotimes_{i=1}^l C_\bullet^{\Colo}(W^i).
	\]
\end{lem}
\begin{proof}
	By \Cref{rem::decompo_colo} \ref{rem_iii::decompo_colo}, and the fact that the functor of linearization preserves coproducts.
\end{proof}

\begin{exmp}\label{ex::colouring_complex} Consider the colouring complex of the wall $W\in\Wall([\![1,4]\!])$ represented by 
$~
	 \begin{tikzpicture}[scale=0.3,baseline=0.6ex]
		\draw (0,0) rectangle (2,0.5);
		\draw (2.25,0) rectangle (4.25,0.5);
		\draw (1.25,0.75) rectangle (3.25,1.25);
		\draw (-1,0.75) rectangle (1,1.25);
	\end{tikzpicture}~~. 
$
\begin{itemize}
\item $\Colo_4(W)$  contains only the top-colouring of $W$:
\[
	\begin{tikzpicture}[scale=0.3,baseline=1ex]
		\draw[fill=black] (0,0) rectangle (2,0.5);
		\draw[fill=white] (2.25,0) rectangle (4.25,0.5);
		\draw[fill=gray] (1.25,0.75) rectangle (3.25,1.25);
		\draw[fill=lightgray] (-1,0.75) rectangle (1,1.25);
	\end{tikzpicture} ~.
\]
the poset $\mathrm{Succ}(top)$, where we denote by $\prec_{top}$ its total order, is given by:
\[
	\mathrm{Succ}(top)=\left\{
	\begin{tikzpicture}[scale=0.3,baseline=1ex]
		\draw[fill=black] (0,0) rectangle (2,0.5);
		\draw[fill=lightgray] (-1,0.75) rectangle (1,1.25);
	\end{tikzpicture} 
	\prec_{top}
	\begin{tikzpicture}[scale=0.3,baseline=1ex]
		\draw[fill=black] (0,0) rectangle (2,0.5);
		\draw[fill=gray] (1.25,0.75) rectangle (3.25,1.25);
	\end{tikzpicture} 
	\prec_{top}
	\begin{tikzpicture}[scale=0.3,baseline=1ex]
		\draw[fill=white] (2.25,0) rectangle (4.25,0.5);
		\draw[fill=gray] (1.25,0.75) rectangle (3.25,1.25);
	\end{tikzpicture} 
	\right\}
\]
where any diagram represent a pair of successive bricks. So, we have the following three arrows
$
 	\xymatrix{
 		\Colo_4(W) \ar[r] \ar@<-0.4pc>[r] \ar@<0.4pc>[r]^{\partial_i} & \Colo_3(W) 
 	}
$
represented by:
\[
	\begin{tikzpicture}[scale=0.3,baseline=1ex]
		\draw[fill=black] (0,0) rectangle (2,0.5);
		\draw[fill=white] (2.25,0) rectangle (4.25,0.5);
		\draw[fill=gray] (1.25,0.75) rectangle (3.25,1.25);
		\draw[fill=lightgray] (-1,0.75) rectangle (1,1.25);
	\end{tikzpicture}
	\ \overset{\partial_1}{\rightarrow}\
	\begin{tikzpicture}[scale=0.3,baseline=1ex]
		\draw[fill=black] (0,0) rectangle (2,0.5);
		\draw[fill=white] (2.25,0) rectangle (4.25,0.5);
		\draw[fill=gray] (1.25,0.75) rectangle (3.25,1.25);
		\draw[fill=black] (-1,0.75) rectangle (1,1.25);
	\end{tikzpicture} \ ; \
	\begin{tikzpicture}[scale=0.3,baseline=1ex]
		\draw[fill=black] (0,0) rectangle (2,0.5);
		\draw[fill=white] (2.25,0) rectangle (4.25,0.5);
		\draw[fill=gray] (1.25,0.75) rectangle (3.25,1.25);
		\draw[fill=lightgray] (-1,0.75) rectangle (1,1.25);
	\end{tikzpicture}
	 \ \overset{\partial_2}{\rightarrow} \
	\begin{tikzpicture}[scale=0.3,baseline=1ex]
		\draw[fill=black] (0,0) rectangle (2,0.5);
		\draw[fill=white] (2.25,0) rectangle (4.25,0.5);
		\draw[fill=black] (1.25,0.75) rectangle (3.25,1.25);
		\draw[fill=lightgray] (-1,0.75) rectangle (1,1.25);
	\end{tikzpicture} \ ; \
	\begin{tikzpicture}[scale=0.3,baseline=1ex]
		\draw[fill=black] (0,0) rectangle (2,0.5);
		\draw[fill=white] (2.25,0) rectangle (4.25,0.5);
		\draw[fill=gray] (1.25,0.75) rectangle (3.25,1.25);
		\draw[fill=lightgray] (-1,0.75) rectangle (1,1.25);
	\end{tikzpicture}
	 \ \overset{\partial_3}{\rightarrow}\
	\begin{tikzpicture}[scale=0.3,baseline=1ex]
		\draw[fill=black] (0,0) rectangle (2,0.5);
		\draw[fill=white] (2.25,0) rectangle (4.25,0.5);
		\draw[fill=white] (1.25,0.75) rectangle (3.25,1.25);
		\draw[fill=lightgray] (-1,0.75) rectangle (1,1.25);
	\end{tikzpicture} \ .
\]
\item $\Colo_3(W)$ is the set of the following three colourings
\[
	\begin{tikzpicture}[scale=0.3,baseline=1ex]
		\draw[fill=black] (0,0) rectangle (2,0.5);
		\draw[fill=white] (2.25,0) rectangle (4.25,0.5);
		\draw[fill=gray] (1.25,0.75) rectangle (3.25,1.25);
		\draw[fill=black] (-1,0.75) rectangle (1,1.25);
	\end{tikzpicture} ~~,~~\quad
	\begin{tikzpicture}[scale=0.3,baseline=1ex]
		\draw[fill=black] (0,0) rectangle (2,0.5);
		\draw[fill=white] (2.25,0) rectangle (4.25,0.5);
		\draw[fill=black] (1.25,0.75) rectangle (3.25,1.25);
		\draw[fill=lightgray] (-1,0.75) rectangle (1,1.25);
	\end{tikzpicture} \quad~~\mbox{ and }~~\quad
	\begin{tikzpicture}[scale=0.3,baseline=1ex]
		\draw[fill=black] (0,0) rectangle (2,0.5);
		\draw[fill=white] (2.25,0) rectangle (4.25,0.5);
		\draw[fill=white] (1.25,0.75) rectangle (3.25,1.25);
		\draw[fill=lightgray] (-1,0.75) rectangle (1,1.25);
	\end{tikzpicture} ~~. 
\]
For each colouring in $\Colo_3(W)$, we have the set of successive colours:
\begin{align*}
	\mathrm{Succ}\left(
	\begin{tikzpicture}[scale=0.3,baseline=1ex]
		\draw[fill=black] (0,0) rectangle (2,0.5);
		\draw[fill=white] (2.25,0) rectangle (4.25,0.5);
		\draw[fill=gray] (1.25,0.75) rectangle (3.25,1.25);
		\draw[fill=black] (-1,0.75) rectangle (1,1.25);
	\end{tikzpicture} 
	\right)= &~
	\left\{
	\begin{tikzpicture}[scale=0.3,baseline=1ex]
		\draw[fill=black] (0,0) rectangle (2,0.5);
		\draw[fill=gray] (1.25,0.75) rectangle (3.25,1.25);
		\draw[fill=black] (-1,0.75) rectangle (1,1.25);
	\end{tikzpicture} 
	\prec
	\begin{tikzpicture}[scale=0.3,baseline=1ex]
		\draw[fill=white] (2.25,0) rectangle (4.25,0.5);
		\draw[fill=gray] (1.25,0.75) rectangle (3.25,1.25);
	\end{tikzpicture} 
	\right\}~~; \\
	\mathrm{Succ}\left(
	\begin{tikzpicture}[scale=0.3,baseline=1ex]
		\draw[fill=black] (0,0) rectangle (2,0.5);
		\draw[fill=white] (2.25,0) rectangle (4.25,0.5);
		\draw[fill=black] (1.25,0.75) rectangle (3.25,1.25);
		\draw[fill=lightgray] (-1,0.75) rectangle (1,1.25);
	\end{tikzpicture}
	\right)= &~
	\left\{
	\begin{tikzpicture}[scale=0.3,baseline=1ex]
		\draw[fill=black] (0,0) rectangle (2,0.5);
		\draw[fill=black] (1.25,0.75) rectangle (3.25,1.25);
		\draw[fill=lightgray] (-1,0.75) rectangle (1,1.25);
	\end{tikzpicture} 
	\prec
	\begin{tikzpicture}[scale=0.3,baseline=1ex]
		\draw[fill=black] (0,0) rectangle (2,0.5);
		\draw[fill=white] (2.25,0) rectangle (4.25,0.5);
		\draw[fill=black] (1.25,0.75) rectangle (3.25,1.25);
	\end{tikzpicture} 
	\right\}~~; \\
	\mathrm{Succ}\left(
	\begin{tikzpicture}[scale=0.3,baseline=1ex]
		\draw[fill=black] (0,0) rectangle (2,0.5);
		\draw[fill=white] (2.25,0) rectangle (4.25,0.5);
		\draw[fill=white] (1.25,0.75) rectangle (3.25,1.25);
		\draw[fill=lightgray] (-1,0.75) rectangle (1,1.25);
	\end{tikzpicture} 
	\right)= &~
	\left\{
	\begin{tikzpicture}[scale=0.3,baseline=1ex]
		\draw[fill=black] (0,0) rectangle (2,0.5);
		\draw[fill=lightgray] (-1,0.75) rectangle (1,1.25);
	\end{tikzpicture} 
	\prec
	\begin{tikzpicture}[scale=0.3,baseline=1ex]
		\draw[fill=black] (0,0) rectangle (2,0.5);
		\draw[fill=white] (2.25,0) rectangle (4.25,0.5);
		\draw[fill=white] (1.25,0.75) rectangle (3.25,1.25);
	\end{tikzpicture}  
	\right\}~~;
\end{align*}
which are given to us the maps from $\Colo_3(W)$ to $\Colo_2(W)$. 
\item $\Colo_2(W)$ contains the following colourings:
\[
	\begin{tikzpicture}[scale=0.3,baseline=1ex]
		\draw[fill=black] (0,0) rectangle (2,0.5);
		\draw[fill=white] (2.25,0) rectangle (4.25,0.5);
		\draw[fill=black] (1.25,0.75) rectangle (3.25,1.25);
		\draw[fill=black] (-1,0.75) rectangle (1,1.25);
	\end{tikzpicture} ~~,~~\quad
	\begin{tikzpicture}[scale=0.3,baseline=1ex]
		\draw[fill=black] (0,0) rectangle (2,0.5);
		\draw[fill=white] (2.25,0) rectangle (4.25,0.5);
		\draw[fill=white] (1.25,0.75) rectangle (3.25,1.25);
		\draw[fill=black] (-1,0.75) rectangle (1,1.25);
	\end{tikzpicture} \quad~~\mbox{ and }~~\quad
	\begin{tikzpicture}[scale=0.3,baseline=1ex]
		\draw[fill=black] (0,0) rectangle (2,0.5);
		\draw[fill=black] (2.25,0) rectangle (4.25,0.5);
		\draw[fill=black] (1.25,0.75) rectangle (3.25,1.25);
		\draw[fill=lightgray] (-1,0.75) rectangle (1,1.25);
	\end{tikzpicture} ~~.
\]
As these colourings just have two colours, the sets of successive colours associated to them are reduced to only one element.
\end{itemize}
Finally, we have the complete description of the colouring complex of $W$:
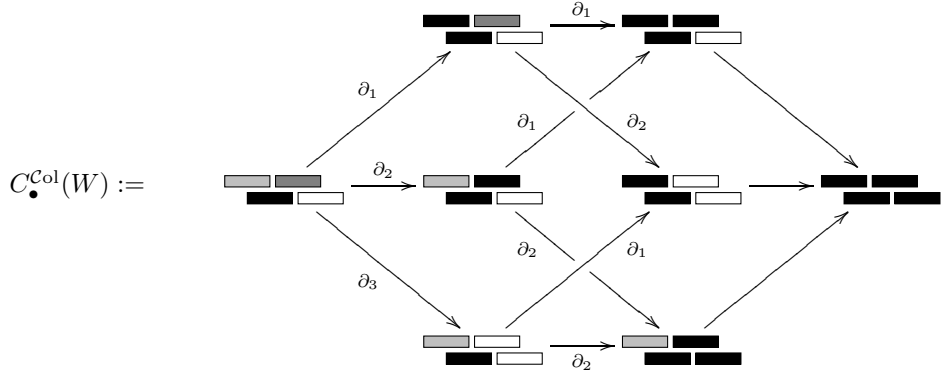
\[	
	\xymatrix@R=1.5cm{
	&
	&
	\begin{tikzpicture}[scale=0.3,baseline=1ex]
		\draw[fill=black] (0,0) rectangle (2,0.5); 
		\draw[fill=white] (2.25,0) rectangle (4.25,0.5); 
		\draw[fill=gray] (1.25,0.75) rectangle (3.25,1.25); 
		\draw[fill=black] (-1,0.75) rectangle (1,1.25); 
	\end{tikzpicture} 
	\ar[r]^{\partial_1} \ar[rd]^(0.7){\partial_2}
	& 
	\begin{tikzpicture}[scale=0.3,baseline=1ex]
		\draw[fill=black] (0,0) rectangle (2,0.5);
		\draw[fill=white] (2.25,0) rectangle (4.25,0.5);
		\draw[fill=black] (1.25,0.75) rectangle (3.25,1.25);
		\draw[fill=black] (-1,0.75) rectangle (1,1.25);
	\end{tikzpicture} 
	\ar[rd]
	& 	\\
	C^{\Colo}_\bullet(W):=	
	&
	\begin{tikzpicture}[scale=0.3,baseline=1ex]
		\draw[fill=black] (0,0) rectangle (2,0.5);
		\draw[fill=white] (2.25,0) rectangle (4.25,0.5);
		\draw[fill=gray] (1.25,0.75) rectangle (3.25,1.25);
		\draw[fill=lightgray] (-1,0.75) rectangle (1,1.25);
	\end{tikzpicture} 
	\ar[ru]^{\partial_1} \ar[r]^{\partial_2} \ar[rd]_{\partial_3}
	&
	\begin{tikzpicture}[scale=0.3,baseline=1ex]
		\draw[fill=black] (0,0) rectangle (2,0.5);
		\draw[fill=white] (2.25,0) rectangle (4.25,0.5);
		\draw[fill=black] (1.25,0.75) rectangle (3.25,1.25);
		\draw[fill=lightgray] (-1,0.75) rectangle (1,1.25);
	\end{tikzpicture}  
	\ar[ru]|{\hole}^(0.3){\partial_1} \ar[rd]|{\hole}_(0.3){\partial_2}
	&
	\begin{tikzpicture}[scale=0.3,baseline=1ex]
		\draw[fill=black] (0,0) rectangle (2,0.5);
		\draw[fill=white] (2.25,0) rectangle (4.25,0.5);
		\draw[fill=white] (1.25,0.75) rectangle (3.25,1.25);
		\draw[fill=black] (-1,0.75) rectangle (1,1.25);
	\end{tikzpicture} 
	\ar[r]
	 & 
	 \begin{tikzpicture}[scale=0.3,baseline=1ex]
		\draw[fill=black] (0,0) rectangle (2,0.5);
		\draw[fill=black] (2.25,0) rectangle (4.25,0.5);
		\draw[fill=black] (1.25,0.75) rectangle (3.25,1.25);
		\draw[fill=black] (-1,0.75) rectangle (1,1.25);
	\end{tikzpicture} 
	\\
	&
	&
	\begin{tikzpicture}[scale=0.3,baseline=1ex]
		\draw[fill=black] (0,0) rectangle (2,0.5); 
		\draw[fill=white] (2.25,0) rectangle (4.25,0.5); 
		\draw[fill=white] (1.25,0.75) rectangle (3.25,1.25); 
		\draw[fill=lightgray] (-1,0.75) rectangle (1,1.25); 
	\end{tikzpicture} 
	\ar[r]_{\partial_2} \ar[ru]_(0.7){\partial_1}
	& 
	\begin{tikzpicture}[scale=0.3,baseline=1ex]
		\draw[fill=black] (0,0) rectangle (2,0.5);
		\draw[fill=black] (2.25,0) rectangle (4.25,0.5);
		\draw[fill=black] (1.25,0.75) rectangle (3.25,1.25);
		\draw[fill=lightgray] (-1,0.75) rectangle (1,1.25);
	\end{tikzpicture} 
	\ar[ru]
	& 
	}
\]
where the differential is given by the sum of $\partial_i$ with the sign rule defined above. We have the following second example:
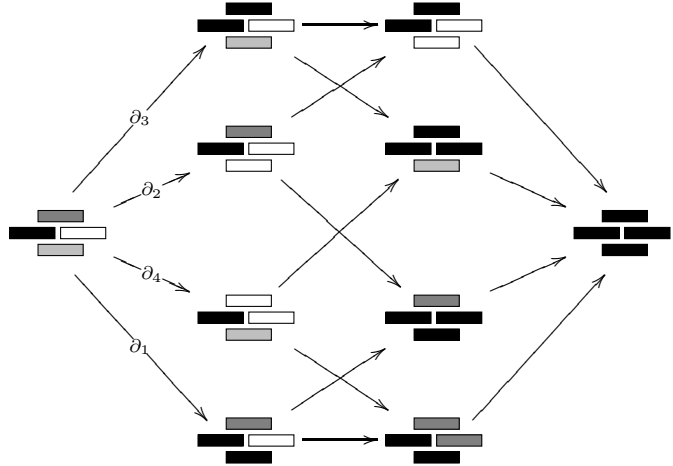
\[
\xymatrix@R=0.3cm@C=1cm{
	&
	\begin{tikzpicture}[scale=0.3,baseline=0ex]
	\draw[fill=black] (0,0) rectangle (2,0.5); 
	\draw[fill=white] (2.25,0) rectangle (4.25,0.5); 
	\draw[fill=black] (1.25,0.75) rectangle (3.25,1.25); 
	\draw[fill=lightgray] (1.25,-.75) rectangle (3.25,-.25); 
	\end{tikzpicture} 
	\ar[r] \ar[rdd]
	& 
	\begin{tikzpicture}[scale=0.3,baseline=0ex]
	\draw[fill=black] (0,0) rectangle (2,0.5); 
	\draw[fill=white] (2.25,0) rectangle (4.25,0.5); 
	\draw[fill=black] (1.25,0.75) rectangle (3.25,1.25); 
	\draw[fill=white] (1.25,-.75) rectangle (3.25,-.25); 
	\end{tikzpicture} 
	\ar@<.5pc>[rddd]
	& 	\\
	&&&  \\
	&
	\begin{tikzpicture}[scale=0.3,baseline=0ex]
	\draw[fill=black] (0,0) rectangle (2,0.5); 
	\draw[fill=white] (2.25,0) rectangle (4.25,0.5); 
	\draw[fill=gray] (1.25,0.75) rectangle (3.25,1.25); 
	\draw[fill=white] (1.25,-.75) rectangle (3.25,-.25); 
	\end{tikzpicture} 
	\ar[ruu]  \ar[rdd]
	& 
	\begin{tikzpicture}[scale=0.3,baseline=0ex]
	\draw[fill=black] (0,0) rectangle (2,0.5); 
	\draw[fill=black] (2.25,0) rectangle (4.25,0.5); 
	\draw[fill=black] (1.25,0.75) rectangle (3.25,1.25); 
	\draw[fill=lightgray] (1.25,-.75) rectangle (3.25,-.25); 
	\end{tikzpicture} 
	\ar[rd]
	& 	\\
	\begin{tikzpicture}[scale=0.3,baseline=0ex]
	\draw[fill=black] (0,0) rectangle (2,0.5); 
	\draw[fill=white] (2.25,0) rectangle (4.25,0.5); 
	\draw[fill=gray] (1.25,0.75) rectangle (3.25,1.25); 
	\draw[fill=lightgray] (1.25,-.75) rectangle (3.25,-.25); 
	\end{tikzpicture} 
	\ar[ru]|{\partial_2} \ar@<0.5pc>[ruuu]|{\partial_3} \ar[rd]|{\partial_4} \ar@<-.5pc>[rddd]|{\partial_1}
	&&&
	\begin{tikzpicture}[scale=0.3,baseline=0ex]
	\draw[fill=black] (0,0) rectangle (2,0.5); 
	\draw[fill=black] (2.25,0) rectangle (4.25,0.5); 
	\draw[fill=black] (1.25,0.75) rectangle (3.25,1.25); 
	\draw[fill=black] (1.25,-.75) rectangle (3.25,-.25); 
	\end{tikzpicture} 
	\\
	&
	\begin{tikzpicture}[scale=0.3,baseline=0ex]
	\draw[fill=black] (0,0) rectangle (2,0.5); 
	\draw[fill=white] (2.25,0) rectangle (4.25,0.5); 
	\draw[fill=white] (1.25,0.75) rectangle (3.25,1.25); 
	\draw[fill=lightgray] (1.25,-.75) rectangle (3.25,-.25); 
	\end{tikzpicture}  
	\ar[ruu] \ar[rdd]
	&
	\begin{tikzpicture}[scale=0.3,baseline=0ex]
	\draw[fill=black] (0,0) rectangle (2,0.5); 
	\draw[fill=black] (2.25,0) rectangle (4.25,0.5); 
	\draw[fill=gray] (1.25,0.75) rectangle (3.25,1.25); 
	\draw[fill=black] (1.25,-.75) rectangle (3.25,-.25); 
	\end{tikzpicture} 
	\ar[ru]
	& 
	~~\\
	&&&  \\
	&
	\begin{tikzpicture}[scale=0.3,baseline=0ex]
	\draw[fill=black] (0,0) rectangle (2,0.5); 
	\draw[fill=white] (2.25,0) rectangle (4.25,0.5); 
	\draw[fill=gray] (1.25,0.75) rectangle (3.25,1.25); 
	\draw[fill=black] (1.25,-.75) rectangle (3.25,-.25); 
	\end{tikzpicture} 
	\ar[r] \ar[ruu]
	& 
	\begin{tikzpicture}[scale=0.3,baseline=0ex]
	\draw[fill=black] (0,0) rectangle (2,0.5); 
	\draw[fill=gray] (2.25,0) rectangle (4.25,0.5); 
	\draw[fill=gray] (1.25,0.75) rectangle (3.25,1.25); 
	\draw[fill=black] (1.25,-.75) rectangle (3.25,-.25); 
	\end{tikzpicture} 
	\ar@<-.5pc>[ruuu]
	& 
} \ .
\]

\end{exmp}

\begin{thm}\label{prop::cpx_colo_acyclique}
	Let $S$ be a finite totally ordered set and $W$ a wall over $S$. If the set $\mathrm{Succ}(W)$ is not empty, then the colouring complex $C_\bullet^{\Colo}(W)$ is acyclic.
\end{thm}
\begin{proof}
	Let $S$ be a totally ordered finite set and $W$, a wall over $S$ with $\mathrm{Succ}(W)\ne\varnothing$. We prove the proposition by induction on the number of bricks in $W$. If $W$ has only one brick, then the set   $\mathrm{Succ}(W)$ is empty. If $\#W =2$ then the complex $C_\bullet^{\Colo}(W)$ is as follows:
	\[
		C_\bullet^{\Colo}(W)=~~
		0 \longrightarrow
		\begin{tikzpicture}[scale=0.3,baseline=13]
			\draw (0,1.25) rectangle (2,1.75);
			\draw[fill=black] (-1,2) rectangle (1,2.5);
		\end{tikzpicture}
		\overset{\partial}{\longrightarrow}
		\begin{tikzpicture}[scale=0.3,baseline=13]
			\draw (0,1.25) rectangle (2,1.75);
			\draw (-1,2) rectangle (1,2.5);
		\end{tikzpicture}
		\longrightarrow 0,
	\]
which is acyclic. We suppose, by induction, that, for all wall $W$ over $S$ such that $2\leqslant \#W<n$ and $\mathrm{Succ}(W)\ne\varnothing$, the chain complex $C_\bullet^{\Colo}(W)$ is acyclic. If $W$ is a non-connected wall with $n$ bricks, by \Cref{rem::decompo_cpx_colo} and the induction hypothesis, the chain complex $C_\bullet^{\Colo}(W)$ is acyclic. Now, let $W$ be a connected wall with $n$ bricks. 

We start be treating the exceptional case where $W$ has the following shape:
\begin{equation}
	\begin{tikzpicture}[scale=0.3,baseline=13]
		\draw (0,1.25) rectangle (2,2);
		\draw (2.25,1.25) rectangle (4.25,2);
		\draw[dotted] (4.5,1.5) -- (6.5,1.5);
		\draw (6.75,1.25) rectangle (8.75,2);
		\draw (0,0.5) rectangle (8.75,1);
		\draw (4.32,0.75) node {$\scriptscriptstyle{k}$};
		\draw (1,1.5) node {$\scriptscriptstyle{l_1}$};
		\draw (3.25,1.5) node {$\scriptscriptstyle{l_2}$};
		\draw (7.75,1.5) node {$\scriptscriptstyle{l_{n-1}}$};
	\end{tikzpicture} ~.
	\label{fig::mur_exceptionnel}
\end{equation}
We consider the sub-complex $C_\bullet^{\Colo,k<l_1}(W)$ which is isomorphic to  $C_\bullet^{\Colo}(W/_{k\sim l_1})$ which is acyclic by induction. So we consider the following short exact sequence:
\[
	\xymatrix{
		0 \ar[r] & C_\bullet^{\Colo,k<l_1}(W)\ar[r] & C_\bullet^{\Colo}(W) \ar[r] & C_\bullet^{\Colo}(W) \big/ C_\bullet^{\Colo,k<l_1}(W) \ar[r] & 0~;
	}
\]
The chain complex$C_\bullet^{\Colo}(W) \big/ C_\bullet^{\Colo,k<l_1}(W)$ is isomorphic to $C_\bullet^{\Colo}(W\backslash \{l_1\})$, so the term on the right hand side of the exact sequence is also acyclic, so the complex $C_\bullet^{\Colo}(W)$ too.
	
Now, we suppose that $W$ does not have the exceptional shape \Cref{fig::mur_exceptionnel}. We choose $(k<l)\in\mathrm{Succ}(W)$ with $\mathfrak{h}(k)=1$ and we consider the subcomplex $C_\bullet^{\Colo, (k<l)}(W) \subset C_\bullet^{\Colo}(W)$ of colourings $\phi$ of $W$ such that $\phi(k)=\phi(l)$. We consider $W/_{k\sim l}$, the poset with $\big(W\backslash \{k,l\}\big)\cup \{k \cup l\}$, its  underlying set with the partial order induced by the $W$ one: for $j\in W$ such that $j>k$ or $j>l$ (resp. $j<k$ or $j<l$) then we have $j>(k\cup l)$  (resp. $j<(k\cup l)$). By the definition of the differential of the colouring complex, we have the following isomorphism of chain complexes:
\[
	C_\bullet^{\Colo, k<l}(W) \cong C_\bullet^{\Colo}\big(W/_{k\sim l}\big).
\]
As $|W/_{k\sim l}|=|W|-1$, then $C_\bullet^{\Colo, k<l}(W)$ is acyclic by induction.

We denote by $\mathrm{Succ}_W(k)$, the set $\big\{l\in W~|~(k<l)\in \mathrm{Succ}(W)\big\} $ and we consider the chain complex 
\[
	\sum_{l\in\mathrm{Succ}_W(k)} C_\bullet^{\Colo, k<l}(W).
\]
By  \Cref{prop::mayer_vietoris} below, this complex is acyclic if, for all non-empty subsets $\widetilde{\mathrm{Succ}}(k)\subset\mathrm{Succ}_W(k)$, the complex 
\[
	\bigcap_{l\in \widetilde{\mathrm{Succ}}(k)} C_\bullet^{\Colo, k<l}(W)
\]
is acyclic. Let $\widetilde{\mathrm{Succ}}(k)$ be a non-empty subset of $\mathrm{Succ}_W(k)$: we have the isomorphism of chain complexes 
\[
	\bigcap_{l\in\widetilde{\mathrm{Succ}}(k)} C_\bullet^{\Colo, k<l}(W) \cong C_\bullet^{\Colo, k<l}(W/_{\widetilde{\mathrm{Succ}}(k)}).
\]
There are two cases: if $|W/_{\widetilde{\mathrm{Succ}}(k)}|=1$, which means that $\widetilde{\mathrm{Succ}}(k)= \mathrm{Succ}_W(k)$ and so $W$ has the exceptional shape as in  \Cref{fig::mur_exceptionnel}, which is excluded by the hypothesis. Otherwise $|W/_{\widetilde{\mathrm{Succ}}(k)}|>1$: in this case, $C_\bullet^{\Colo, k<l}(W/_{\widetilde{\mathrm{Succ}}(k)})$ is acyclic by induction.

We consider the complex 
\[
	C_\bullet^{\Colo}(W)\Big/ \Big(\sum_{l\in\mathrm{Succ}_W(k)} C_\bullet^{\Colo, k<l}(W) \Big).
\]
which is isomorphic to $C_\bullet^{\Colo}\big(W\backslash\{k\}\big)$. Note that this complex is generally a $\otimes$-product of complexes because $W\backslash\{k\}$ is not necessarily connected. As we have the short exact sequence 
\[
\begin{tikzcd}[column sep=2ex]
	0 \ar[r]&	
	\underset{l\in\mathrm{Succ}_W(k)}{\sum} C_\bullet^{\Colo, k<l}(W)
	\ar[r]&
	C_\bullet^{\Colo}(W)
	\ar[r]&
	\dfrac{C_\bullet^{\Colo}(W)}{ \underset{l\in\mathrm{Succ}_W(k)}{\sum} C_\bullet^{\Colo, k<l}(W) }
	\ar[r]& 0~~,
\end{tikzcd}
\]
with the left and the right hand sides acyclic, the complex $C_\bullet^{\Colo}(W)$ is too.
\end{proof}

\begin{lem}[Algebraic Mayer-Vietoris]\label{prop::mayer_vietoris}
	\begin{enumerate}
	\item Let $A$ and $B$ be acyclic chain complexes. If the complex $A\cap B$ is acyclic, then the complex $A+B$ is too.
	\item More generally, let $m$ be an integer in $\N^*$ and $\{A_i\}_{i\in [\![1,m]\!]}$ a sequence of $m$ acyclic complexes. If, for all subsets $J\subset [\![1,m]\!]$, the complex $\bigcap_{j\in J} A_j$ is acyclic, then the complex $\sum_{j=1}^m A_j$ is too.
	\end{enumerate}
\end{lem}
\begin{proof}
	\begin{enumerate}
		\item We have the square
		\[
			 \xymatrix{
			 	A\cap B \ar[r]  \ar[d] \ar@{}[rd]|(0.75)\pushout & B \ar[d] \\
			 	A \ar[r] & A+B
			 }
		\]
		which induces the short exact sequence 
		\[
			\xymatrix@C=0.7cm{
				0 \ar[r] & A\cap B \ar[r] & A\oplus B \ar[r] & A+B \ar[r] &0.
			}
		\] 
		We conclude the proof by the associated long exact sequence in homology and the additivity of the functor $\mathrm{H}_\bullet(-)$.
		\item We prove the result by induction on $m$. For $m=2$, we have (1). Let $m$ be an integer and suppose by induction that for all family  $A_1, \ldots, A_m$ of $m$ acyclic complexes, the complex $\bigcap_{j\in J} A_j$ is acyclic, then the complex $\sum_{j=1}^m A_j$ is acyclic. Let $A_0, \ldots, A_m$ be a family of $m+1$ acyclic complexes such that, for all $J\subset [\![0,m]\!]$, the complex $\bigcap_{j\in J} A_j$ is acyclic: so the complex $\sum_{j=1}^m A_j$ is acyclic, by induction. We have the following commutative square:
		\[
			 \xymatrix{
			 	\sum_{j=1}^m A_j\cap A_0 \ar[r]  \ar[d] \ar@{}[rd]|(0.75)\pushout & A_0\ar[d] \\
			 	\sum_{j=1}^m A_j\ar[r] & \sum_{j=0}^m A_j
			 }
		\]		
		with $A_0$, $\sum_{j=1}^m A_j$ and $\sum_{j=1}^m A_j\cap A_0$ acyclic. We conclude by (1).
	\end{enumerate}
\end{proof}

\nocite{Ler17,Val03,Val09,Wei94}
\bibliographystyle{alpha}
\bibliography{biblio_these}
\end{document}